\documentclass[11pt, preprint,natbib]{imsart}
\usepackage[plainpages=false]{hyperref}

\usepackage{xcolor}
\definecolor{ultramarine}{RGB}{0,32,96}
\definecolor{hookersgreen}{rgb}{0.0, 0.44, 0.0}
\hypersetup{linktocpage,
            colorlinks,
            linkcolor=hookersgreen,
            linktoc=blue,
            citecolor=ultramarine,
            urlcolor=black,
            filecolor=black
            }

\usepackage{nameref}

\startlocaldefs
\endlocaldefs

\usepackage[utf8]{inputenc}
\usepackage{amsmath}
\usepackage{amssymb}
\usepackage{nicefrac}
\usepackage{bbm} 
\usepackage{stackengine}
\usepackage{amsthm}
\usepackage{algorithm}
\usepackage{algorithmic}
\usepackage{dsfont}
\usepackage{graphicx}
\usepackage{subcaption}

\newcommand{\diff}[1]{\frac{d}{d#1}}
\newcommand\numgeq[1]%
  {\stackrel{\scriptscriptstyle(\mkern-1.5mu#1\mkern-1.5mu)}{\geq}}
  \newcommand\numleq[1]%
  {\stackrel{\scriptscriptstyle(\mkern-1.5mu#1\mkern-1.5mu)}{\leq}}
\newcommand{\defined}{\mathrel{\mathop:}=}
\newcommand{\definedas}{=\mathrel{\mathop:}}
\newcommand{\vol}{\begin{normalfont}\text{vol}\end{normalfont}}
\newcommand{\kom}[1]{}
\newcommand{\cupcap}{\pm}
\DeclareMathOperator{\esssup}{\begin{normalfont}\text{ess sup}\end{normalfont}}
\DeclareMathOperator{\reach}{\begin{normalfont}\text{reach}\end{normalfont}}

\newtheorem{thm}{Theorem}
\newtheorem{rmk}{Remark}
\newtheorem{lem}{Lemma}
\newtheorem{cor}{Corollary}
\newtheorem{defn}{Definition}
\newtheorem{prop}{Proposition}

\begin{document}

\begin{frontmatter}

\title{Adaptive Manifold Clustering\thanksref{t1}}
\runtitle{Adaptive Manifold Clustering}

 \author{\fnms{Franz} \snm{Besold}\corref{}}
 \thankstext{t1}{Financial support by German Ministry for Education via the Berlin Center for
Machine Learning (01IS18037I) is gratefully acknowledged.} 
\address{Weierstrass Institute, Mohrenstr. 39, 10117 Berlin, Germany\\ \href{mailto:franz.besold@wias-berlin.de}{\textcolor{ultramarine}{franz.besold@wias-berlin.de}}}
\and
\author{\fnms{Vladimir} \snm{Spokoiny}\thanksref{t2}}
 \thankstext{t2}{The research was supported by the Russian Science Foundation grant No. 18-11-00132.} 
\address{Weierstrass Institute and HU Berlin,
HSE University, IITP RAS,\\
Mohrenstr. 39, 10117 Berlin, Germany \\ \href{mailto:vladimir.spokoiny@wias-berlin.de}{\textcolor{ultramarine}{vladimir.spokoiny@wias-berlin.de}}}

\runauthor{F. Besold and V. Spokoiny}

\begin{abstract}
Clustering methods seek to partition data such that elements are more similar to elements in the same cluster than to elements in different clusters. 
The main challenge in this task is the lack of a unified definition of a cluster, especially for high-dimensional data. 
Different methods and approaches have been proposed to address this problem. 
This paper continues the study originated by 
\cite{AWC} where a novel approach to adaptive nonparametric clustering called \emph{Adaptive Weights Clustering} (AWC) 
was offered. 
The method allows analyzing 
high-dimensional data with an unknown number of unbalanced clusters of arbitrary shape under very weak modeling assumptions. 
The procedure demonstrates a state-of-the-art performance and is very efficient even for large data dimension $D$.
However, the theoretical study in \cite{AWC} is very limited and did not really address the question of efficiency.
This paper makes a significant step in 
understanding the promising performance of the AWC procedure, particularly in high dimension.
The approach is based on combining the ideas of adaptive clustering and manifold learning.
The manifold hypothesis means that high-dimensional data
can be well approximated by a $d$-dimensional manifold for small $d$
helping to overcome the \emph{curse of dimensionality} problem and to get sharp bounds on the cluster separation 
which only depend on the intrinsic dimension $d$. 
We also address the problem of parameter tuning. 
Our general theoretical results are illustrated by some numerical experiments.
\end{abstract}



\end{frontmatter}

\newpage
\tableofcontents

\section{Introduction}
\subsection{Manifold Clustering}
The task of clustering is often informally described as partitioning a set of objects such that objects in the same group are more similar to each other than to those in other groups.
The lack of a unified definition has led to a range of algorithms with different objectives. 
One of the oldest and best-known procedures are centroid-based methods such as k-means \cite{k-means}. 
Other well-known approaches are density-based methods, like DBSCAN \cite{DBSCAN} or spectral methods \cite{spectral_clustering}. 
For a comprehensive survey of clustering methods, we refer to \cite{survey_clustering}. 
A more general task is to obtain a hierarchical collection of clusters, the so-called \emph{density cluster tree} \cite{hartigan_cluster_tree}. This problem has been studied thoroughly, see e.g. \cite{separation_condition_clusters_for_tree}, \cite{Kpotufev2011_cluster_tree}, \cite{eldridge15_cluster_tree} and \cite{belakrishnan_cluster_tree_on_manifold} for more recent work. Allthough this approach avoids the choice of a scale parameter, it utilizes a specific definition of clusters beeing connected components of superlevel sets of the underlying density.
In this paper, we study a nonparametric clustering algorithm originated from \cite{AWC} and called \emph{Adaptive Weights Clustering (AWC)}. 
It is \emph{adaptive} as it does not require the user to specify the number of clusters, and it is able to recover clusters of different size, level of density and shape, including non-convex clusters.
The cluster structure of the data is represented by an adjacency matrix containing binary entries, so-called \emph{weights}, hence the name. The adjacency matrix is not guaranteed to correspond to a partition of the data, but rather will give information about local clusters for each data point. Informally speaking, the objective of the algorithm is to find maximal subsets of the data without any significant gap, that is a region within the cluster adjoining two areas in opposite direction of relatively larger density. This novel objective is in fact the reason for the high adaptivity of AWC to clusters with very different structural properties. 

This paper focuses on a theoretical study of the algorithm, as \cite{AWC} already provides a comprehensive comparative numerical study. In particular, we want to address the challenges that arise from high-dimensional data that does not concentrate on lower-dimensional linear subspaces and where the PCA analysis does not yield a significant spectral gap. We are therefore interested in the case of high-dimensional data lying close to a lower-dimensional submanifold \( \mathcal M \). This setup has already been studied for other clustering algorithms, e.g. in \cite{belakrishnan_cluster_tree_on_manifold} and \cite{jiang17a_DBSCAN_on_manifold}. Moreover, it appears in the context of homology inference \cite{different_noise_conditions_on_manifold}. It has been shown that this is a realistic model for various data, e.g. for images which are represented in a patch space \cite{mfd_images, ldmm} and a wide range of algorithms have been proposed to deal with the problem of non-linear dimension reduction \cite{dimension_reduction_review}, e.g. multidimensional scaling (MDS), kernel PCA, Isomap, Laplacian eigenmaps, self-organizing maps (SOM), locally-linear embeddings and autoencoders \cite{autoencoder}. In this work, we will not rely on any of these techniques, however, we recommend using a manifold denoising algorithm in practice such as \cite{mfd_denoising4} as an additional preprocessing step in order to reduce the magnitude of the noise.


\subsection{Submanifolds with positive reach}

As regularity condition for the manifold we assume a positive \emph{reach}, see Definition \ref{defn_reach}. 

\begin{defn}\label{defn_reach}
For \( \epsilon > 0 \) and a set \( S\subset \mathbb R^D \), let us denote the \( \epsilon \)-offset of \( S \) by 
\begin{align*}
S^\epsilon &\hspace{3pt}=\{y\in \mathbb R^D: \exists x\in S \text{ with }\|x-y\| \leq \epsilon\}
\end{align*}
and define the \emph{reach} of S to be
\[\hspace{3pt}\reach (S) := \sup\{r \geq 0 : \forall y \in S^r \text{ there exists a unique } x \in S \text{ nearest to }y\}.\]
\end{defn}

Originally introduced by \cite{Federer}, a positive reach has proven to be a widely used minimal condition in geometric and topological inference, c.f. \cite{chazal_book}. This includes in particular the topics of manifold estimation \cite{manifold_estimation_wasserman}, \cite{manifold_estimation_aamari} as well as homology inference \cite{homology_niyogi_smale_weinberger}, \cite{different_noise_conditions_on_manifold}. The latter can in fact be seen as a generalization of the clustering problem.

If a set has a positive reach \( \frac{1}{\kappa} \), it is also \( \frac{1}{\kappa} \)-convex and one can freely roll a ball of radius \( r < \frac{1}{\kappa}  \) around it \cite{rolling_ball}. The reach provides information about the local and the global structure of the manifold at the same time \cite{reach_estimation}: Any unit speed geodesic of a compact smooth submanifold \( \mathcal M \) without boundary with \( \text{reach}(\mathcal M) \geq \frac{1}{\kappa}> 0 \) has a curvature bounded by \( \kappa \) and also any so-called bottleneck, i.e. a point on the manifold that has two distinct projections onto the manifold in exactly opposite directions, has a distance of at least \( \frac{1}{\kappa} \) to \( \mathcal M \). More precisely, it can be shown that the reach is either attained by the curvature of a unit speed geodesic or is equal to the distance of a bottleneck to the manifold. See Figure \ref{fig_reach} for a visualization. Moreover, \( \mathcal M \) has a local Lipschitz continuous parametrization in terms of the tangent plane, see Lemma \ref{lem_local_lipschitz_param}. We exploit this property, using that any \( L \) -Lipschitz function changes the \( d \) -dimensional Lebesgue volume at most by a factor \( L^d \), see Lemma \ref{lem_volume_lipschitz}. For a survey on sets with positive reach see \cite{survey_reach}.

\begin{figure}[t]
 \begin{minipage}[c]{0.49\textwidth}
 \centering
\includegraphics[scale = 0.14025]{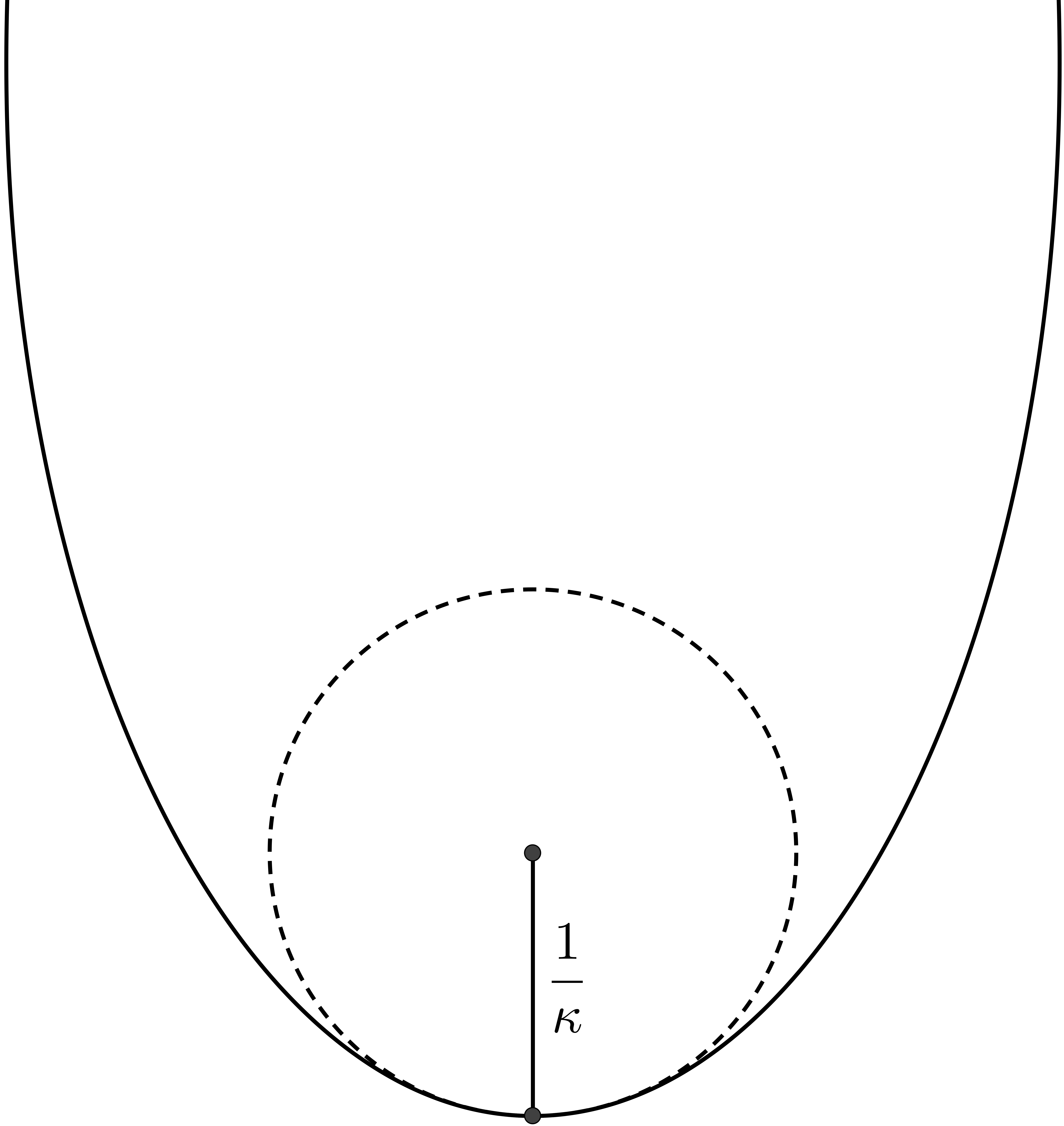}
 \end{minipage}
  \begin{minipage}[c]{0.49\textwidth}
  \centering
\includegraphics[scale = 0.15]{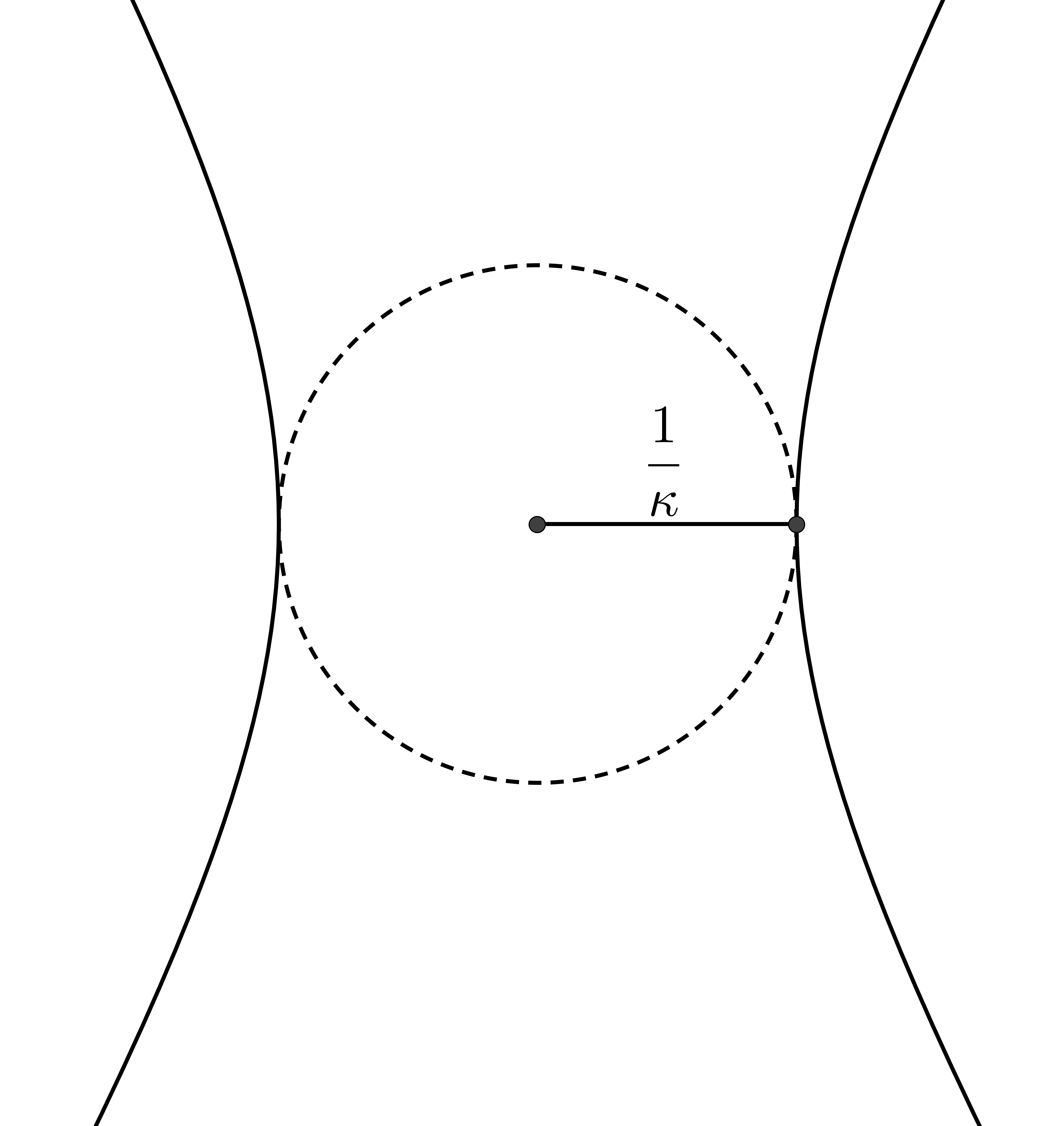}
 \end{minipage}
 \caption{The reach of a manifold can be either attained by the curvature radius of a geodesic (left) or the distance to a bottleneck (right)}
 \label{fig_reach}
\end{figure}

\subsection{AWC revisited}
The key ingredient of the AWC procedure is a so-called \emph{test of no gap}, which is based on a likelihood-ratio test for local homogeneity from \cite{Polzehl}. 
Given a sequence of radii \( 0<h_0 < \dots < h_K \) in addition to our data \( X_1,\dots, X_n\in\mathbb R^D \) and using the test of no gap, the algorithm successively screens subsets of increasing diameters. 
Using information from previous steps, AWC defines at each step \( k \) around each point \( X_i \) a so-called \emph{local cluster} \( \mathcal C_i^{(k)} \) that is supposed to be a maximal subset of the data in a vicinity of the given radius \( h_k \) satisfying the no gap objective.

In the following, let us explain the main idea of the algorithm more formally. An exact description via pseudocode is given in Algorithm \ref{algorithm}. By \( \|\cdot \| \) we denote the euclidean norm, \( \lambda \) denotes the \( D \)-dimensional Lebesgue measure and \( B(\cdot, \cdot) \) is the usual notation for a closed euclidean Ball in \( \mathbb R^D \) with given center and radius. 
Suppose our data \( X_1, \dots, X_n \in \mathbb  R^D \) is sampled independently from a common probability distribution \( \mathbb P \). Using regular conditional distributions, let us treat \( X_i \) and \( X_j \) as deterministic for some \( i\neq j \). From a given sequence of radii \( h_0 < h_1 < \dots < h_K \) s.t. \( \frac{h_{l+1}}{h_l} < 2 \) we choose \( h_k \) such that \( \|X_i-X_j\| < h_k \) and define the so-called \emph{gap coefficient} 
\[ \theta_{ij}^{(k)} = \frac{\mathbb P\left(B(X_i, h_{k-1})\cap B(X_j, h_{k-1})\right)}{\mathbb P\left(B(X_i, h_{k-1})\cup B(X_j, h_{k-1})\right)} . \]
In case of our distribution being uniform on a neighborhood of \( B(X_i, h_k)\cup B(X_j, h_k) \), or more generally, having a linear density, the gap coefficient coincides with the so-called \emph{volume coefficient}
\[ q_{ij}^{(k)} = \frac{\lambda \left(B(X_i, h_{k-1})\cap B(X_j, h_{k-1})\right)}{\lambda \left(B(X_i, h_{k-1})\cup B(X_j, h_{k-1})\right)} . \]
\begin{figure}[t]
 \begin{minipage}[c]{0.49\textwidth}
 \centering
\includegraphics[scale = 0.18]{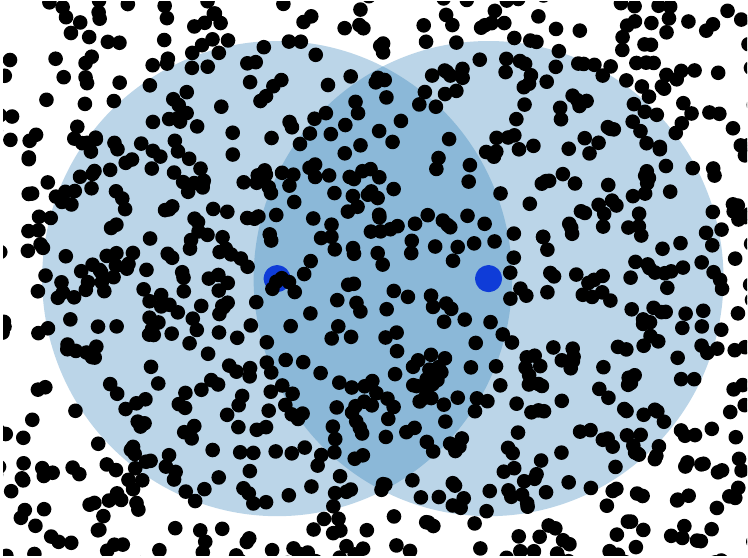}
 \end{minipage}
  \begin{minipage}[c]{0.49\textwidth}
  \centering
\includegraphics[scale = 0.18]{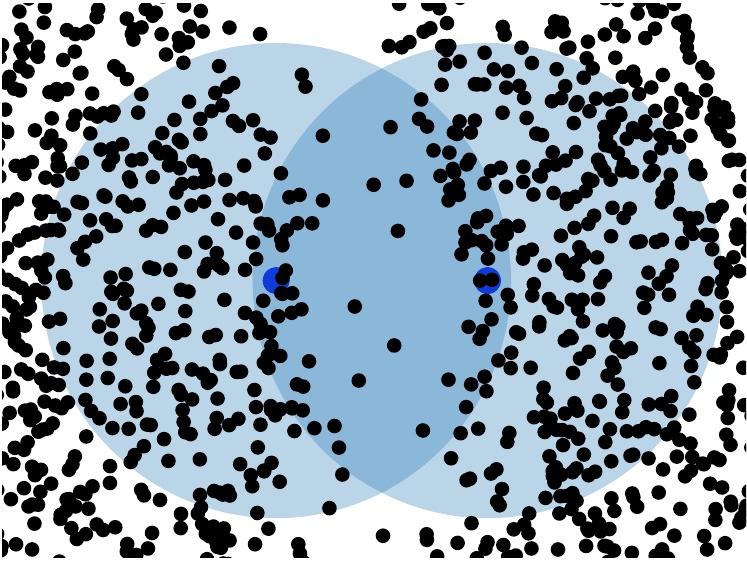}
 \end{minipage}
 \caption{For locally homogeneous data we observe \( \theta_{ij}^{(k)} \approx q_{ij}^{(k)} \) (left), whereas a significant gap is characterized by \( \theta_{ij}^{(k)} \ll q_{ij}^{(k)} \) (right)}
 \label{fig_gap}
\end{figure}
 \begin{figure}[t]
\centering
\includegraphics[width=12cm,height=2.15cm]{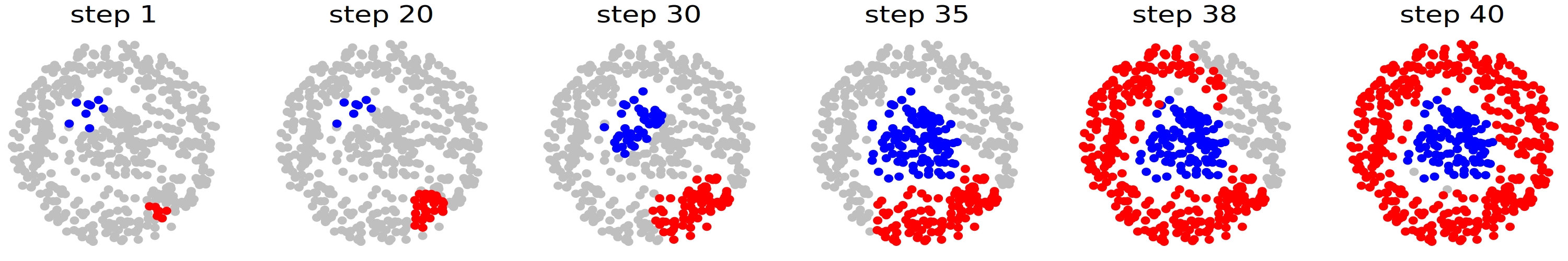}
\caption{Local clusters during different steps of the AWC algorithm}
\label{fig_steps}
\end{figure} 
In Figure \ref{fig_gap}, we visualize the relationship between those two quantities. The idea of a significant gap is formalized using a likelihood-ratio test of the null hypothesis
\[H_0 : \theta_{ij}^{(k)} \geq q_{ij}^{(k)} \]
against the alternative \[H_1 : \theta_{ij}^{(k)} < q_{ij}^{(k)}. \]
Suppose we are given binary weights \( w_{ij}^{(k-1)} = \mathds 1(\|X_i-X_j\| \leq h_{k-1}) \) and let us denote the local cluster around \( X_i \) of radius \( h_{k-1} \) by \( \mathcal C_i^{(k-1)} = \{X_j:w_{ij}^{(k-1)}=1\}  \). Then the corresponding test statistic can be written as
 \begin{equation}\label{eqn_test_statistic}
  T_{ij}^{(k)} =N_{i\lor j}^{(k)} \mathcal K(\widetilde\theta_{ij}^{(k)}, q_{ij}^{(k)})\left({\mathds 1}({\widetilde\theta_{ij}^{(k)}< q_{ij}^{(k)}})-\mathds 1({\widetilde\theta_{ij}^{(k)}\geq q_{ij}^{(k)}})\right), 
 \end{equation}
 where  \[N_{i\lor j}^{(k)} = \sum_{l \neq i, j}\mathds 1{(X_l \in \mathcal C_i^{(k-1)}\cup C_j^{(k-1)})}\]
 denotes the \emph{empirical  mass of the union}, \( \mathcal K(\alpha, \beta) \) denotes the Kullback-Leibler divergence of two Bernoulli variables with means \( \alpha \) and \( \beta \)  and
 \[\widetilde\theta_{ij}^{(k)}  = \frac{\sum_{l \neq i, j}\mathds 1{(X_l \in \mathcal C_i^{(k-1)}\cap C_j^{(k-1)})}}{N_{i\lor j}^{(k)}}\]
is an estimator for the gap coefficient. In the AWC algorithm, the assumption of the weights being of the non-adaptive form \( w_{ij}^{(k-1)} = \mathds 1(\|X_i-X_j\| \leq h_{k-1}) \) will only be guaranteed for the first step, as the weights are successively updated as 
\[w_{ij}^{(k)} =  \mathds 1(d(X_i, X_j)\leq h_k)\mathds 1(T_{ij}^{(k)}\leq \lambda)\]
 for some parameter \( \lambda \in\mathbb R \). That is, the so-called \emph{test of no gap} given in \eqref{eqn_test_statistic} that is used in the procedure does not necessarily coincide with the likelihood-ratio test, complicating the theoretical study. However, those successive updates allow the weights to carry information from all previous steps and enable the algorithm to detect gaps at any scale, in particular at a significantly smaller scale than the size of the final clusters.
 
 The output of the algorithm will be a weight matrix \( \left(w_{ij}^{(K)}\right)_{i, j = 1}^n \). Experiments have shown this matrix to carry relevant information about the cluster structure of the data. In fact, AWC performs well on artificial and real-live data benchmarks. However, there is no theoretical guarantee, that these weights actually describe the edge-disjoint union of fully connected graphs. The lack of a well-defined global cluster objective of AWC distinguishes it from most other methods and can be seen as a disadvantage from a comparative point of view. But from a practical point of view, this allows the algorithm to adapt well to a very inhomogeneous and unknown cluster structure. Moreover, the local cluster structure can also be seen as an advantage as it allows for overlapping clusters.
 
The idea of the no gap test seems similar to a density-based method such as DBSCAN. This is in fact true on a local level in most situations. However, the absolute density levels are irrelevant for the local decisions of the AWC procedure. Thus, the results on a global level differ significantly from those obtained at a certain level of a density level tree, c.f. figure \ref{fig_counterexamples_density_tree}.

 \begin{figure}[t]
\centering
\includegraphics[scale = 0.14]{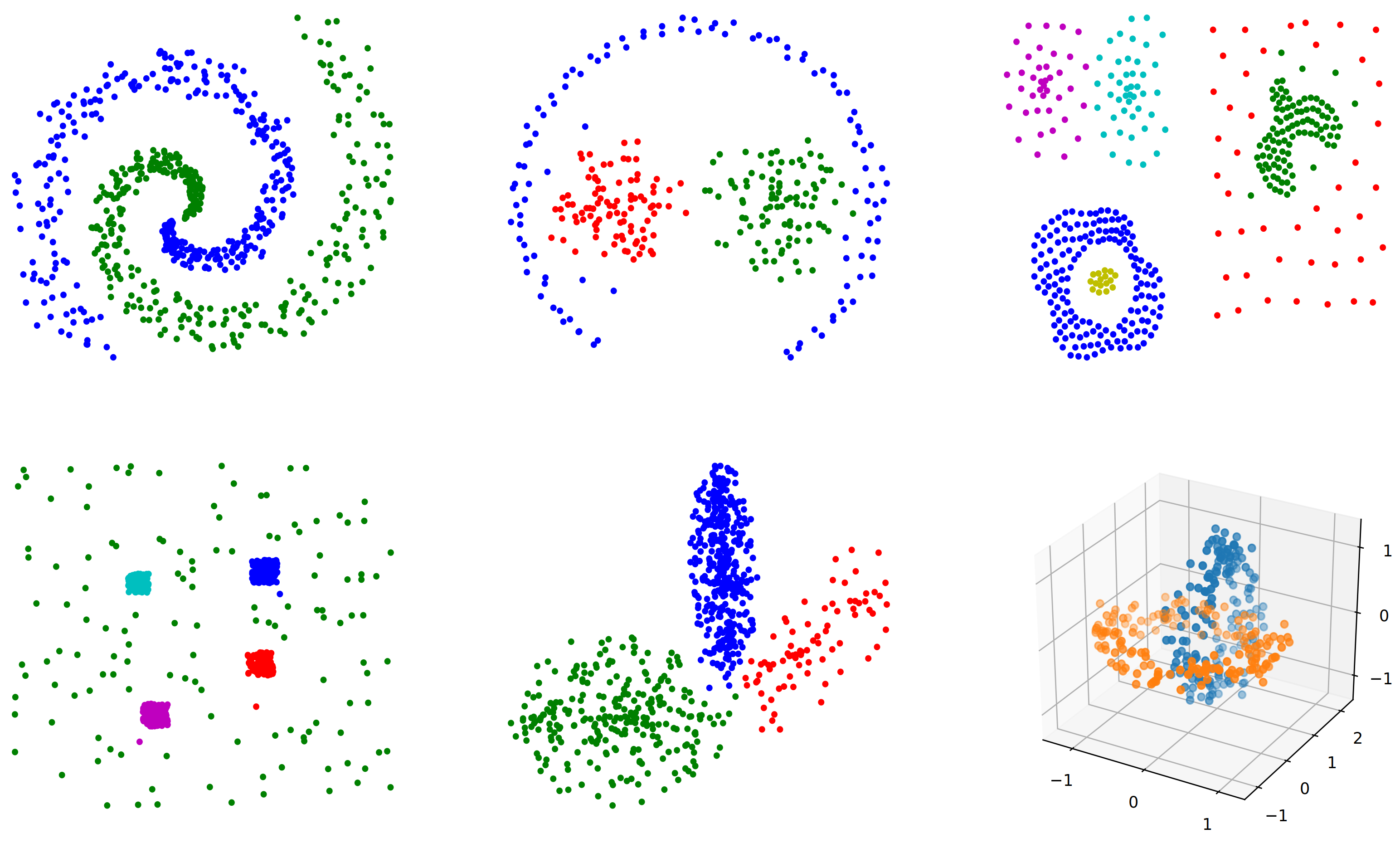}
\caption{Six artificial examples demonstrate the adaptivity of AWC w.r.t. clusters of different size and density, non-convex shapes and clusters with manifold structure. The top left and the bottom right examples are original data sets, the rest are taken from \cite{clustering_benchmarks}.}
\label{fig_examples}
\end{figure}

\kom{
\begin{figure}[t]
 \begin{minipage}[c]{0.34\textwidth}
 \centering
\includegraphics[scale = 0.23]{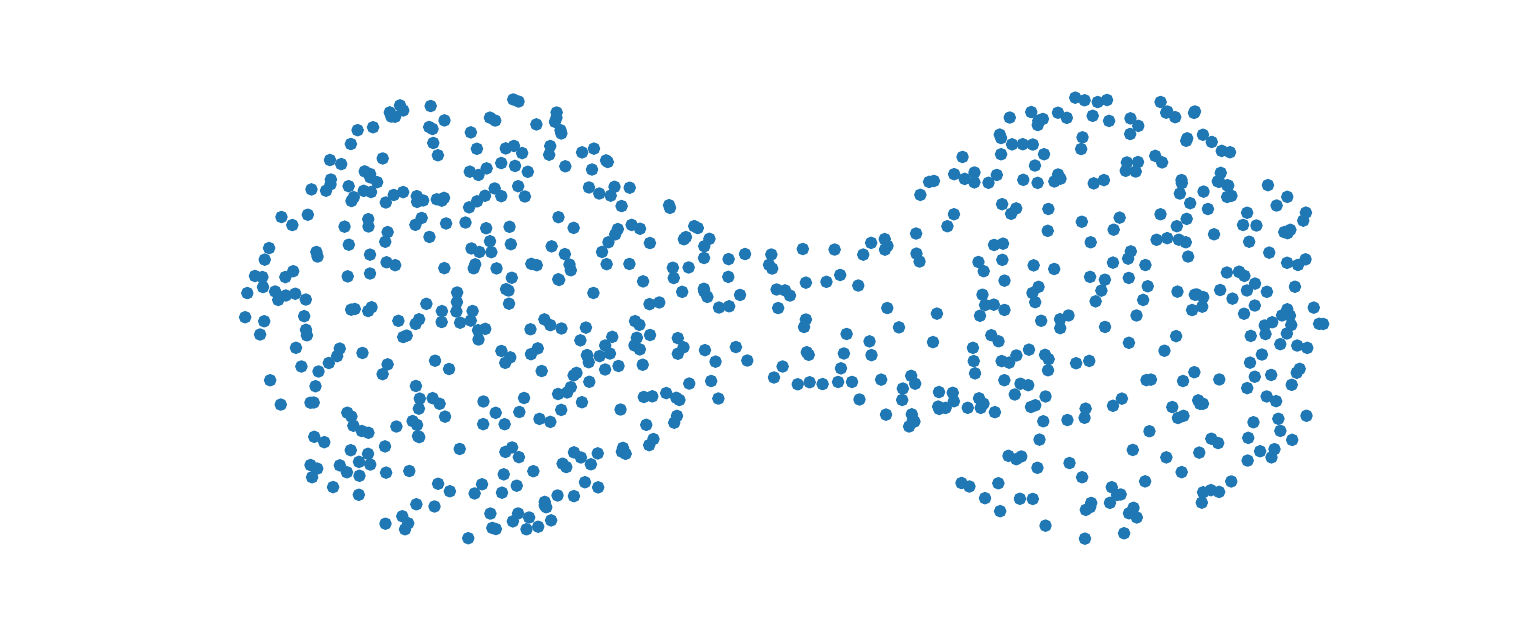}
 \end{minipage}
  \begin{minipage}[c]{0.56\textwidth}
  \centering
\includegraphics[scale = 0.28]{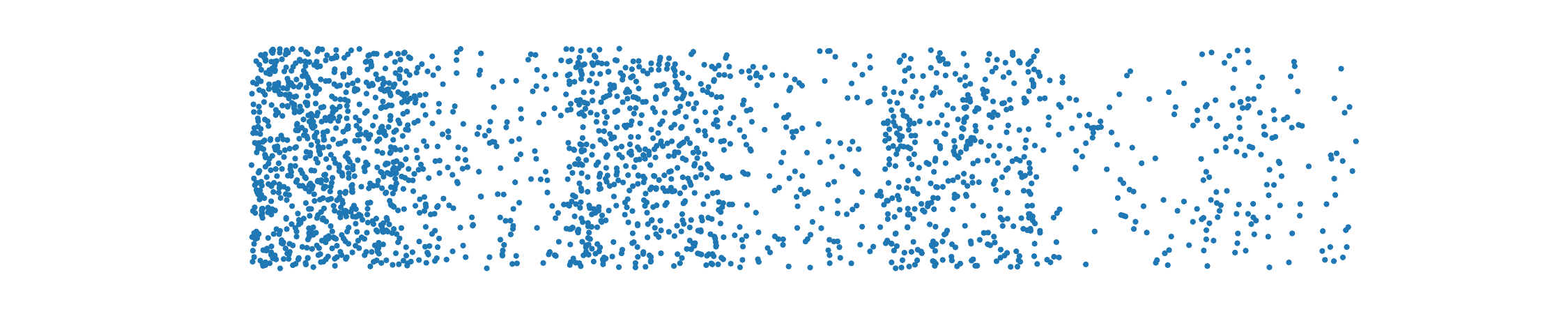}
 \end{minipage}\\
  \begin{minipage}[c]{0.34\textwidth}
 \centering
\includegraphics[scale = 0.23]{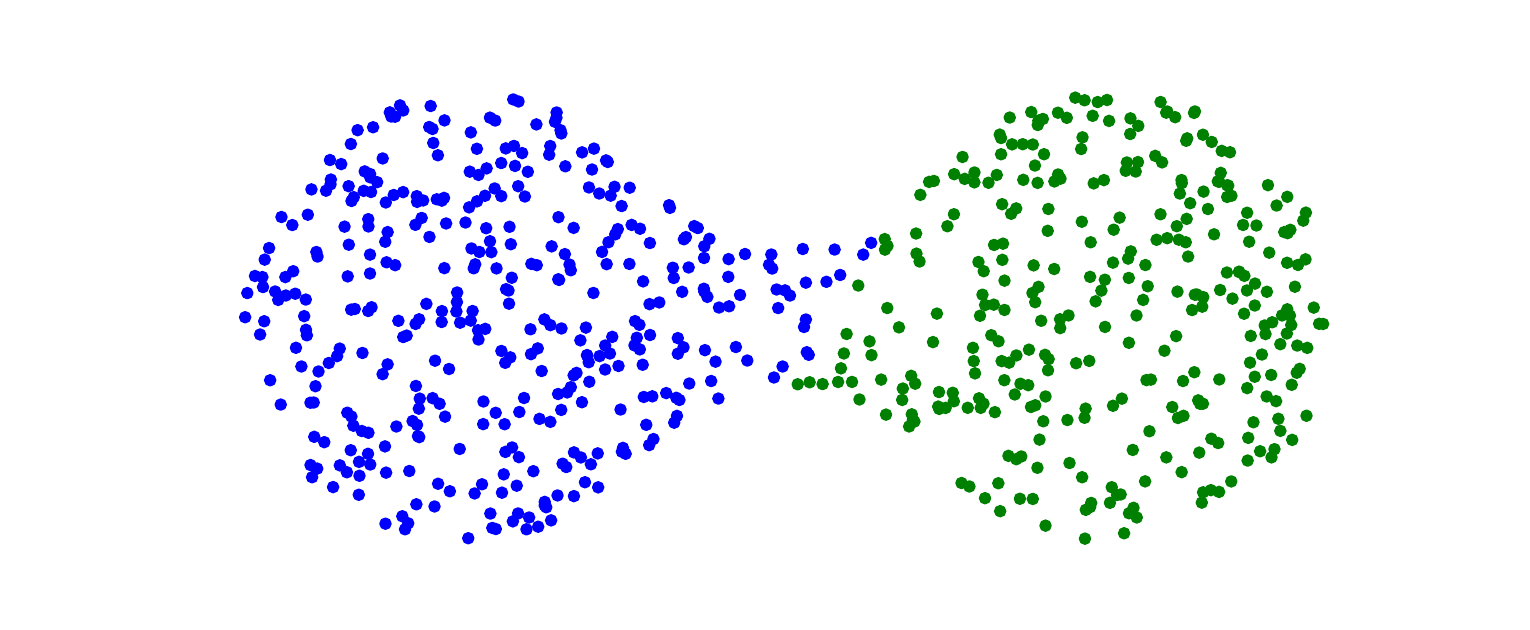}
 \end{minipage}
  \begin{minipage}[c]{0.56\textwidth}
  \centering
\includegraphics[scale = 0.28]{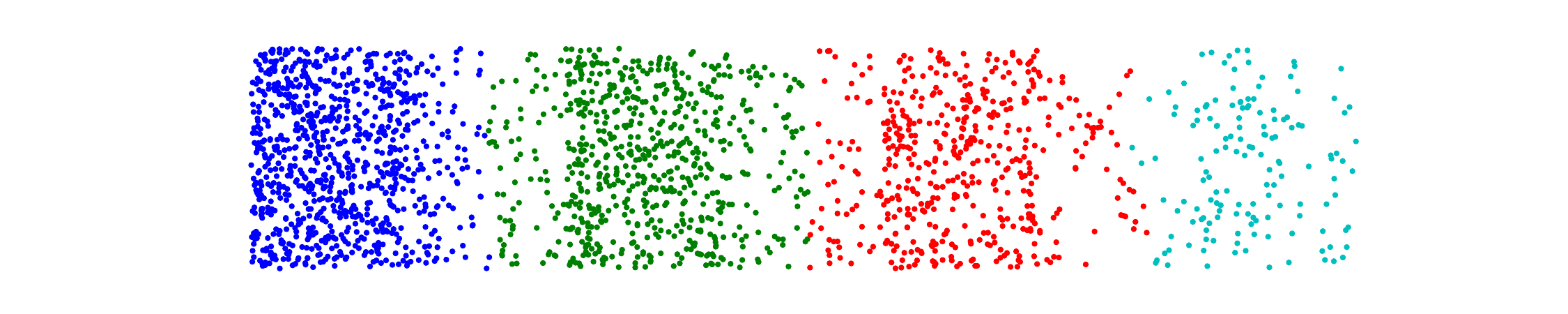}
 \end{minipage}
\caption{Two datasets (top) and the corresponding clusters obtained via AWC (bottom). These cluster structures differ from those obtained at a certain level of a density cluster tree, e.g. via DBSCAN. In the left example, this is because the density is constant, whereas in the right example the density levels of the different clusters and the spaces between them vary significantly.}
 \label{fig_counterexamples_density_tree}
\end{figure}
}

\begin{figure}
\begin{subfigure}{0.32 \textwidth}
    \includegraphics[width=\textwidth]{circles_clustered}
    \caption{AWC, \(\lambda = 10\)}
\end{subfigure}
\hfill
\begin{subfigure}{0.65 \textwidth}
    \includegraphics[width=\textwidth]{boxes_linear_clustered}
    \caption{AWC, \(\lambda = 20\)}
\end{subfigure}
\hfill
\begin{subfigure}{0.32\textwidth}
    \includegraphics[width=\textwidth]{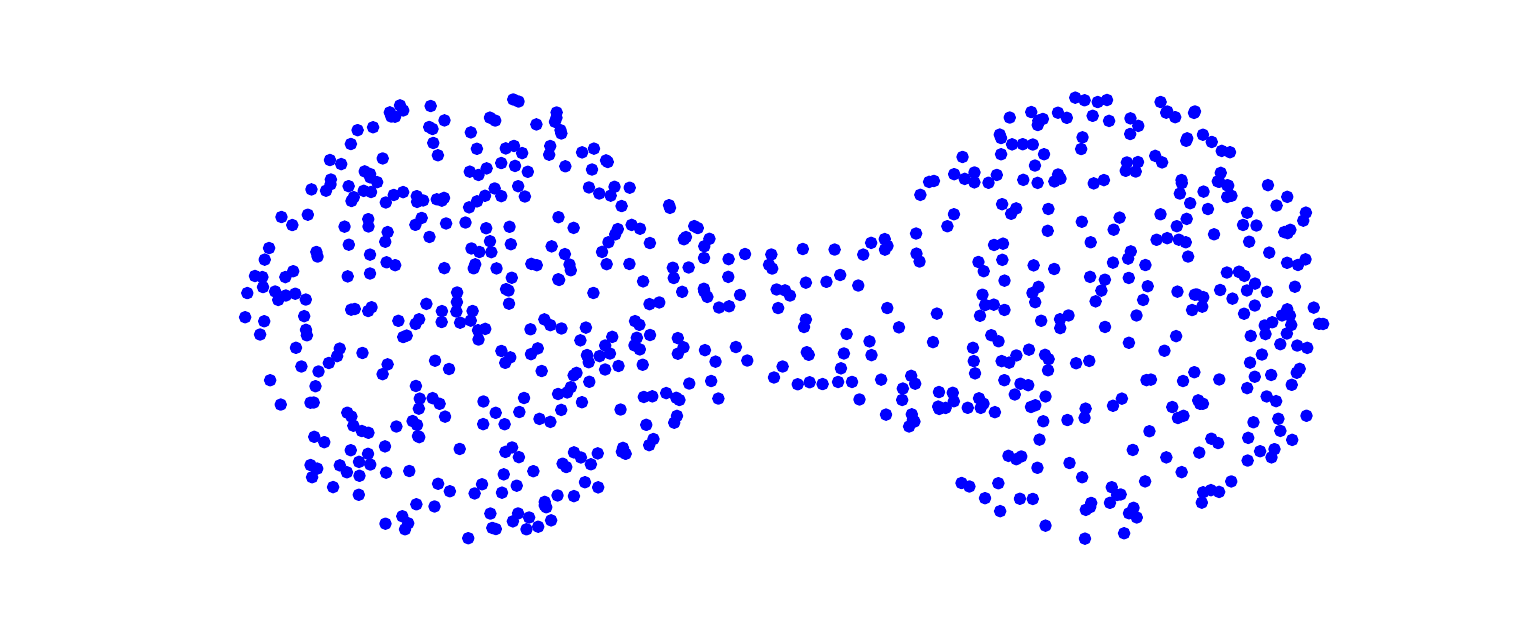}
    \caption{DBSCAN, \(\text{eps}=1.0\)}
\end{subfigure}
\hfill
\begin{subfigure}{0.65 \textwidth}
    \includegraphics[width=\textwidth]{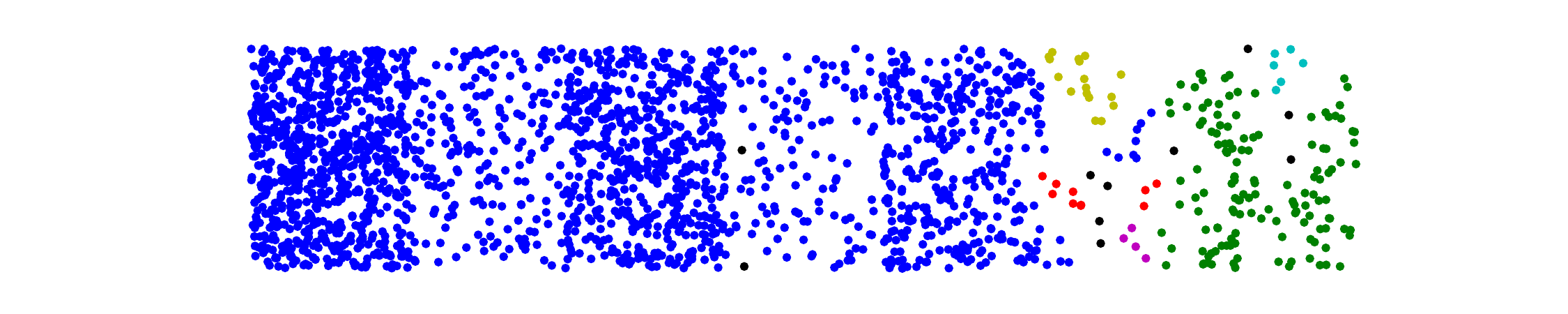}
    \caption{DBSCAN, \(\text{eps}=1.2\)}
\end{subfigure}
\hfill
\begin{subfigure}{0.32\textwidth}
    \includegraphics[width=\textwidth]{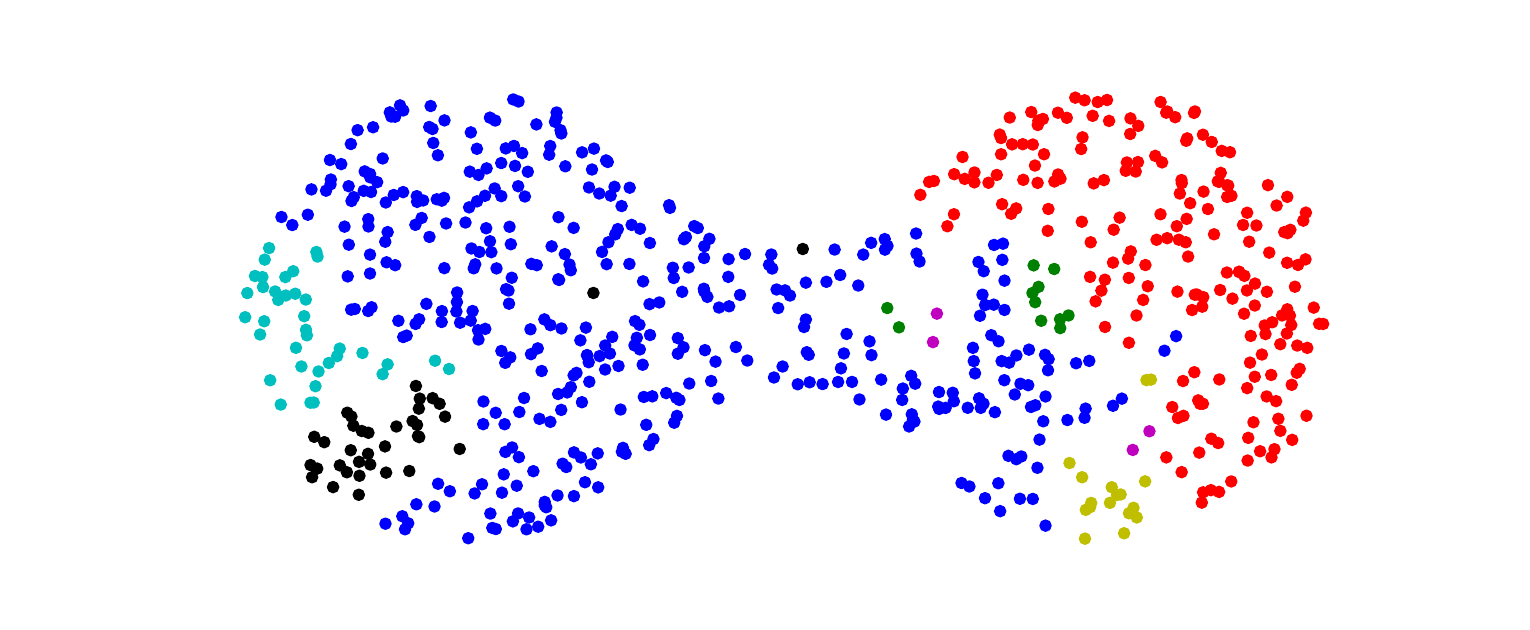}
    \caption{DBSCAN, \(\text{eps}=0.9\)}
\end{subfigure}
\hfill
\begin{subfigure}{0.65 \textwidth}
    \includegraphics[width=\textwidth]{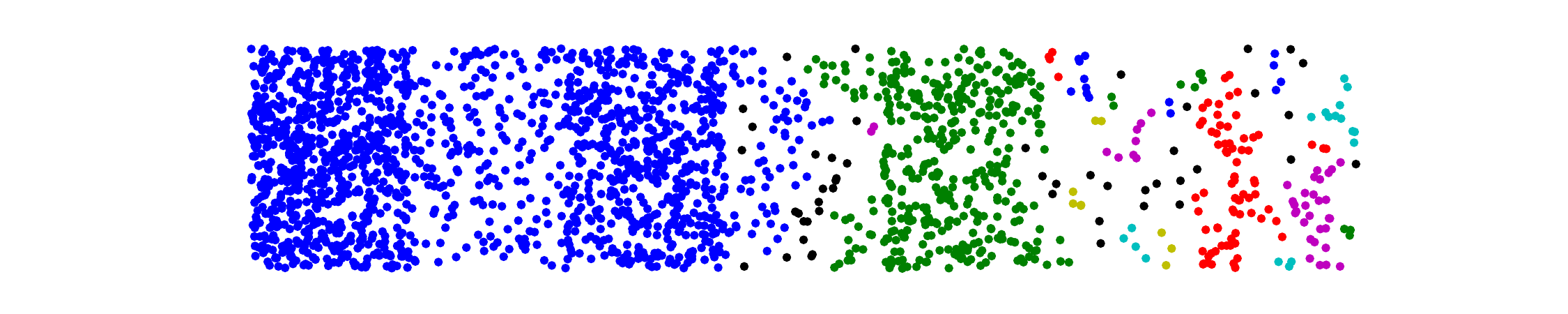}
    \caption{DBSCAN, \(\text{eps}=1.0\)}
\end{subfigure}
\hfill
\begin{subfigure}{0.32\textwidth}
    \includegraphics[width=\textwidth]{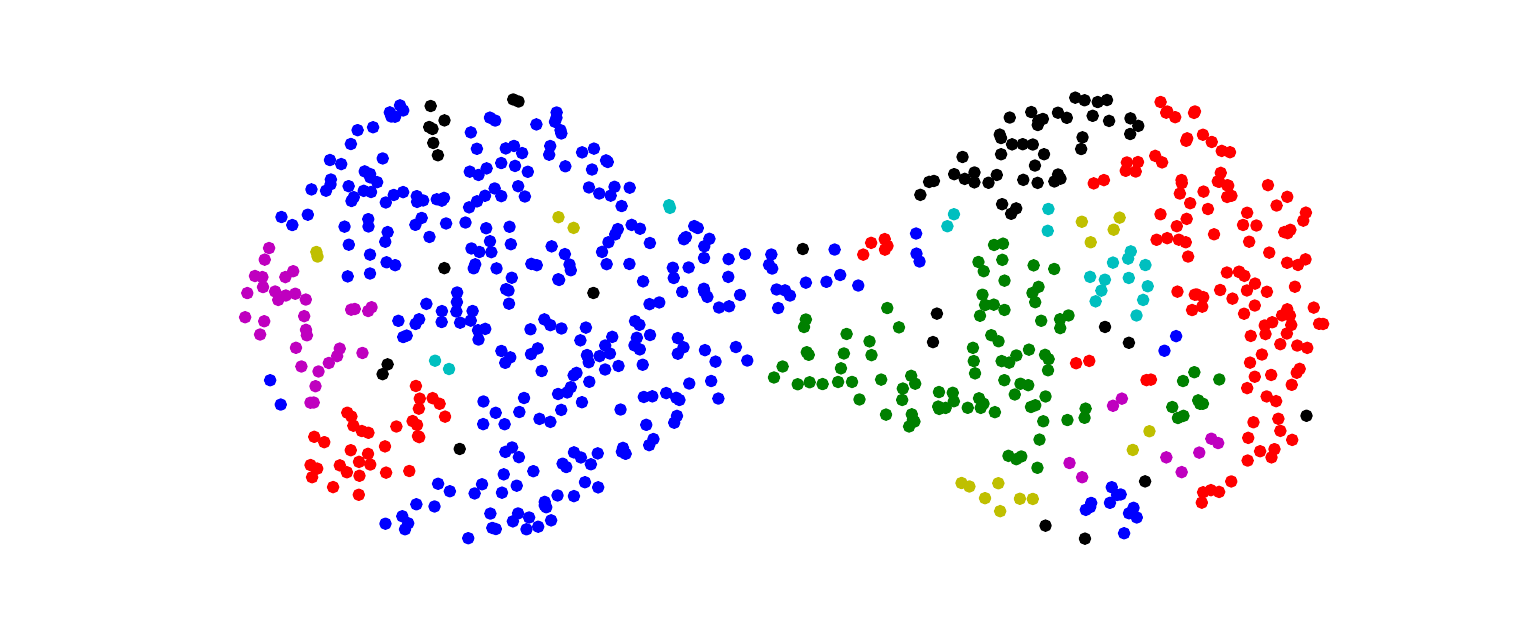}
    \caption{DBSCAN, \(\text{eps}=0.8\)}
\end{subfigure}
\hfill
\begin{subfigure}{0.65 \textwidth}
    \includegraphics[width=\textwidth]{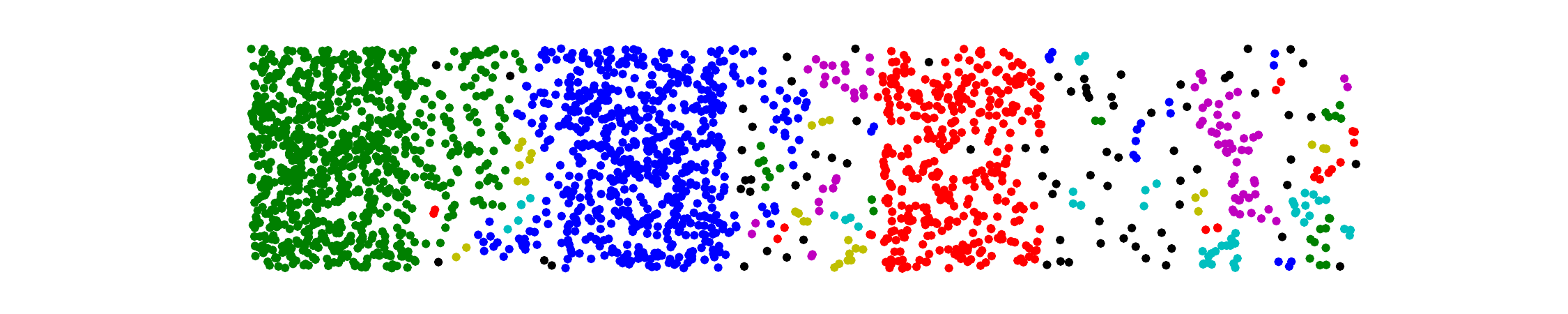}
    \caption{DBSCAN, \(\text{eps}=0.8\)}
\end{subfigure}

\caption{Two datasets and the corresponding clusters obtained via AWC and DBSCAN. The cluster structures obtained via AWC differ from those obtained at a certain level of a density cluster tree. In the left example, DBSCAN is not able to recover the cluster structure because the density is constant, whereas for the right example the density levels of the different clusters and the spaces between them vary too much.}
\label{fig_counterexamples_density_tree}
\end{figure}
 
\begin{algorithm}
\caption{Adaptive Weights Clustering (AWC)}
\label{algorithm}
\begin{algorithmic} [1]
\STATE\textbf{input}: data \( X_1, \dots, X_n\in\mathbb R^D \), a sequence of bandwidths \( 0<h_0<\dots < h_K \) and a threshold \( \lambda \in \mathbb R \) for the likelihood-ratio test
\STATE initialize the weights \( w_{ij}^{(0)} =  \mathds 1{( \|X_i - X_j\|\leq h_0)} \), \( 1\leq i, j \leq n \) 
\FOR{\( k \) from \( 1 \) to \( K \)}
\FOR{\( i\neq j \) s.t. \( \|X_i - X_j\| \leq h_k \)}
\STATE compute the empirical mass of the union \[N_{i\lor j}^{(k)} = \sum_{l \neq i, j}\mathds 1{(X_l \in \mathcal C_i^{(k-1)}\cup C_j^{(k-1)})}\]
where \( \mathcal C_i^{(k-1)}\defined \{X_j:w_{ij}^{(k-1)}=1\}  \). 
\STATE compute the estimation of the gap coefficient \[\widetilde\theta_{ij}^{(k)}  = \frac{\sum_{l \neq i, j}\mathds 1{(X_l \in \mathcal C_i^{(k-1)}\cap C_j^{(k-1)})}}{N_{i\lor j}^{(k)}}\]
\STATE compute the likelihood-ratio test statistic \[T_{ij}^{(k)} =N_{i\lor j}^{(k)} \mathcal K(\widetilde\theta_{ij}^{(k)}, q_{ij}^{(k)})\left({\mathds 1}({\widetilde\theta_{ij}^{(k)}< q_{ij}^{(k)}})-\mathds 1({\widetilde\theta_{ij}^{(k)}\geq q_{ij}^{(k)}})\right)\]
where \( \mathcal K(\alpha, \beta) =  \alpha\log\frac{\alpha}{\beta} + (1-\alpha)\log\frac{1-\alpha}{1-\beta} \) and 
\[q_{ij}^{(k)} =  \left(2\frac{\mathcal B\left(\frac{D+1}{2}, \frac{1}{2} \right)}{\mathcal B\left(1-\frac{\|X_i - X_j\|^2}{4h_{k-1}^2}, \frac{D+1}{2}, \frac{1}{2}\right)} - 1\right)^{-1}\text{} \]
with \( \mathcal B(\cdot, \cdot, \cdot) \) denoting the incomplete beta function and \( \mathcal B(\cdot, \cdot) =\mathcal B(1, \cdot, \cdot) \) denoting the usual beta function
\ENDFOR
\STATE  update the weights 
\[w_{ij}^{(k)} = \begin{cases}
                             \mathds 1(\|X_i - X_j\| \leq h_k)\mathds 1(T_{ij}^{(k)}\leq \lambda) &\text{for } 1\leq i\neq j \leq n\\
                             1 &\text{for }  1\leq i = j \leq n
                           \end{cases}\]
\ENDFOR
\STATE \textbf{output}: matrix of weights \( \left(w_{ij}^{(K)}\right)_{i, j=1}^n \) 
\end{algorithmic}
\end{algorithm}

Currently, there is a significant gap between practical and theoretical results on AWC. Experiments have shown the algorithm to deliver state-of-the-art performance on a wide range of artificial and real-life examples. Some artificial examples are shown in Figure \ref{fig_examples}. Theoretical results are fairly limited: First of all, they are limited to the case where no gaps have been detected in the previous step, as otherwise, the test of no gap does not necessarily coincide with a likelihood-ratio test. Finite sample guarantees on the propagation effect are only given at a local scale under the assumption of homogeneity due to the lack of results concerning the propagation at the boundaries of the clusters. A result about consistent separation is stated for the special case of i.i.d. data \( X_1, \dots, X_n \) from a piecewise constant density supported on three neighboring regions of equal cylindrical shape. A sufficient condition that allows consistency is that the density is smaller by a factor \( (1-\epsilon_n) \) on the middle cylinder than on the other two and that \( n\epsilon_n^2 (\log n)^{–1} \) is large enough. It turns out that this rate is optimal up to the logarithmic factor, more precisely it is impossible for any algorithm to achieve consistent separation if \( n\epsilon_n^2\nrightarrow \infty \). It has also been shown, that AWC adapts asymptotically to a linear submanifold structure of the data if the intrinsic dimension is known. However, specific conditions on the size of the considered deviation from the linear manifold are missing. Moreover, the procedure requires a crucial tuning parameter \( \lambda \). This parameter has to grow logarithmically in the data size \( n \) to ensure both propagation and separation. Unfortunately, these results do not indicate how to scale \( \lambda \), as no finite sample guarantee is given for the separation case.

In this work, we will significantly improve the current theory for AWC, and also solve some of the open problems mentioned above. First of all, we will consider distributions supported in the vicinity of closed non-linear submanifolds. We propose a slight adjustment of the algorithm in order to take into account the intrinsic dimension as well as local deviations due to the curvature of the manifold and the magnitude of the noise. In addition to generalizing the previous results to this setup, we will give finite sample guarantees both for propagation and separation and propose a theoretically justified choice for \( \lambda \) under rather general assumptions on the structure of the clusters. Moreover, we show that the propagation effect is still valid for points close to the boundary of a homogeneous cluster. This means that the propagation and separation results do no longer need to be stated separately, c.f. Corollary \ref{cor_prop_and_sep}.
The rest of the paper is organized as follows. In section \ref{section_theory} we present our main results. We start in subsection \ref{subsection_ineq_gap_coef} by introducing the manifold hypothesis and studying properties of the gap coefficient. This leads to the introduction of the so-called \emph{adjusted volume coefficient} and a minor modification of the algorithm which will preserve consistency under the manifold hypothesis. In subsection \ref{subsection_propagation} we discuss the case of uniform data without any clusters and continue in \ref{subsection_separation} by studying the sensitivity of the algorithm w.r.t. local gaps. We will show that the procedure is rate-optimal and discuss the problem of parameter tuning. Finally, we discuss the boundary case in subsection \ref{subsection_boundary_case}. In the following section \ref{section_experiments} we present numerical results illustrating the main results of section \ref{section_theory}. Proofs are collected in section \ref{section_proofs}.

\section{Theoretical results} \label{section_theory}
\subsection{Inequalities for the gap coefficient} \label{subsection_ineq_gap_coef}
When the dimension of the data is too large, the curse of dimensionality will cause the AWC procedure to fail. That is why we want to study the case where our data is locally lying approximately on a linear subspace. We start by studying the relationship between two central quantities of the algorithm. The first is the so-called \emph{gap coefficient}
\begin{equation}
q_{\mathbb P} \defined \frac{\int \mathds{1}_{B(M_1, r)\cap B(M_2, r)} d  \mathbb P}{\int \mathds{1}_{B(M_1, r)\cup B(M_2, r)} d \mathbb P}\text{,}\nonumber
\end{equation}
where \( \mathbb P \) is a probability measure on \( \mathbb R^D \) underlying our data, \( r>0 \) is a bandwidth parameter that increases subsequently by a factor \( b\in(1, 2) \) during the procedure and \( M_1 \) and \( M_2 \) are two points in \( \mathbb R^D \). We only need to compute it if \( \|M_1-M_2\| \leq br \). The purpose of this quotient is to measure whether there is a significant \emph{gap} in the data between \( M_1 \) and \( M_2 \), e.g. a region with a lower density, by comparing it to the \emph{volume coefficient}
\begin{equation}
 q \defined \frac{\int \mathds{1}_{B(M_1, r)\cap B(M_2, r)} d \mathbb \lambda}{\int \mathds{1}_{B(M_1, r)\cup B(M_2, r)} d \mathbb \lambda}\text{,}\nonumber
\end{equation}
with \( \lambda \) being the Lebegue measure. The volume coefficient in dimension \( D \) is a function of \( s\defined \frac{\|M_1 - M_2\|}{r} \) and is given by  \cite{AWC}
\begin{equation}
q = q_D(s) \defined \left(2\frac{\mathcal B\left(\frac{D+1}{2}, \frac{1}{2} \right)}{\mathcal B\left(1-\frac{s^2}{4}, \frac{D+1}{2}, \frac{1}{2}\right)} - 1\right)^{-1}\text{,}\label{def_vol_coef}
\end{equation}
where \( \mathcal B(\cdot, \cdot, \cdot) \) denotes the incomplete beta function and \( \mathcal B(\cdot, \cdot) = \mathcal B(1, \cdot, \cdot) \) denotes the beta function. As the dimension \( D \) increases, the volume coefficient decreases approximately exponentially in \( D \) as stated in the following Proposition. This demonstrates the curse of dimensionality, as we need at least an exponential growth in the data size w.r.t. the data dimension to guarantee a reasonable estimation of the gap coefficient, which is a necessity for the AWC algorithm.
\begin{prop}\label{lem_exp_growth_coefficient}
For \( 0< s < 2 \), we have
\[\frac{1}{2} \leq q_D(s) \left(\frac{ \left(1-\frac{s^2}{4}\right)^\frac{D+1}{2}}{\Gamma\left(\frac{1}{2}\right)\sqrt{d+1} }\right)^{-1} \leq  \frac{2^\frac{5}{2}}{s^2}.\]
\end{prop}
By considering locally homogeneous data lying close to a lower-dimensional submanifold of dimension \( d \), we show in the second Lemma that the gap coefficient essentially behaves locally as for homogeneous data on a linear subspace of the same dimension. We will use this in the following to prove theoretical guarantees for the AWC procedure. Let us start by listing all the assumptions on the distribution \( \mathbb P \) and the tuning parameters of the algorithm that we need - these are mainly a lower bound for the reach of the manifold on which the data is concentrated, an upper bound for the size of the additional noise in terms of the size of the considered vicinity and an upper bound for the radius of the considered vicinity in terms of the reach.
\begin{description}
\item[\label{assumptionsA} Assumptions A(\( r_0 \), \( r_1 \)):]
\end{description}     
                        \begin{itemize}
                           \item \( \mathbb P \) is the probability distribution of a random variable of the form \( X+\xi \), where \( X \) follows a density \(f\) on a manifold \( \mathcal M \) and \( \|\xi\| \leq r_\xi \) 
                           \item \( \mathcal M \) is a connected and compact d-dimensional \( C^2 \) submanifold of \( \mathbb{R}^D \) without boundary
                           \item \( \reach (\mathcal M) \geq \frac{1}{\kappa} \) for \( \kappa > 0 \) 
                           \item \( r_\xi \leq \frac{r_0}{\max\{20, 5d\}} \) 
                           \item \( r_1\leq\frac{1}{\max\{120,\sqrt{720d}\}\kappa} \) 
                           \item \( 1<b\leq\frac{b'}{(1+360\kappa^2 r_1^2)\left(1+5\frac{r_\xi}{r_0}\right)} \) for some \( b'<2 \) 
                          \end{itemize}
                    
Our assumption of bounded noise is identical to the one in the work of \cite{belakrishnan_cluster_tree_on_manifold} about the cluster density tree on manifolds and is relatively weak. It can be seen as a generalization of the so-called \emph{tubular noise} and \emph{additive noise}, c.f. \cite{different_noise_conditions_on_manifold}. Some authors additionally require orthogonality of the noise, c.f. \cite{homology_nsw_orthogonal_noise} and \cite{mfd_denoising4}. Moreover, note that the upper bound for \( b \) is not a very restrictive assumption, as it will always be satisfied for $1<b\leq\frac{3}{2}$. The complexity of the AWC algorithm with respect to \( b \) is \( \mathcal O\left(\frac{1}{\log b}\right) \), so as long as \( b \) is bounded away from \( 1 \), e.g. as long as \( b'\geq\sqrt{2} \), this does not change the overall complexity. 

\begin{prop}\label{lem_gap_coef}
Suppose assumptions \hyperref[assumptionsA]{A(\( r \), \( r \))} are satisfied for a constant density \(f\) and \( M_1 \), \( M_2 \) are two points in the support of \( \mathbb P \) whose distance is at most \( br \). Then
\begin{equation}
(1+\varepsilon_{\mathcal M})^{-1} (1+\varepsilon_\xi)^{-1} \leq \frac{q_{\mathbb P} }{q_d(s)} \leq (1+\varepsilon_{\mathcal M}) (1+\varepsilon_\xi)\nonumber
\end{equation}
for
\[\varepsilon_{\mathcal M} \defined \frac{9600  (d + 1) \kappa^2r^2 }{\left(1-\left(\frac{b'}{2}\right)^2\right)^\frac{d+1}{2}}\]
and
\[\varepsilon_\xi \defined  \frac{80 (d + 1) \frac{r_\xi}{r}}{\left(1-\left(\frac{b'}{2}\right)^2\right)^\frac{d+1}{2}}.\]
\label{prop_gap_coeff_uniform_case}
\end{prop}
\kom{
\begin{rmk}
If we replace \(\varepsilon_{\mathcal M}\) in the previous Lemma by the upper bound \[\varepsilon_{\mathcal M} \leq \frac{84  (d + 1) \kappa r }{\left(1-\left(\frac{b'}{2}\right)^2\right)^\frac{d+1}{2}} = \mathcal O(\kappa r),\]
we can replace the assumption \[ r_1\leq\frac{1}{\max\{120,\sqrt{720d}\}\kappa} \]
by the slightly weaker condition
 \[r \leq\frac{1}{\max\{48, 6d\}\kappa}. \]
\end{rmk}
}
Let us point out that our bound on the deviation of the gap coefficient from the volume coefficient is a product of the form \( \left(1+\mathcal O(\kappa^2 r^2)\right) \left(1 + \mathcal O\left(\frac{r_\xi}{r}\right)\right) \), as long as the intrinsic dimension \( d \) is bounded and as long as \( b' \) is bounded away from 2. The first factor takes into account the reach of the manifold, whereas the second factor only depends on the size of the noise. In particular, using a manifold denoising algorithm \cite{mfd_denoising1, mfd_denoising2, mfd_denoising3, mfd_denoising4}, we can preprocess our data in order to reduce noise and expect the second factor to be irrelevant. Thus, it might also be reasonable to study a setup without noise as in the following trivial Corollary.
\begin{cor}
Suppose \( r_\xi = 0 \) in addition to the assumptions of Proposition \ref{lem_gap_coef}. Then
\begin{equation}
(1+\varepsilon_{\mathcal M})^{-1} \leq \frac{q_{\mathbb P} }{q_d(s)} \leq 1+\varepsilon_{\mathcal M}.\nonumber
\end{equation}
\end{cor}
Recall that the main idea of the AWC algorithm is to distinguish a homogeneous area from a gap between two clusters by estimating and comparing the gap coefficient with the volume coefficient. However, due to the non-linear manifold structure as well as the noise, we cannot establish a strict inequality between the two quantities even for the uniform case. Nevertheless, Proposition \ref{lem_gap_coef} guarantees a strict inequality for the homogeneous case if we adjust the volume coefficient by a factor \( (1+\varepsilon_{\mathcal M})^{-1} (1+\varepsilon_{\xi})^{-1}  \). Consequently, we will adjust the proposed test of the AWC procedure to
 \[T_{ij}^{(k)} \defined  N_{i\lor j}^{(k)}\mathcal K\left(\widetilde\theta_{ij}^{(k)}, \mathfrak q_{ij}^{(k)}\right)\left\{\mathds 1\left({\widetilde\theta_{ij}^{(k)}< \mathfrak q_{ij}^{(k)}}\right)-\mathds 1\left({\widetilde\theta_{ij}^{(k)}\geq \mathfrak q_{ij}^{(k)}}\right)\right\}\]
by considering an \emph{adjusted volume coefficient} \[\mathfrak q_{ij}^{(k)}\defined (1+\varepsilon_{\mathcal M})^{-1} (1+\varepsilon_{\xi})^{-1} q_d\left(\frac{\|X_i-X_j\|}{h_{k-1}}\right).\]
Note that in practice, the parameters \( d \), \( \frac{1}{\kappa} \) and \( r_\xi \) are unknown. We refer to \cite{dimension_estimation} for an overview of procedures dedicated to estimating the intrinsic dimension. The estimation of the noise is related to the estimation of the manifold and is particularly related to the problem of recovering the projections of the data onto the manifold, see \cite{mfd_denoising4}. The estimation of the reach has been studied in \cite{reach_estimation}. However, the effect of the reach is locally small and can be ignored. Similarly, using a manifold denoising algorithm, we can assume the effect of the noise to be insignificant. In contrast, the effective dimension parameter is crucial for the computation of the test statistic. Following the proofs of theorems \ref{thm_propagation} and \ref{thm_separation}, we see that the AWC procedure is still consistent in case of overestimation of \(d\) as long as the gap is significant enough. However, we cannot expect the algorithm to be rate optimal in this case. In subsection \ref{subsection_experiment_intrinsic_dimension_effect} we discuss a simple numerical example, that suggests that the procedure might be stable in practice w.r.t. to over- and underestimation of \(d\).

\subsection{Propagation in the uniform case} \label{subsection_propagation}
 In the following, we generalize the results from \cite{AWC} to our considered setup. As expected, the adjusted AWC algorithm consistently propagates homogeneous areas of our data: If the threshold \( \lambda \) of our likelihood-ratio test is of the form \( C \log n \), then the accuracy in estimating the weights of the adjacency matrix is of order \( 1-\mathcal O \left(n^{-(C-3)}\right) \).

\begin{thm}\label{thm_propagation}
With high probability, the AWC algorithm does not detect a gap between two points from a distribution that is nearly uniform on a manifold, as long as it did not detect any gaps in the previous step. To be precise, suppose assumptions \hyperref[assumptionsA]{A(\( h_{k-1} \), \( h_{k-1} \))} hold and \( X_1, X_2, \dots, X_{n} \overset{i.i.d.} {\sim}\mathbb P \). We consider a constant density \(f\) and assume that the AWC algorithm did not detect any gaps in the previous step. If we choose the threshold \( \lambda = C \log n \) for some \( C>0 \), then \[\mathbb P^{\otimes n} \left(T_{ij}^{(k)}> C\log n \middle\vert \|X_i - X_j\| \leq h_{k}\right) \leq 2n^{–C}.\]
\end{thm}

\begin{cor}\label{cor_propagation}
With high probability, the AWC algorithm does not detect any gaps if our data distribution is close to a uniform distribution on a submanifold of \(\mathbb R^D\). To be precise, suppose assumptions \hyperref[assumptionsA]{A(\( h_{0} \), \( h_{K} \))} hold and \( X_1, X_2, \dots, X_{n} \overset{i.i.d.} {\sim}\mathbb P \). We consider a constant density \(f\). If \( K<n \) and we choose the threshold \( \lambda = C \log n \) for \( C>3 \), then
\[\mathbb P^{\otimes n} \left( w_{ij}^{(K)} = \mathbbm 1(\|X_i-X_j\|\leq h_K) \forall i, j\right)  \geq 1 - 2n^{-(C-3)}.\]
\end{cor}

\begin{rmk}
By symmetricity, a linear density also satisfies the no gap condition in the full dimensional case \(d = D\). So up to the constants in the terms \(\varepsilon_{\mathcal M}\) and \(\varepsilon_\xi\), Proposition \ref{prop_gap_coeff_uniform_case} is still valid if the underlying density is of the form \(f=\mathbbm 1(\mathcal M) \mathfrak f\) for a linear function \(\mathfrak f\) on \(\mathbb R^D\). Consequently, the above results on propagation in the uniform case can also be generalized to this linear model.
\end{rmk}

\subsection{Separation in the gap case} \label{subsection_separation}
For the case of a significant gap in the data, we can also generalize the results of \cite{AWC} to the manifold setup and show that we consistently separate the data achieving nearly rate-optimality. In addition, we give a finite sample guarantee. Together with the previous results for the homogeneous case, this yields a first theoretically justified proposal to choose the parameter \( \lambda \). Moreover, we do not only generalize from a linear to a smooth subspace structure of our data but also significantly generalize the definition of the considered clusters.
\begin{description}
\item[\label{assumptionsB} Assumptions B(\( r \)):]
\end{description}     
\begin{itemize}
 \item First of all, we include assumptions \hyperref[assumptionsA]{A(\( r \), \( r \))} 
 \item Additionally, we consider disjoint subsets \( \mathcal C_1, \dots, \mathcal C_{k_{\mathcal C}} \) of \( \mathcal M \) 
  \item Spatial separation of clusters is ensured by\[d_\infty(\mathcal C_{l}, \mathcal C_{m}) \defined \min_{x \in\mathcal C_{l}, y \in \mathcal C_{m}} \|x-y\| \geq r+2 r_\xi \quad \text{   for   }\quad  1\leq l\neq m \leq k_{\mathcal C}\]
 \item  Similarly as in \cite{rigollet}, we assume a thickness condition on each cluster: We assume there is a constant \( f_0>0 \) s.t. for any \( x\in \mathcal C_l \) and \( r'\in [r-2r_\xi, r+2r_\xi] \) we have
 \[\int f \mathds 1_{B(x, r')}\geq f_0 \int \mathds 1_{B(x, r') \cap \mathcal M} \]
 \item Separation of clusters is also ensured by a significant depth of the gap: For \(x_1\in \mathcal C_l, x_2\in \mathcal C_m\), \(r'\in [r-2r_\xi, r+2r_\xi] \) with \(l\neq m\) and $\|x_1-x_2\| \leq br$ we have
 \[\int f \mathds 1_{B(x_1, r')\cap B(x_2, r')} \leq (1-\epsilon)f_0 \int \mathds 1_{B(x_1, r')\cap B(x_2, r') \cap \mathcal M} \]
 \item The sample size \( n \) has to be large enough, i.e. for some \(\beta>0\) we have \[\frac{n}{\log n} \geq \frac{2\beta}{ z_k^2}\] where \( f_0^{-1}z_k \) denotes the volume of a \(d\)-dimensional ball of radius \(\frac{7}{8} r\)
 \item The depth \( \epsilon < 1 \) of must be significant w.r.t. the effect of curvature and noise, and decreases not faster than \( (\log n)^{\frac{1}{2}}n^{-\frac{1}{2}} \), i.e. it satisfies the lower bound
\[\epsilon \geq \max\left\{7(\varepsilon_{\mathcal M} +\varepsilon_\xi +  \varepsilon_{\mathcal M} \varepsilon_\xi ), \sqrt{\frac{2\alpha \log n}{z_k q_d^{2}(b) n}}\right\} \text{}\]
for some \( \alpha > \beta \). 
\end{itemize}
The integral conditions are up to a change of constants a generalization of the simpler separation condition
\[\esssup_{\mathcal M \setminus \cup_i \mathcal C_i} f \leq (1-\epsilon)\inf_{\cup_i \mathcal C_i} f \]
from \cite{separation_condition_clusters_for_tree}. However, the here introduced generalization allows for both smooth  \(f\) as well as a step function. Moreover, the upper bound on the size of the bounded noise \(\epsilon \gtrsim \frac{r_\xi}{r} d\) also appears in the work of \cite{belakrishnan_cluster_tree_on_manifold} (with parameters \((\theta, \sigma)\) instead of \((r_\xi, r)\)).

The assumptions above are designed to be comparable to the framework of other density-based methods. However, AWC does not reconstruct connected superlevel sets of the unterlying density. Conversely, other procedures will in general not find a cluster structure respecting the idea of significant gaps. Moreoever, theoretical guarantees for AWC are only given for local clusters. In general, it is difficult to assign a global partition of the data from this information, as the local clusters might form connected components that are heavily overlapping. This limits the comparability of the presented results to a local level.
\begin{thm}\label{thm_separation}
We consider a distribution on the vicinity of a submanifold of $\mathbb R^D$ containing different clusters separated by significant gaps in the density. As long as the AWC algorithm did not detect gaps in the previous step, it will detect the gap between two points from different clusters with high probability. To be precise, consider the assumptions \hyperref[assumptionsB]{B(\( h_{k-1}\))} and \( X_1, X_2, \dots, X_{n+2} \overset{i.i.d.} {\sim} \mathbb P \). Suppose that the algorithm did not detect any gaps in the previous steps. Then \[\mathbb P^{\otimes (n+2)}\left(T_{ij}^{(k)}\geq \left(\sqrt{\alpha}-\sqrt{\beta}\right)^2 \log n\Big\vert \substack{ \|X_i-X_j\|\leq h_k\\ \exists l\neq m: X_i\in\mathcal C_l^{r_\xi}, X_j\in\mathcal C_m^{r_\xi}}\right) \geq 1-3n^{-\beta}.\]
\end{thm}

\begin{rmk}
Under the previous assumptions, the gap will be consistently detected at the step \( k \) where the considered vicinity first exceeds the width of the gap. However, as in the homogeneous case, the speed of convergence depends on the choice of the tuning parameter \( \lambda \). Theorems \ref{thm_propagation} and \ref{thm_separation} suggest choosing a threshold of the form \( \lambda = C \log n \). Moreover, the optimal constant \( C^* \) that yields the fastest convergence \( w_{ij}^{(k)}\xrightarrow{}w_{ij} \) in probability for both discussed cases according to the given lower bounds for the accuracy of the estimation of the weights is given by
\begin{align*}
 C^{*} &= \sup\limits_{\beta\in(0, \alpha)}\min\left\{\left(\sqrt{\alpha}-\sqrt{\beta}\right)^2, \beta\right\}\\
 &=\frac{\alpha}{4}.
\end{align*}
The corresponding rate of misclassification is for both cases 
\[\mathbb P^{\otimes n} \left(w_{ij}^{(k)} \neq w_{ij}\right) \leq \mathcal O(n^{-\frac{\alpha}{4}}).\]
\end{rmk}

\begin{rmk}
We consider a low manifold dimension \(d\) as a reasonable assumption and thus consider only asymptotics in \(n\) while \(d\) is bounded from above. While the rate of the algorithm is essentially (i.e. up the involved constants) independent of \(d\), we have the following dependencies on \(d\):
\begin{itemize}
 \item To guarantee a fixed level of uncertainty, i.e. with fixed \(\beta\), the lower bound on the sample size \(n\) in the list of assumptions increases exponentially in \(d\), demonstrating the curse of dimensionality if the manifold dimension is very large.
 \item For larger \(d\) we allow a smaller level of noise \(\propto d^{-1}\) and a smaller size of the considered vicinity \(\propto d^{-\frac{1}{2}}\).
\end{itemize}
\end{rmk}

\subsection{Boundary case}\label{subsection_boundary_case}
In the previous subsection \ref{subsection_propagation} we considered a homogeneous distribution on the manifold. In the presence of non-trivial clusters, this assumption can only be satisfied locally and only for points far enough from the boundaries of the clusters. However, the no gap condition enjoys the remarkable property that is still valid for points close to a locally linear boundary. In fact, the corresponding gap coefficient might only be larger than in the homogeneous case.
\begin{lem}\label{lem_half_space_coeff}
We assume $M_1\neq M_2\in\mathbb R^D$ and $r_1, r_2 >0$. Moreover, suppose that $\mathcal H$ is a $D$-dimensional half-space containing $M_1$ and $M_2$. Then
\[\frac{\lambda( B(M_1, r) \cap B(M_2, r_2))}{\lambda (B(M_1, r) \cup B(M_2, r_2))} \leq \frac{\lambda(\mathcal H\cap B(M_1, r) \cap B(M_2, r_2))}{\lambda(\mathcal H \cap (B(M_1, r) \cup B(M_2, r_2)))} \]
\end{lem}
The proof of Lemma \ref{lem_half_space_coeff} relies on the following result via Fubini's theorem. Again, we assume \(D>0\) and denote the \(D\)-dimensional Lebesgue measure by \(\lambda\).
\begin{lem}\label{lem_plane_monotonicity}
Suppose \(M_1\neq M_2\in\mathbb R^{D+1}\) and \(r >0\). We consider a hyperplane \(H\subset \mathbb R^{D+1}\) containing \(\frac{M_1+M_2}{2}\). Suppose \(v\) is vector of norm 1 that is orthogonal to \(H\). Moreover, we define \(t_\text{max}\defined \sup \{t : (H + tv) \cap (B(M_1, r) \cup B(M_2, r))\neq \{\}\}\). Then the function \(\mathcal Q: \mathopen{[}0, t_\text{max}\mathclose{)} \to \mathbb R_{\geq 0}\),
\begin{align*}
 \mathcal Q(t)\defined  \frac{\lambda((H + tv)\cap B(M_1, r) \cap B(M_2, r))}{\lambda((H + tv) \cap (B(M_1, r) \cup B(M_2, r)))}
\end{align*}
is monotonely decreasing in \(t\).
\end{lem}
The quantity \(\mathcal Q\) in the result above is a generalization of the volume coefficient in a lower dimension: The intersection of the considered hyperplane with each ball is again a ball of a lower dimension - however, the corresponding radii are in general not identical.

Lemma \ref{lem_plane_monotonicity} shows in fact more than what is claimed in Lemma \ref{lem_half_space_coeff}: As we move the two center points closer to the linear boundary, the volume coefficient starts increasing monotonely as soon as the two balls are not completely contained by the half-space anymore. At some point, the volume coefficient attains its maximum, after which it decreases monotonely. By symmetricity, the volume coefficient has the same value again as in the homogeneous case, when the boundary of the half-space contains \(\frac{M_1+M_2}{2}\). If we consider a stepfunction
\begin{equation}
f \propto \mathbbm 1(\mathcal H \cap (B(M_1, r)\cup B(M_2, r)))+ (1-\epsilon) \mathbbm 1(\mathcal H^C \cap (B(M_1, r)\cup B(M_2, r))) \label{eq_density_half_space_sketch}
\end{equation}
as a generalization of the uniform density considered in Lemma \ref{lem_half_space_coeff}, we observe the analogue monotonicity, if we move the two center points further away from the half-space \(\mathcal H\), c.f. Figure \ref{fig_half_space_visualization}.
\begin{figure}[t]
\begin{center}
\begin{subfigure}{0.6 \textwidth}
    \includegraphics[width=\textwidth]{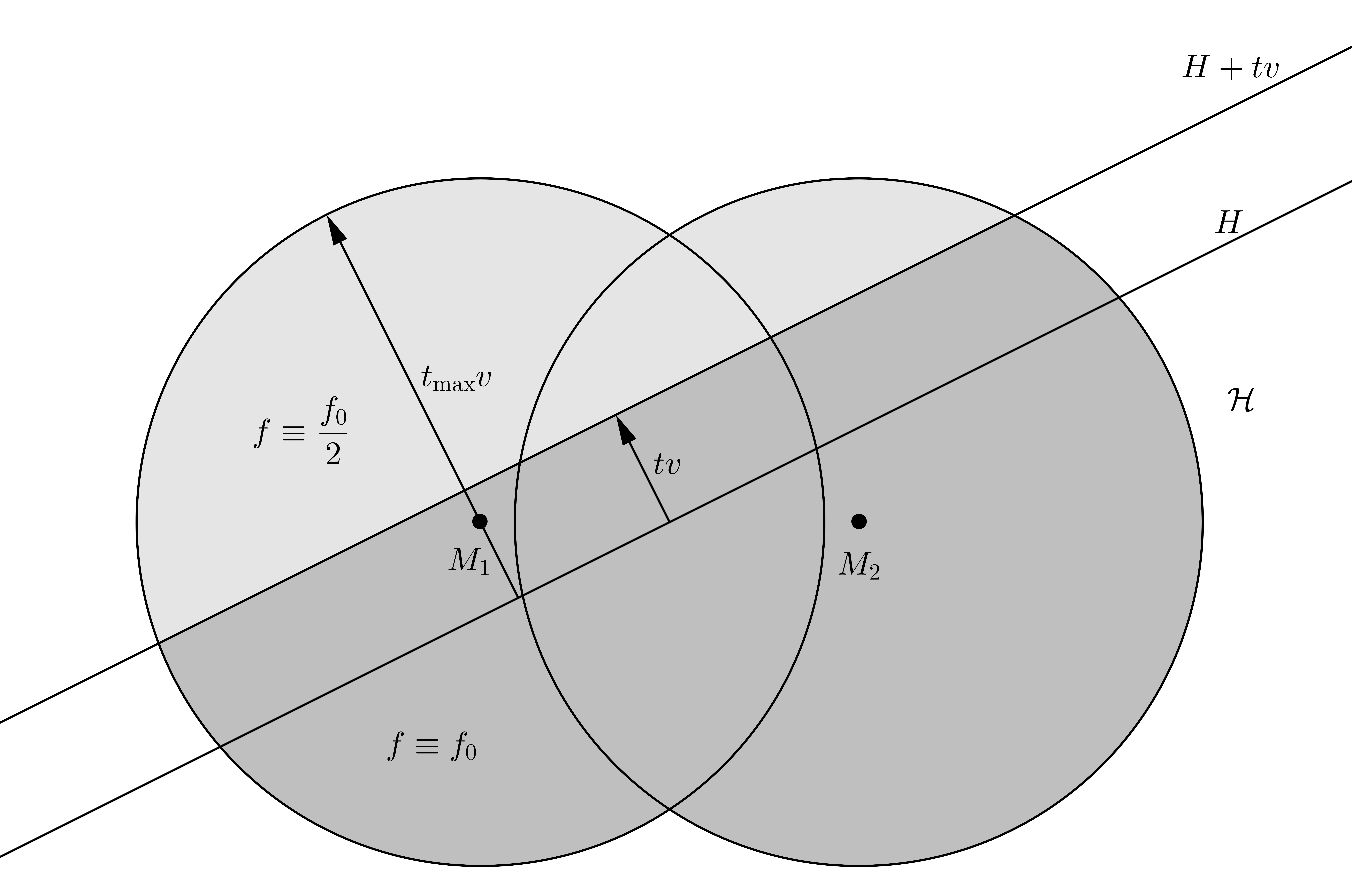}
\end{subfigure}\\

\begin{subfigure}{0.3\textwidth}
    \includegraphics[width=\textwidth]{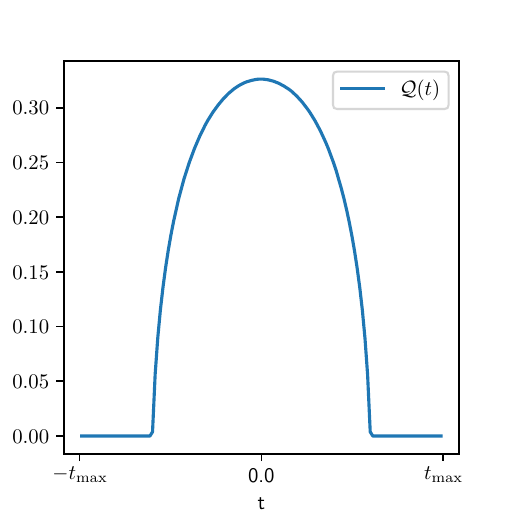}
\end{subfigure}
\begin{subfigure}{0.3 \textwidth}
    \includegraphics[width=\textwidth]{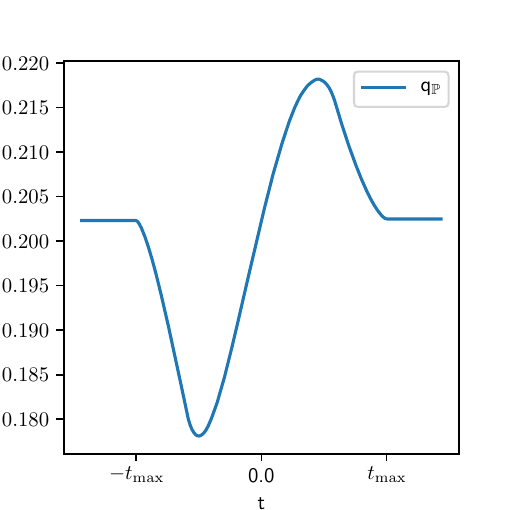}
\end{subfigure}

\end{center}
\caption{The top sketch illustrates the notation and relation between Lemma \ref{lem_half_space_coeff} and \ref{lem_plane_monotonicity}: The half-plane \(H+tv\) is the boundary of the half-space \(\mathcal H\). However, the uniform assumption of Lemma \ref{lem_half_space_coeff} is modified to a piecewise constant density as described in \eqref{eq_density_half_space_sketch} with \(\epsilon = \frac{1}{2}\). At the bottom, we see a plot of the corresponding function \(\mathcal Q(t)\) from Lemma \ref{lem_plane_monotonicity} (left) as well as the gap coefficient \({q_\mathbb P}$ (right). These values were obtained by Monte Carlo integration.}
\label{fig_half_space_visualization}
\end{figure}

Lemma \ref{lem_half_space_coeff} allows to extend the lower bound of Proposition \ref{lem_gap_coef} to the boundary case under an almost identical set of assumptions with an additional cluster structure.
\begin{description}
\item[\label{assumptionsD} Assumptions C(\(r\)):]
\end{description}     
                        \begin{itemize}
                           \item First of all, we consider assumptions \hyperref[assumptionsA]{\(A(r,r)\)}
                           \item Additionally, we consider disjoint clusters \(\mathcal C_1, \dots, \mathcal C_{k_\mathcal C}\) of \(d_\infty\)-distance at least \(r+2r_\xi\) as submanifolds of \(\mathcal M\) with boundaries \(\partial \mathcal C_i\) of reach at least \(\frac{1}{\kappa'}\)
                           \item The density \(f\) on \(\mathcal M\) is constant with value \(c_0\) on \(\cup \mathcal C_i\) and satisfies
                            \[\esssup_{\mathcal M \setminus\ \cup \mathcal C_i} f \leq  c_0\]
                           \item Outside of the clusters we require the following regularity condition for the density: Any nontrivial intersection of a superlevel set of \(f\) with an offset \(\mathcal C_i^{r+2r_\xi}\) is equal to the intersection of that superlevel set with a submanifold of \(\mathcal M\) having a boundary of reach at least \(\frac{1}{\kappa'}\)
                           \item \(r \leq  \frac{1}{132\kappa' \sqrt{d+1}} \)
                          \end{itemize}
The last condition together with the upper bound from \hyperref[assumptionsA]{\(A(r,r)\)} ensures that both the reach of \(\mathcal M\) and \(\partial \mathcal C_i\) are large enough w.r.t. the radius \(r\), such that both the manifold and the boundary of the cluster can be locally approximated by affine subspaces.
\begin{prop}\label{lem_gap_coeff_boundary}
 We consider assumptions \hyperref[assumptionsD]{\(C(r)\)}. Suppose \(M_1, M_2 \in \mathcal C_i\) are points of distance at most \(br\). Then 
 \begin{align*}
   q_{\mathbb P} \geq q_d\left(s\right)  \left(1+\epsilon_{\mathcal M}\right)^{-1}\left(1+\epsilon_{\mathcal \xi}\right)^{-1} \left(1+\epsilon_{\partial \mathcal C}\right)^{-1},
 \end{align*}
for 
\begin{align*}
\epsilon_{\mathcal M} &=  \frac{45360 (d+1) \kappa^2 r^2}{\left(1-\frac{b'^2}{4}\right)^\frac{d+1}{2}} \\
\epsilon_{\xi} &= \frac{264 (d+1) \frac{r_\xi}{r}}{\left(1-\frac{b'^2}{4}\right)^\frac{d+1}{2}}  \\
\epsilon_{\partial \mathcal C} &=132  \kappa' r \sqrt{d+1} 
\end{align*}
\end{prop}
This inequality is stronger than the lower bound from Proposition \ref{lem_gap_coef}. Hence, we have to modify the definition of the \emph{adjusted volume coefficient}. For the following, we consider \[\mathfrak q_{ij}^{(k)}\defined (1+\epsilon_{\mathcal M})^{-1} (1+\epsilon_{\xi})^{-1}  \left(1+\epsilon_{\partial \mathcal C}\right)^{-1} q_d\left(\frac{\|X_i-X_j\|}{h_{k-1}}\right)\]
to allow for consistent propagation in the boundary case as stated in the following Theorem. Again, in practice, the implementation of the adjusted volume coefficient might be ignored, c.f. \cite{AWC}. However, it is important to not underestimate the dimension parameter $d$. In fact, an overestimation of $d$ might compensate for dropping the first three factors of the adjusted volume coefficient and ensure the propagation of homogeneous areas.
\begin{thm}\label{thm_propagation_boundary}
We consider a distribution in the vicinity of a manifold and two points inside a homogenous cluster. Then with high probability, the AWC algorithm will not detect a gap between them, even if the points happen to be in close proximity to the boundary of the cluster. To be precise, suppose assumptions \hyperref[assumptionsD]{C(\( h_{k-1} \))} hold and \( X_1, X_2, \dots, X_{n+2} \overset{i.i.d.} {\sim}\mathbb P \). We assume that the AWC algorithm did not detect any gaps in the previous step. If we choose the threshold \( \lambda = C \log n \) for some \( C>0 \), then \[\mathbb P^{\otimes (n+2)} \left(T_{ij}^{(k)}> C\log n \middle\vert X_i, X_j\in \mathcal C^{r_\xi},  \|X_i - X_j\| \leq h_{k}\right) \leq 2n^{–C}.\]
\end{thm}
Together with Theorem \ref{thm_separation} we are able to cover all the discussed cases at once. In the following corollary, we will use the term \emph{global clusters} to describe the disjoint offsets \(\mathcal C_i^{r_\xi}\).
\begin{cor}\label{cor_prop_and_sep}
We consider the conditions \hyperref[assumptionsD]{\(C(h_{k-1})\)} and \hyperref[assumptionsB]{\(B(h_{k-1})\)} with a slightly stricter lower bound \[\varepsilon\geq 7\left(1+\epsilon_{\mathcal M}\right)\left(1+\epsilon_{\mathcal \xi}\right) \left(1+\epsilon_{\partial \mathcal C}\right) - 7.\]
 Suppose \( X_1, \dots, X_{n+2} \overset{i.i.d.} {\sim} \mathbb P\). We assume that the AWC algorithm did not detect any gaps in the previous step. Moreover, we choose the threshold \(\lambda = \frac{\alpha}{4}\log n\). Then with probability at least \(1-3n^{-\frac{\alpha-8}{4}}\), every local cluster \(\mathcal C_i^{(k)}\) calculated by AWC at step \(k\) satisfies the following: If \(X_i\) belongs to a global cluster, \(\mathcal C_i^{(k)}\) contains all points from this cluster of distance at most \(h_k\) to \(X_i\), while it does not contain any points from other global clusters.
\end{cor}

\subsection{Optimality}
The lack of a rigorous global cluster objective makes it difficult to compare our theoretical results to previous work. Moreover, we have shown that the algorithm differs significantly from other density-based methods, c.f. Figure \ref{fig_counterexamples_density_tree}. However, the local separation considered in Theorem \ref{thm_separation} as well as Corollary \ref{cor_prop_and_sep} is very similar to the split of two components in the cluster density tree. Consistent and rate-optimal estimation of the cluster density tree using a single-linkage clustering algorithm has been established in\cite{separation_condition_clusters_for_tree}. Using different notation (i.e. \(\sigma\) instead of \(r\) as width of the gap and \(\lambda\) instead of \(f_0\) as density level), the authors show that the optimal rate is (up to logarithmic factors and factors dependent on \(d\)) given by
\[\epsilon \gtrsim \sqrt{\frac{1}{n  r^d f_0}}\]
In \cite{belakrishnan_cluster_tree_on_manifold} this has been extended to the manifold setup. Further work by \cite{cluster_tree_optimal_rate_hoelder} shows that, under the assumption of a Hölder smooth density, this rate can be described by only one separation parameter together with the smoothness parameter. 

In view of \(z_k \propto f_0 r^d\), our lower bound on the depth of the gap 
\[\epsilon \geq \sqrt{\frac{2 \alpha \log n}{z_k q_d^2(b) n}}\]
achieves in fact the optimal rate given above w.r.t. \((n, r, f_0)\). We verify the optimality w.r.t. n for our setup under very simple conditions, showing that no algorithm can consistently detect the gap if \( \epsilon \) decreases at the rate \( n^{-\frac{1}{2}} \). 
\begin{description}
\item[\label{assumptionsC} Assumptions D:]
\end{description}     
\begin{itemize}
 \item \( \mathcal C_1, \dots, \mathcal C_k \) are disjoint subsets of a manifold \( \mathcal M\subset \mathbb R^D \) 
 \item \( X_1, \dots, X_n \) are drawn i.i.d. from a density supported on \( \mathcal M \) that is constant on \( V \defined \cup \mathcal C_i  \) with value \( f_V \) and constant on \( G \defined \mathcal M \setminus V \) with value \( f_G \) 
\end{itemize}

\begin{thm}\label{thm_rate}
Let assumptions \hyperref[assumptionsC]{D} be satisfied. We consider the null hypothesis of a uniform distribution on the manifold, i.e.  \[H_0: f_G=f_V\]
against the alternative \[H_1: f_G=(1-\delta) f_V\]
 for \( \delta > 0 \). Then no test can separate the two cases consistently if \( n\delta^2\nrightarrow\infty \) as \( n\to\infty \).
\end{thm}

\section{Experimental Results} \label{section_experiments}
Although manifold models are considered to be realistic, we still impose some assumptions for our theoretical study that are usually not satisfied in real-life. Most importantly we assume that our data lies on a manifold without boundary and positive reach up to bounded noise. A comprehensive numerical study of the procedure including real-life data by \cite{AWC} suggested that these assumptions are not necessary in practice and the performance of the algorithm is competitive with state-of-the-art algorithms. Rather, the limiting factor of the algorithm for clustering so-called big data at a global scale seems to be its polynomial complexity. That being said, in this work, we will restrict to some rather simple artificial examples in order to illustrate and verify our theoretical results.

\subsection{Consistency}\label{subsection_consistency}
In order to verify the sensitivity of the AWC algorithm w.r.t. local gaps for data lying on non-linear submanifolds and illustrate the main results Theorem \ref{thm_propagation} and Theorem \ref{thm_separation}, we will start by studying an artificial example where the embedding dimension is equal to \( 2 \) and the intrinsic dimension of the data is \( 1 \). We consider a distribution on the vacinity of the unit circle \( S^1 \) in \( \mathbb R^2 \) with two clusters 
\[\mathcal C_1 \defined \{(x, y) \in S^1: y > \frac{1}{4}\}\]
and 
\[\mathcal C_2 \defined \{(x, y) \in S^1: y < -\frac{1}{4}\}.\]
By \( \mathbb P_\epsilon \) we denote the distribution corresponding to the density
\[f_\epsilon\defined  \frac{1}{2\pi}\left(\mathds 1_{\mathcal C_1 \cup \mathcal C_2} + (1-\epsilon) \mathds 1_{S^1 \setminus (\mathcal C_1 \cup \mathcal C_2)}\right).\]
Moreover, by \( \mathcal U(r) \) we denote the uniform distribution on a 2-dimensional ball of radius \( r \). Then we sample \( X_1, \dots, X_n \) i.i.d. from
\[\mathbb P_\epsilon^{\mathcal U\left(\frac{1}{10}\right)} \defined \mathbb P_\epsilon \ast \mathcal U\left(\frac{1}{10}\right),\]
cf. Figure \ref{fig_exp}.
\begin{figure}[h]
 \begin{minipage}[c]{0.32\textwidth}
\centering
\includegraphics[scale = 0.18]{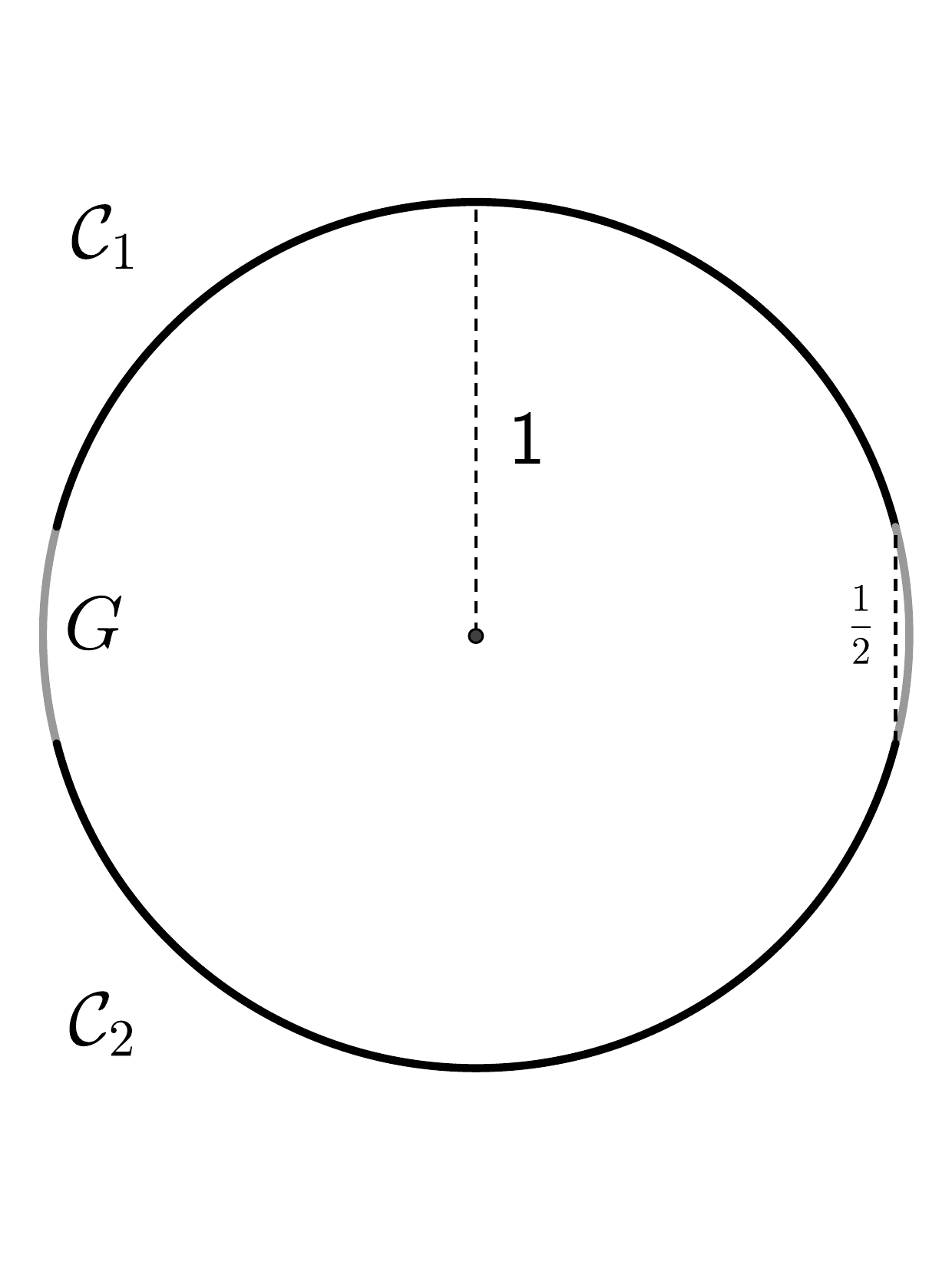}
\end{minipage}
 \begin{minipage}[c]{0.32\textwidth}
\centering
\includegraphics[scale = 0.313]{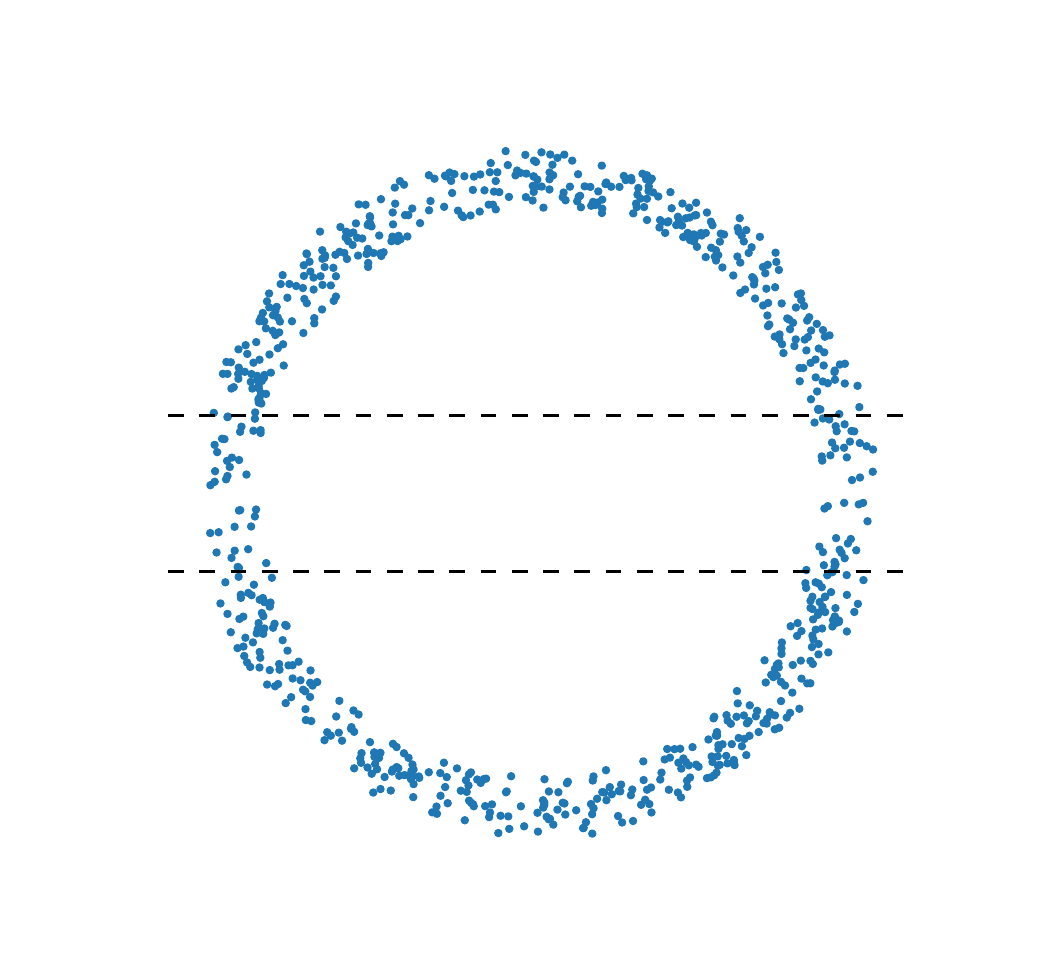}
\end{minipage}
 \begin{minipage}[c]{0.32\textwidth}
\centering
\includegraphics[scale = 0.313]{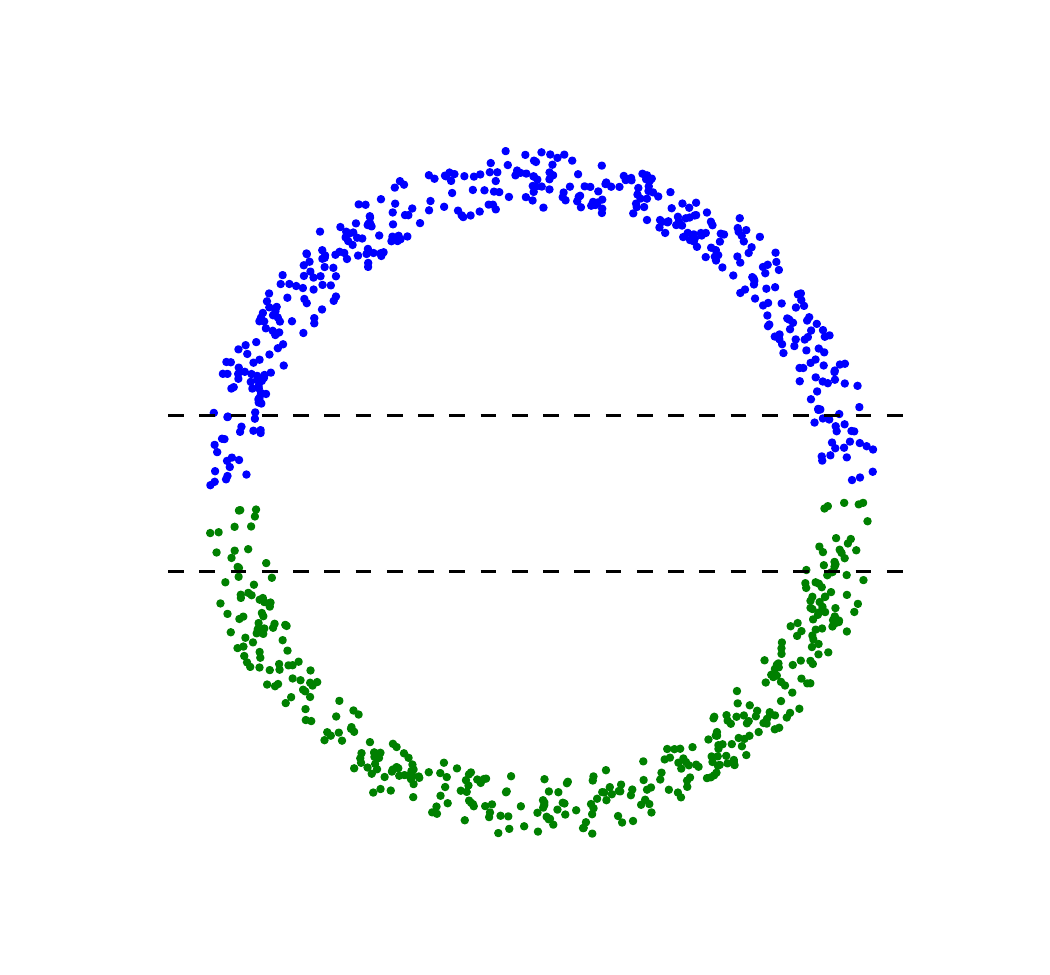}
\end{minipage}
\caption{Density \( f_\epsilon \) (left), i.i.d. sample of size \( n=800 \) from \( \mathbb P_\frac{1}{2}^{\mathcal U\left(\frac{1}{10}\right)} \) with two dashed lines highlighting the gap in the data (center) and clusters obtained via AWC (right)}
\label{fig_exp}
\end{figure}
To measure the performance of the algorithm we use a modified version of the Rand index \cite{rand}
\[\left(\sum_{\substack{(X_i, X_j) \in (\mathcal C_1\cup \mathcal C_2)^2 \\ 0 < ||X_i - X_j || < h_K}} 1\right)^{-1}\left(\sum_{\substack{X_i, X_j \in \mathcal C_1\\  X_i, X_j \in \mathcal C_2\\0 < ||X_i - X_j || < h_K}}w^{(K)}_{ij} +  \sum_{\substack{X_i\in  \mathcal C_1, X_j\in  \mathcal C_2 \\X_i\in  \mathcal C_2, X_j \in  \mathcal C_1\\ ||X_i - X_j || < h_K}}\left(1 - w^{(K)}_{ij} \right)\right) . \]
For simplicity, we refer to this measure as Rand index. It can also be defined as the accuracy of a subset of the weights \( (w_{ij}^{(K)})_{i, j = 1}^n \). As our theoretical results only apply at a local scale, we also restrict here to a local scale \( h_K = 1 \) and fix a series of bandwidths \( h_i = 2^{\frac{i}{2} - 2} \), \(  i = 0, \dots, 4 \). We only adjust the gap coefficient with respect to the intrinsic dimension, that is, we assume the reach and the noise magnitude to be zero in the computation of the adjusted volume coefficient. For each sample, we run the algorithm for different \( \lambda \) and consider only the best resulting Rand index, i.e. we overfit \( \lambda \). Finally, for different values of \( \epsilon \), we repeat the experiment 100 times.
\begin{figure}[t]
\centering
\vspace{-5pt}
\includegraphics[scale = 0.52]{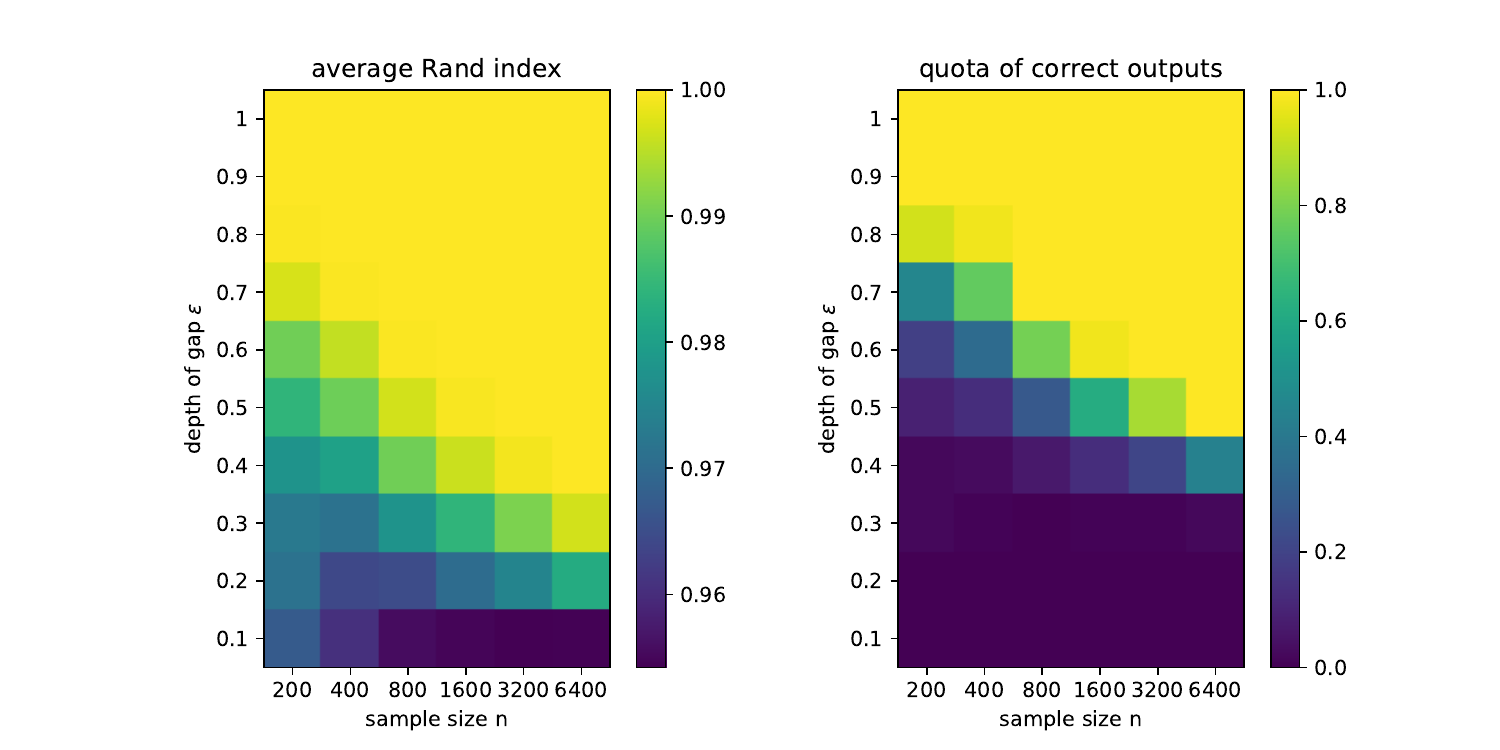}
\caption{Average rand index (left) and quota of experiments yielding a rand index \( 1 \) (right)}
\label{fig_Rand}
\end{figure}
The resulting average rand index is plotted in Figure \ref{fig_Rand} on the left. 
Note that the Rand index is in general quite close to 1, however, this is only due to the imbalance in the considered classification problem. For the evaluation of the results, we are only interested in the relatively large values, e.g. \( \geq 0.99 \). On the right, the quota of experiments is plotted where a rand index of 1 is achieved. This relates to our theoretical results, whereas the average rand index is a more common measure in practice. Our theoretical results show, that the minimal \( \epsilon \), for which we can reconstruct the cluster structure with high probability, is up to logarithmic factors of order \( \sqrt{\frac{1}{n}} \). The experiment is not exhaustive enough to verify this result. However, the results verify the asymptotics \( \epsilon \xrightarrow{n\to\infty} 0 \) and indicate that \( \epsilon \) decreases significantly slower than \( \frac{1}{n} \).

A less expected detail in the plot is the fact, that for small values of the depth \( \epsilon \), we observe better Rand indices as the sample size \( n \) decreases. This can be explained as follows. If \( \epsilon \) is small, our distribution is very close to a distribution without a gap. Thus, for large \( n \), the empirical distribution will also be close to a uniform distribution, and it will be very difficult for the algorithm to detect the clusters. However, for small \( n \), the distribution may deviate more from the uniform distribution and form random clusters that in some cases do accidentally have similarities to the true cluster structure.

\subsection{Scaling of sensitivity parameter \(\lambda\)}

In the experiment above, we also computed for each experiment the minimal value of \( \lambda \) that achieved the largest rand index and plotted the resulting average in Figure \ref{fig_lambda}. The results support our proposition that \( \lambda \) should be scaled logarithmically w.r.t. the data size.
\begin{figure}[h]
\vspace{0pt}
\centering
\vspace{-5pt}
\includegraphics[scale = 0.5]{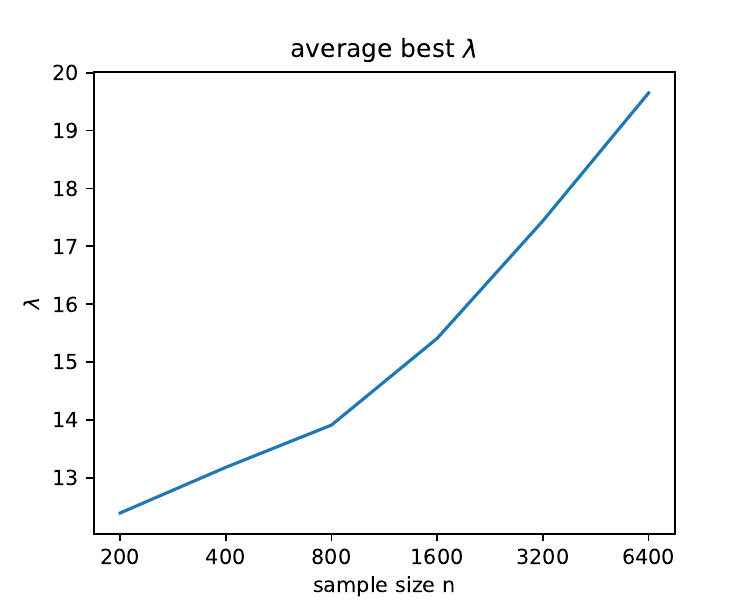}
\caption{Average minimal lambda with best rand index for \( \epsilon = 0.9 \) }
\label{fig_lambda}
\end{figure}

\subsection{High-dimensional data}
\begin{figure}[t]
 \centering
 \includegraphics[scale = 0.55]{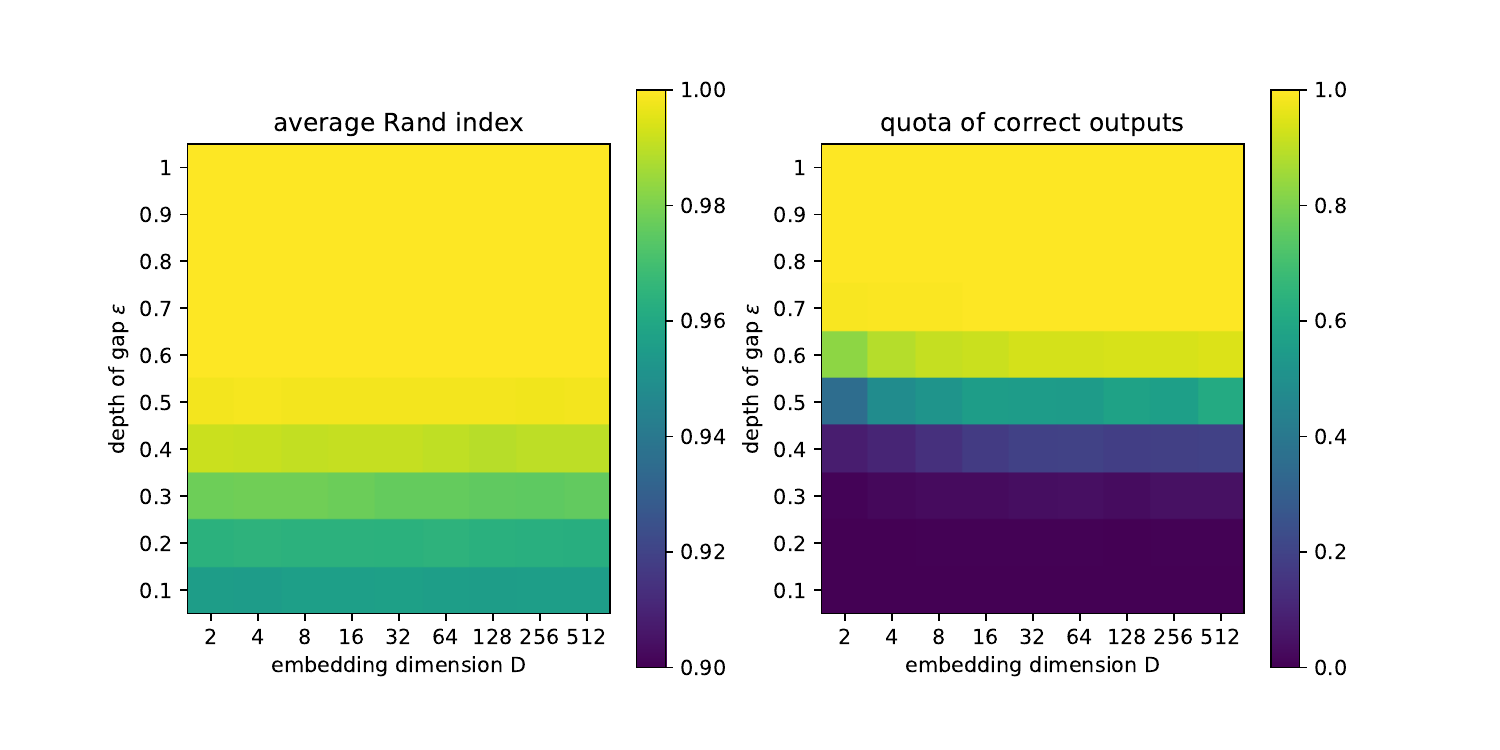}
 \caption{Average rand index (left) and quota of experiments yielding a rand index \( 1 \) (right) for uniform noise of norm \( \leq \frac{1}{10} \) }
 \label{fig_noise_uni}
\end{figure}

\begin{figure}[t]
 \centering
 \includegraphics[scale = 0.55]{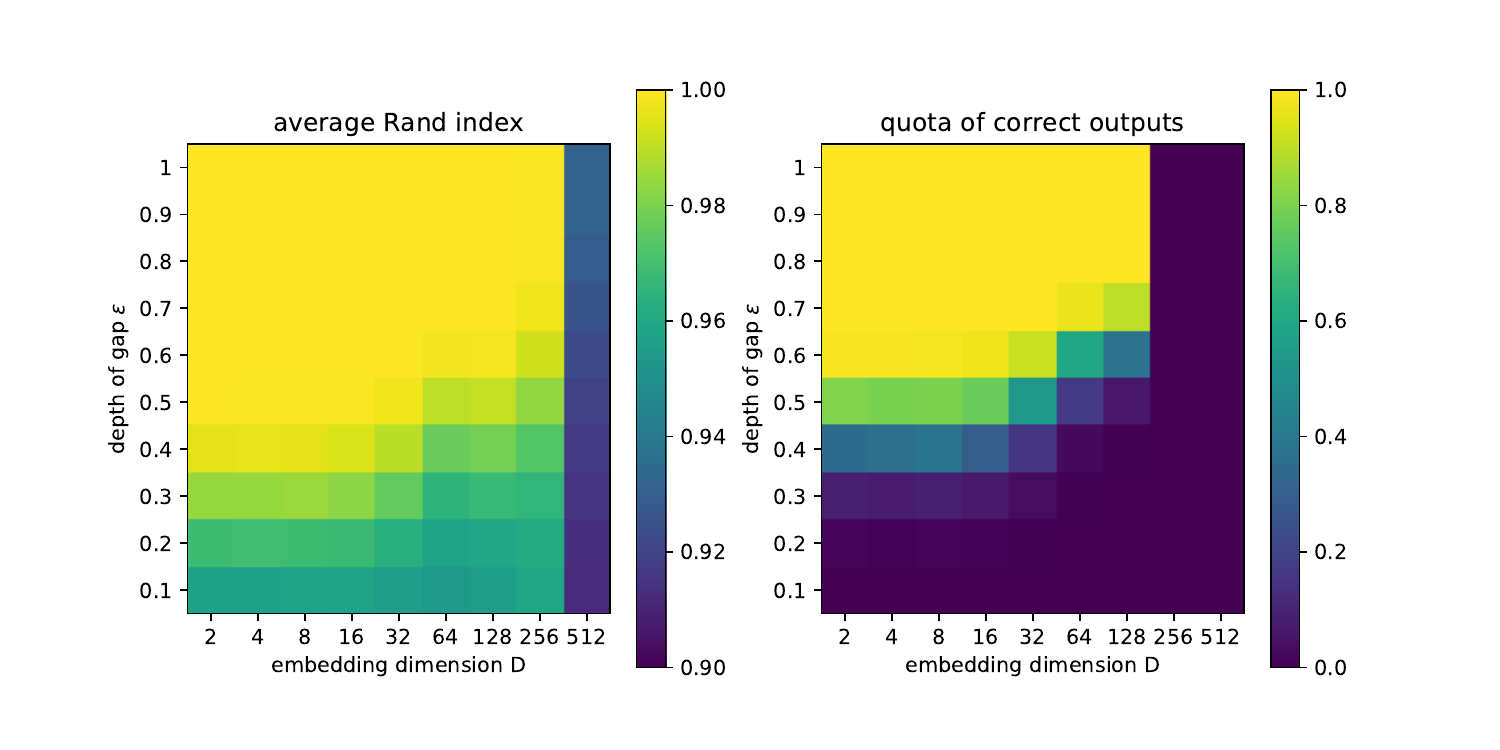}
 \caption{Average rand index (left) and quota of experiments yielding a rand index \( 1 \) (right) for Gaussian noise of variance \( \frac{1}{3200}I_D \) }
 \label{fig_noise_gauss}
\end{figure}

In this subsection, we study the effect of the embedding dimension, i.e. the effect of high-dimensional noise. Recall that the presented results are independent of the embedding dimension \( D \) of the data. However, as we assume the norm of the noise to be bounded. In the case of centered noise with i.i.d. coordinates this implies that for each coordinate the variance is of order \( \mathcal O(D^{-1}) \). This motivates the study of two different noise distributions. Firstly and corresponding to our theoretical results, we consider the uniform distribution \( \mathcal U(r) \) on a centered \( D \) -dimensional ball of radius \( r \). Also we want to consider the centered multivariate normal distribution \( \mathcal N(\sigma^2) \) with covariance matrix \( \sigma^2 I_D \). Note that for large \( D \), \( \mathcal N(\sigma^2) \) is concentrated on a thin annulus around the centered sphere of radius \( \sigma \sqrt{D} \), so the two noise distributions mainly differ in the parametrization of the scale.

By \( \mathbb P_{D, \epsilon} \) we denote an \( D \) -dimensional embedding of the distribution \( \mathbb P_\epsilon \) described in subsection \ref{subsection_consistency}. Then we draw our sample \( X_1, \dots, X_n \) i.i.d. either from
\[\mathbb P_{D, \epsilon}^{\mathcal U\left(\frac{1}{10}\right)} \defined \mathbb P_{D, \epsilon} \ast \mathcal U\left(\frac{1}{10}\right)\]
or
\[\mathbb P_{D, \epsilon}^{\mathcal N\left(\frac{1}{3200}\right)} \defined \mathbb P_{D, \epsilon} \ast \mathcal N\left(\frac{1}{3200}\right).\]
Note that the distribution \( \mathbb P_{\epsilon}^{\mathcal U\left(\frac{1}{10}\right)} \) used in the above experiments is a special case of \( \mathbb P_{D, \epsilon}^{\mathcal U\left(\frac{1}{10}\right)} \) for \( D=2 \). Moreover, for \( D=32 \), both distributions concentrate on the proximity of a centered sphere of radius \( \frac{1}{10} \). Thus we might expect similar performance of the algorithm for both distributions for \( D=32 \). According to our results, the performance should not break down in the uniform case for large \( D \) while we expect the performance to decrease with growing embedding dimension for the Gaussian noise as the noise radius increases.

We fix the sample size \( n=1000 \) and proceed otherwise analogously to the first experiment: For each sample, we optimize \( \lambda \) and repeat the experiment 1000 times for each value of \( \epsilon \). The resulting average rand indices, as well as the quota of experiments with rand index equal to 1, are presented in Figures \ref{fig_noise_uni} and \ref{fig_noise_gauss} and confirm our expectations. We observe one interesting detail in the quota of correct outputs in the presence of uniform noise on the right plot in Figure \ref{fig_noise_uni}. For a very small embedding dimension \( D \) the performance is slightly worse. A possible explanation is that the high-dimensional noise approximately preserves distances up to a constant summand with large probability. So in this experiment, the separation of the two clusters might be more difficult under smaller embedding dimension \( D \).

\subsection{Effect of intrinsic dimension parameter \(d\)} \label{subsection_experiment_intrinsic_dimension_effect}
Our theoretical results require knowledge of the parameter \(d\) of the effective dimension of the data. Otherwise, we cannot expect consistency under the asymptotics \(\epsilon \to 0\). In practical applications, the dimension parameter is often unknown and can be estimated \cite{dimension_estimation}. However, under the reasonable assumption that \(d\) is not too large, e.g. \(d\leq 5\), we can also just run the clustering procedure for the different values of \(d\). In both cases, uncertainty about the true intrinsic dimension remains. Unfortunately, our theoretical study does not provide much insight into the stability of the algorithm with respect to the dimension parameter.

In order to observe the effect of both under- and overstimation of the dimension parameter, we will consider the following simple 2-dimensional example. We consider a distribution on the unit sphere \( S^2 \) in \( \mathbb R^3 \) with two clusters 
\[\mathcal C_1 \defined \{(x, y, z) \in S^2: z > \frac{1}{4}\}\]
and 
\[\mathcal C_2 \defined \{(x, y, z) \in S^2: z < -\frac{1}{4}\}.\]
We sample \( X_1, \dots, X_n \) i.i.d. from the distribution \( \mathbb P_\epsilon \) corresponding to the density
\[f_\epsilon \propto \mathds 1_{\mathcal C_1 \cup \mathcal C_2} + (1-\epsilon) \mathds 1_{S^2 \setminus (\mathcal C_1 \cup \mathcal C_2)},\]
cf. Figure \ref{fig_sphere_densty}.
\begin{figure}
\begin{minipage}[c]{0.49\textwidth}
 \includegraphics[width = 225pt, height =212pt]{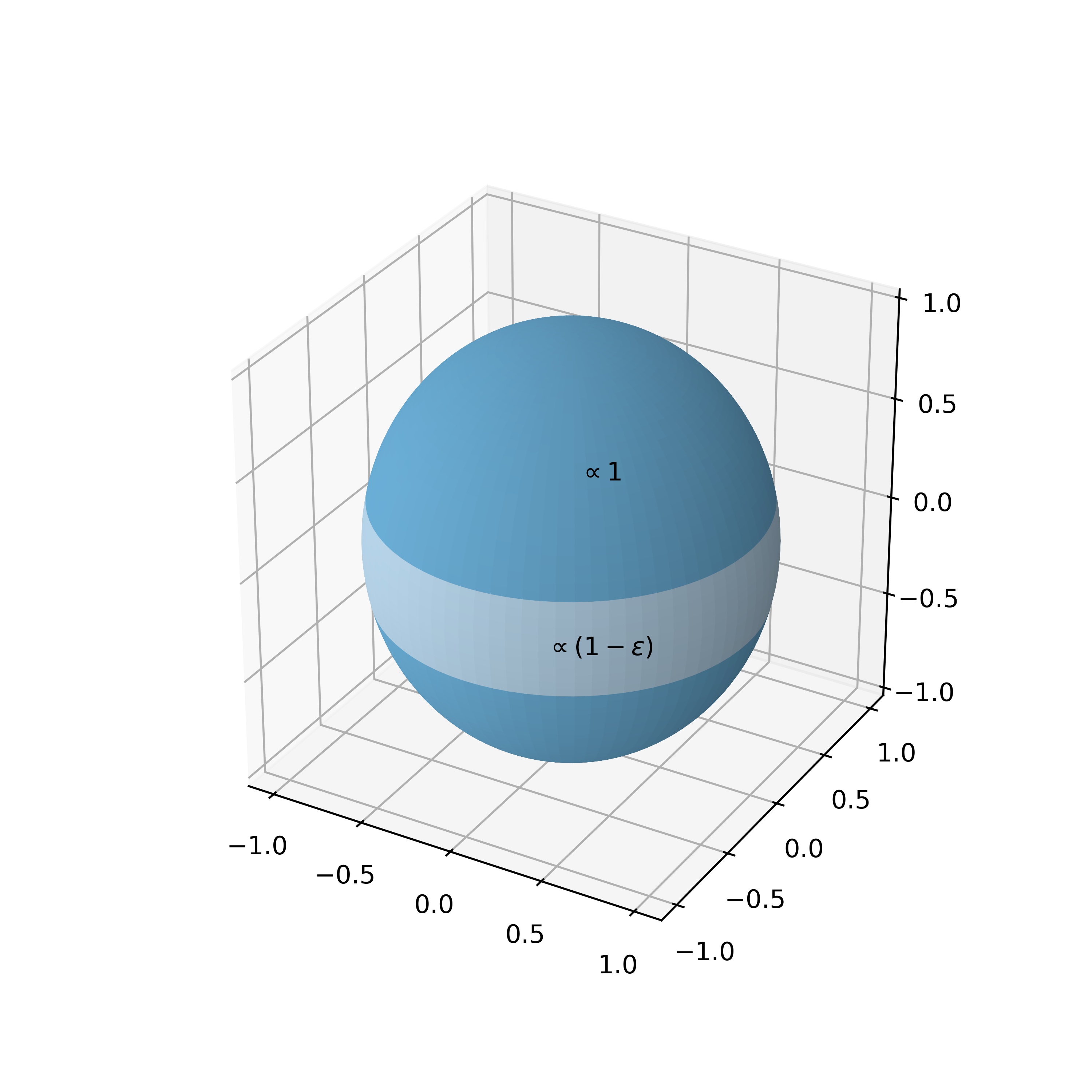}
\end{minipage}
\begin{minipage}[c]{0.49\textwidth}
 \includegraphics[width = 225pt, height =212pt]{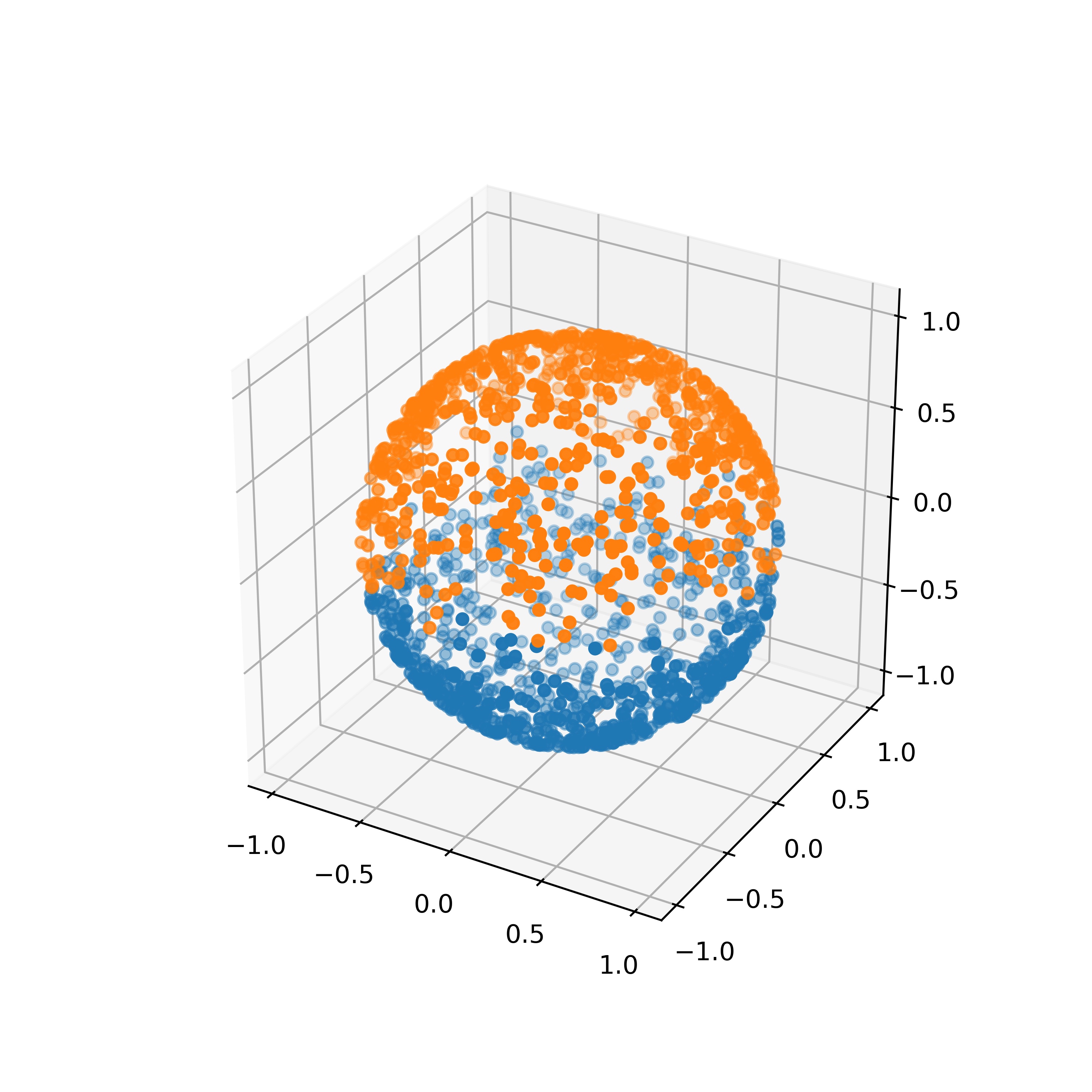}
\end{minipage}
\caption{Left: Sketch of density \(f_\epsilon\). Right: Obtained clustering from AWC with parameters \(d = 1\) and \(\lambda = 50\) for a sample of size \(n=1000\) and depth \(\epsilon = \frac{2}{3}\)}
\label{fig_sphere_densty}
\end{figure}
For a sample of size \(n=1000\) with depth \(\epsilon = \frac{2}{3}\), we consider various parameter \(\lambda\) and plot the corresponding sum of weight heuristic \(S\), i.e. the normalized sum of all weigths obtained at the final step of the AWC procedure. This statistic is a possible way to tune $\lambda$ in practice. One might simply take \(\lambda\) at a plateau of the graph of \(S\), as it is expected for a clear cluster structure that the output of the algorithm is stable with respect to the tuning parameter. The results are shown in figures \ref{fig_sphere_densty} and \ref{fig_sphere_sow}.\\
\begin{figure}
\begin{minipage}[c]{0.32\textwidth}
 \includegraphics[scale = 0.33]{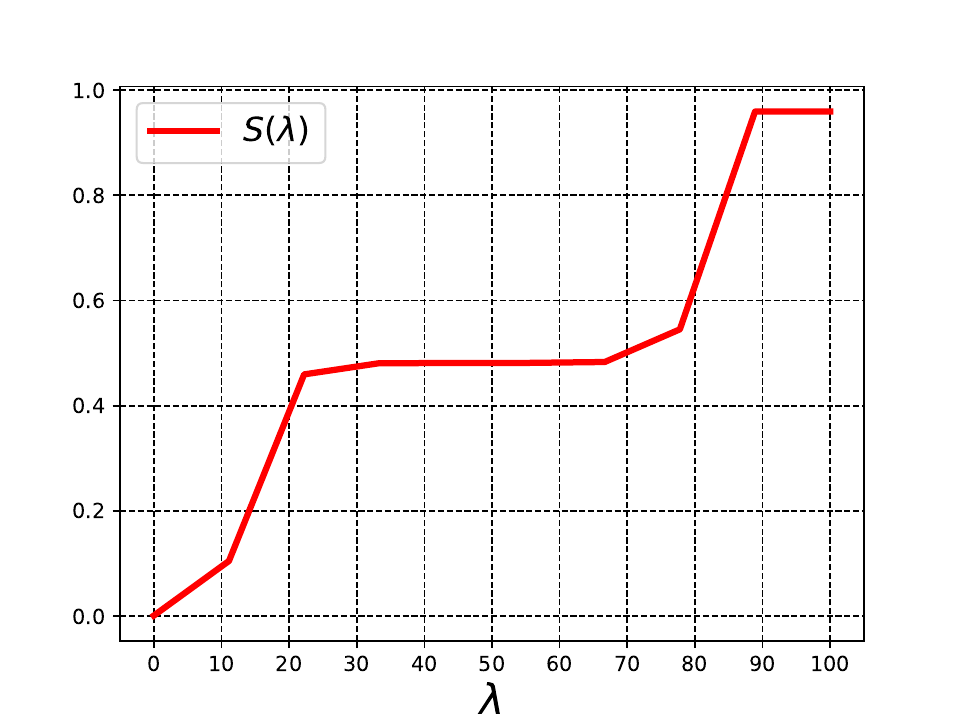}
\end{minipage}
\begin{minipage}[c]{0.32\textwidth}
 \includegraphics[scale = 0.33]{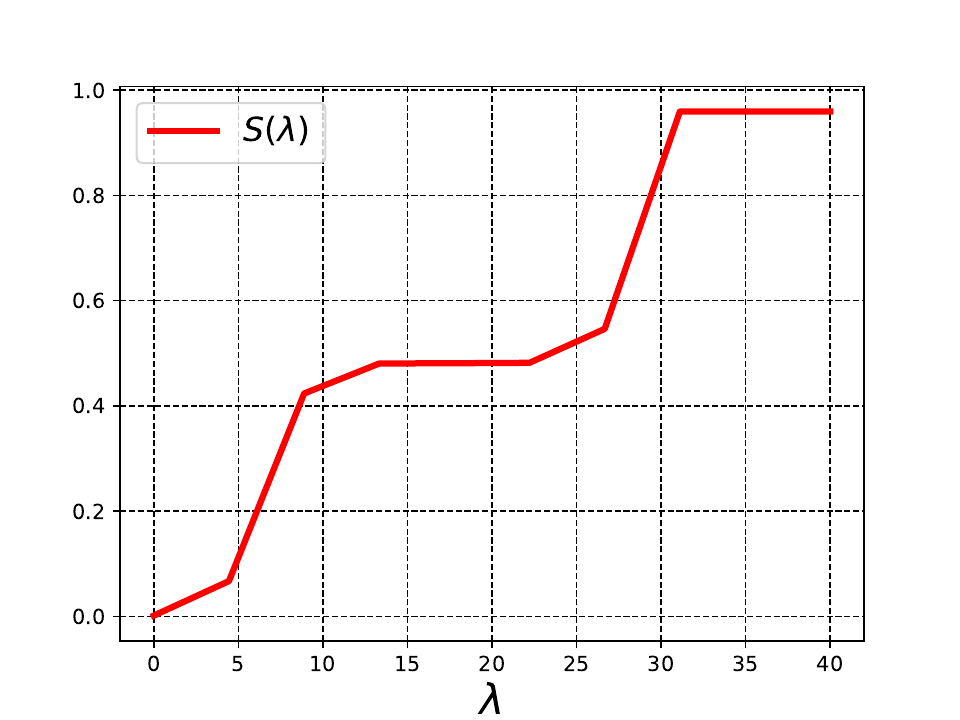}
\end{minipage}
\begin{minipage}[c]{0.32\textwidth}
 \includegraphics[scale = 0.33]{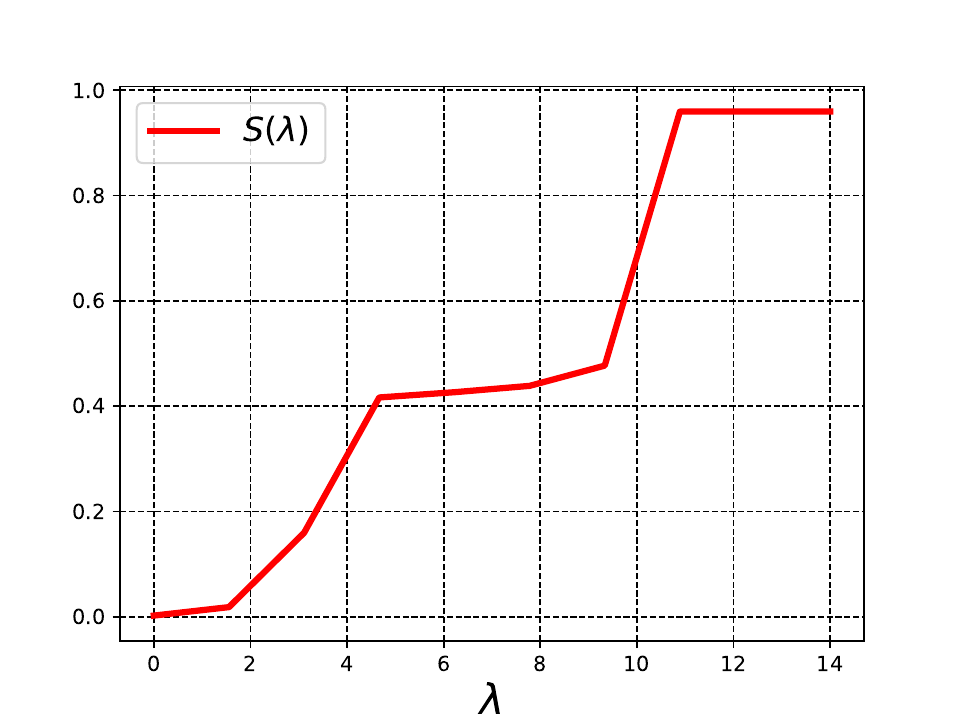}
\end{minipage}
\caption{Sum of weights heuristics for the sample in Figure \ref{fig_sphere_densty} with parameters \(d=1\) (left), \(d=2\) (middle) and \(d=3\) (right)}
\label{fig_sphere_sow}
\end{figure}
In figure \ref{fig_sphere_sow} we see for each dimension parameter \(d=1,2, 3\) a unique plateau at a value around \(0.5\). The value \(S(\lambda)=0.5\) corresponds to two clusters of equal size. Indeed a plot for the parameters \((d=1, \lambda = 50)\) in figure \ref{fig_sphere_densty} verifies that the cluster structure is detected as expected. We ommited plots for \((d=2, \lambda = 20)\) and \((d=3, \lambda = 8)\), as the results are nearly identical. Moreover, we observe that the scaling of \(\lambda\) depends on \(d\). A larger dimension parameter requires smaller \(\lambda\). This can be explained by the fact the the corresponding volume coefficient decreases with an increase of the dimension parameter. So it is harder for the algorithm to detect gaps, while the propagation effect is even stronger. A smaller \(\lambda\) compensates this effect.

The experiment suggests that the AWC procedure is able to detect the cluster structure even if the effective dimension parameter \(d\) is over- or underestimated. However, the scaling of \(\lambda\) depends on the choice of \(d\).

\section{Proofs} \label{section_proofs}

\begin{proof}[Proof of Proposition \ref{lem_exp_growth_coefficient}]
The main tool for the bounds will the series representation
\[\mathcal B\left(x, a, b\right) = x^a \sum\limits_{n=0}^\infty \frac{\Gamma(1 - b + n)}{\Gamma(1 - b)\Gamma(n + 1)(a + n)}x^n\]
for the incomplete beta function \cite{beta-function2}. Also, we use the logarithmic convexity of the gamma function.
For the upper bound we get
\begin{align}
q_d(t) &= \frac{\mathcal B\left(1-\frac{t^2}{4}, \frac{d+1}{2}, \frac{1}{2}\right)}{2\mathcal B\left(\frac{d+1}{2}, \frac{1}{2} \right) - \mathcal B\left(1-\frac{t^2}{4}, \frac{d+1}{2}, \frac{1}{2}\right)}\nonumber\\
& \leq \frac{\mathcal B\left(1-\frac{t^2}{4}, \frac{d+1}{2}, \frac{1}{2}\right)}{\mathcal B\left(\frac{d+1}{2}, \frac{1}{2} \right)}\nonumber\\
&\leq \frac{ \frac{2}{d+1} \sum\limits_{n=0}^{\infty}\left(1-\frac{t^2}{4}\right)^{\frac{d+1}{2}+n}}{\mathcal B\left(\frac{d+1}{2}, \frac{1}{2} \right)}\nonumber\\
&=\frac{\frac{2}{d+1} \left(1-\frac{t^2}{4}\right)^\frac{d+1}{2} \Gamma\left(\frac{d+2}{2}\right)}{\frac{t^2}{4} \Gamma\left(\frac{d+1}{2}\right)\Gamma\left(\frac{1}{2}\right)}\nonumber\\
&\leq\frac{\frac{2}{d+1} \left(1-\frac{t^2}{4}\right)^\frac{d+1}{2}\Gamma^\frac{1}{2}\left(\frac{d+3}{2}\right)}{\frac{t^2}{4} \Gamma^\frac{1}{2}\left(\frac{d+1}{2}\right)\Gamma\left(\frac{1}{2}\right)}\nonumber\\
&= 2^\frac{5}{2} t^{-2}\frac{\left(1-\frac{t^2}{4}\right)^\frac{d+1}{2}}{(d+1)^\frac{1}{2} \Gamma\left(\frac{1}{2}\right)}\nonumber
\end{align}
and similarly, we compute the lower bound
\begin{align}
 q_d(t) &\geq\frac{\mathcal B\left(1-\frac{t^2}{4}, \frac{d+1}{2}, \frac{1}{2}\right)}{2\mathcal B\left(\frac{d+1}{2}, \frac{1}{2} \right) }\nonumber\\
&\geq \frac{\left(1-\frac{t^2}{4}\right)^{\frac{d+1}{2}}}{(d+1)\mathcal B\left(\frac{d+1}{2}, \frac{1}{2} \right)}\nonumber\\
&=\frac{\left(1-\frac{t^2}{4}\right)^\frac{d+1}{2} \Gamma\left(\frac{d+2}{2}\right)}{(d+1) \Gamma\left(\frac{d+1}{2}\right)\Gamma\left(\frac{1}{2}\right)}\nonumber\\
&\geq \frac{\left(1-\frac{t^2}{4}\right)^\frac{d+1}{2} \Gamma^{\frac{1}{2}}\left(\frac{d+2}{2}\right)}{(d+1) \Gamma^\frac{1}{2}\left(\frac{d}{2}\right)\Gamma\left(\frac{1}{2}\right)}\nonumber\\
&= \frac{d^\frac{1}{2} \left(1-\frac{t^2}{4}\right)^\frac{d+1}{2}}{2^{\frac{1}{2}}(d+1)\Gamma\left(\frac{1}{2}\right)}\nonumber\\
&\geq 2^{-1}\frac{\left(1-\frac{t^2}{4}\right)^\frac{d+1}{2}}{(d+1)^\frac{1}{2}\Gamma\left(\frac{1}{2}\right)}.\nonumber
\end{align}
\end{proof}

For the proof of Proposition \ref{lem_gap_coef} we will use the following two auxiliary Lemmas. By \( \vol(\cdot ) \) we denote the Lebesgue volume on a submanifold of \( \mathbb R^D \). We will consider different such manifolds and not specify them explicitly, as long as it clear from the context to which manifold we refer.

\begin{lem}\label{lem_volume_lipschitz}
For any \( d \)-dimensional \( C^2 \) submanifolds \( \mathcal M_1 \), \( \mathcal M_2\in\mathbb R^D \), a measurable subset \( A\subset \mathcal M_1 \) and a \( C \)-Lipschitz function \( f:\mathcal M_1\to\mathcal M_2 \), we have
\[\vol(f(A))\leq C^d\vol(A).\]
\end{lem}

\begin{proof}
This inequality is also valid for the \( d \) -dimensional Hausdorff measure. In this case, it is a simple consequence of the definition of the Hausdorff measure \cite{castro}. As the Lebesgue measure is related by a constant factor \cite{measure_analysis}, it also holds for the Lebesgue measure.
\end{proof}

For the second auxiliary Lemma we consider a connected and compact \( C^2 \) submanifold \( \mathcal M\subset\mathbb R^D \) with reach \( \frac{1}{\kappa}>0 \) and without boundary. For some fixed \( x\in\mathcal M \) we denote the tangent plane of \( \mathcal M \) at \( x \) by \( \mathcal T \). Also, we consider the projection \( P: \mathbb R^D \rightarrow  \mathcal T \) associating each \( y\in \mathbb R^D \) with the closest point in \( \mathcal T \).

\begin{lem}\label{lem_local_lipschitz_param}
Suppose \( 0<r \leq \frac{1}{40\kappa} \). Then the restriction \(P\vert_{\mathcal M\cap B(x, r)}\) is a 1-Lipschitz injection and its image contains \( \mathcal T\cap B(x, r/L) \). Moreover, its inverse is \( L \)-Lipschitz for \[ L \defined 1+40\kappa^2 r^2 \leq 1+\kappa r.\]
\end{lem}

\begin{proof}
 This Lemma is given in \cite{castro} with some unspecified small enough constant instead of \( \frac{1}{40} \). Following the corresponding proof, it can be easily verified that this constant is indeed small enough.
\end{proof}

\begin{proof}[Proof of Proposition \ref{lem_gap_coef}] 
Let us denote the uniform measure on the manifold with \( \mu \). For \( i=1, 2 \), we choose a point \( M_i' \) on the manifold \( \mathcal M \) of distance at most \( r_\xi \) to \( M_i \). Because the Euclidean norm of the noise \( \xi \) is bounded by \( r_\xi \), we get
\begin{align}
 q_l &\defined  \frac{\int \mathds{1}_{B(M_1', r-2r_\xi)\cap B(M_2', r-2r_\xi)} d \mu}{\int \mathds{1}_{B(M_1', r+2r_\xi)\cup B(M_2', r+2r_\xi)} d \mu} \nonumber\\
 &\leq q_{\mathbb P}\nonumber\\
 &\leq  \frac{\int \mathds{1}_{B(M_1', r+2r_\xi)\cap B(M_2', r+2r_\xi)} d \mu}{\int \mathds{1}_{B(M_1', r-2r_\xi)\cup B(M_2', r-2r_\xi)} d \mu} \nonumber\\
 &\definedas q_u \label{ineq_def_ql_qu}
\end{align}
Let us denote by \( \cupcap \) one of the symbols \( \cap \) or \(  \cup \) and suppose \( r'\in\mathopen{[}r-2r_\xi, r+2r_\xi\mathclose{]} \). By \( P \) we denote the orthogonal projection onto the tangent plane \( \mathcal T \) of \( \mathcal M \) at \( M_1' \). Our assumptions ensure that a ball of radius \( 3r \) around \( M_1' \) contains both \( B(M_1', r') \) and \( B(M_2', r') \). Since the restriction \( P\vert_{\mathcal M\cap B(M_1', 3r)} \) is an injective 1-Lipschitz map with an \( L \) -Lipschitz inverse with \( L\defined 1+360\kappa^2 r^2 \), we conclude (cf. \cite{castro})
 \begin{equation}
L^{-d} \leq \frac{\vol(P(\mathcal M\cap (B(M_1', r')\cupcap B(M_2', r'))))}{\vol(\mathcal M\cap (B(M_1', r')\cupcap B(M_2', r')))}\leq 1. \label{ineq_vol_projection}
 \end{equation}
Moreover, the above Lipschitz constants imply 
\[\mathcal T\cap B\left(P(M_i'), \frac{r'}{L}\right)\subseteq P(\mathcal M\cap B(M_i', r')) \subseteq \mathcal T\cap B(P(M_i'), r')\] for \( i = 1, 2 \) and therefore
 \begin{align}
  1 &\leq \frac{\vol \left(\mathcal T \cap (B(P(M_1'), r') \cupcap B(P(M_2'), r'))\right)}{\vol\left(P(\mathcal M\cap (B(M_1', r') \cupcap B(M_2', r')))\right)}\nonumber\\
 &\leq \frac{\vol \left(\mathcal T \cap (B(P(M_1'), r') \cupcap B(P(M_2'), r'))\right)}{\vol \left(\mathcal T \cap (B(P(M_1'), \frac{r'}{L}) \cupcap B(P(M_2'), \frac{r'}{L}))\right)} \definedas q_{\cupcap, r'}. \label{ineq_def_q_box_r}
 \end{align}
Note also that according to our assumptions, any intersections encountered so far are nonempty.
From \eqref{ineq_vol_projection} and \eqref{ineq_def_q_box_r} we conclude
\begin{align}
&q_{\cupcap, r'}^{-1}  \vol \left(\mathcal T \cap (B(P(M_1'), r') \cupcap B(P(M_2'), r'))\right) \nonumber\\
&\qquad\leq \vol\left(P(\mathcal  M\cap (B(M_1', r') \cupcap B(M_2', r')))\right)\nonumber\\
&\qquad\leq  \vol\left(\mathcal M\cap (B(M_1', r')\cupcap B(M_2', r'))\right)\nonumber\\
 &\qquad\leq L^d \vol\left(P(\mathcal M\cap (B(M_1', r') \cupcap B(M_2', r')))\right)\nonumber\\  
  &\qquad\leq L^d \vol \left(\mathcal T \cap (B(P(M_1'), r') \cupcap B(P(M_2'), r'))\right) \nonumber
\end{align}
and obtain
\begin{equation}
q_{\cupcap, r'}^{-1}\leq \frac{ \vol\left(\mathcal M\cap (B(M_1', r')\cupcap B(M_2', r'))\right)}{\vol \left(\mathcal T \cap (B(P(M_1'), r') \cupcap B(P(M_2'), r'))\right)} \leq L^d. \label{ineq_m_vs_t}
\end{equation}
In particular, considering \( (\cupcap, r')=(\cap, r+2r_\xi) \) and \( (\cupcap, r')=(\cup, r-2r_\xi) \) in \eqref{ineq_m_vs_t}, we get
\begin{equation}
q_u \leq q_{\cup, r-2r_\xi}L^d q_\cup q_{r+2r_\xi}\text{,} \label{ineq_q_u}
\end{equation}
where \( q_{r'} \) is defined as
\begin{equation}
q_{r'} \defined \frac{\vol \left(\mathcal T \cap B(P(M_1'), r') \cap B(P(M_2'), r')\right)}{\vol \left(\mathcal T \cap (B(P(M_1'), r') \cup B(P(M_2'), r'))\right)}\nonumber
\end{equation}
for \( r'\in\mathopen{[}r-2r_\xi, r+2r_\xi\mathclose{]} \) and
\begin{equation}
q_\cup \defined \frac{\vol \left(\mathcal T \cap (B(P(M_1'), r+2r_\xi) \cup B(P(M_2'), r+2r_\xi))\right)}{\vol \left(\mathcal T \cap (B(P(M_1'), r-2r_\xi) \cup B(P(M_2'), r-2r_\xi))\right)}.\nonumber
\end{equation}
For the lower bound, we similarly obtain
\begin{equation}
q_l \geq q_{\cap, r-2r_\xi}^{-1}L^{-d} q_\cup^{-1} q_{r-2r_\xi}.\label{ineq_q_l}
\end{equation}
The quotient \( q_{r'} \) is exactly the volume coefficient defined in \eqref{def_vol_coef} in dimension \( d \) at \(  \frac{\|P(M_1') - P(M_2')\|}{r'} \). The derivative of \( q_d \) is given by
\begin{equation}
 q_d'(t) = -2\left(1-\frac{t^2}{4}\right)^{\frac{d-1}{2}} \frac{\mathcal B\left(\frac{d+1}{2}, \frac{1}{2}\right)}{\left(2\mathcal B\left(\frac{d+1}{2},  \frac{1}{2}\right) - \mathcal B\left(1-\frac{t^2}{4}, \frac{d+1}{2}, \frac{1}{2}\right)\right)^2}.\nonumber
\end{equation}
Its absolute value on \( \mathopen[0, 2\mathclose) \) is bounded from above by \( \frac{2}{\mathcal B\left(\frac{d+1}{2}, \frac{1}{2}\right)} \). For the following we define \( s\defined \frac{\|M_1 - M_2\|}{r} \).
Because \( q_d \) is a monotonely decreasing function on \( \mathopen[0, 2\mathclose) \) and  \[\|P(M_1')-P(M_2')\| - 2r_\xi\leq \|M_1-M_2\|\leq L \|P(M_1')-P(M_2')\| + 2r_\xi\text{,}\] we have
\begin{align}
q_{r+2r_\xi} 
&\leq q_d\left(\frac{\max\{0, \|M_1-M_2\|-2r_\xi\}}{L(r+2r_\xi)}\right)\nonumber\\
&\leq q_d(s) +  \frac{2}{\mathcal B\left(\frac{d+1}{2}, \frac{1}{2}\right)} \left(s - \frac{\|M_1-M_2\|-2r_\xi}{L(r+2r_\xi)}\right)\nonumber\\
&= q_d(s) + \frac{2}{\mathcal B\left(\frac{d+1}{2}, \frac{1}{2}\right)} \left(\frac{sr(L-1)}{L(r+2 r_\xi)} + \frac{2sr_\xi}{r+2r_\xi} + \frac{2r_\xi}{L(r+2r_\xi)}\right)\nonumber\\
&\leq q_d(s) +  \frac{1440\kappa^2 r^2}{\mathcal B\left(\frac{d+1}{2}, \frac{1}{2}\right) }  +  \frac{12\frac{r_\xi}{r}}{\mathcal B\left(\frac{d+1}{2}, \frac{1}{2}\right) }\nonumber\\
& \leq q_d(s)\left(1+\frac{1440\kappa^2r^2}{q_d(b')\mathcal B\left(\frac{d+1}{2}, \frac{1}{2}\right) }  \right)\left(1+\frac{12\frac{r_\xi}{r}}{q_d(b')\mathcal B\left(\frac{d+1}{2}, \frac{1}{2}\right) } \right).\label{ineq_q_r+2rxi}
\end{align}
Similarly, we obtain
\begin{align}
q_{r-2r_\xi} & \geq q_d\left(\frac{\|M_1-M_2\|+2r_\xi}{r-2r_\xi}\right)\nonumber\\
&= q_d(s) \left(\frac{q_d(s)}{q_d\left(\frac{\|M_1-M_2\|+2r_\xi}{r-2r_\xi}\right)}\right)^{-1}\nonumber\\
& \geq q_d(s) \left(\frac{q_d\left(\frac{\|M_1-M_2\|+2r_\xi}{r-2r_\xi}\right)+\frac{2}{\mathcal B\left(\frac{d+1}{2}, \frac{1}{2}\right)}\left(\frac{\|M_1-M_2\|+2r_\xi}{r-2r_\xi} - s\right)}{q_d\left(\frac{\|M_1-M_2\|+2r_\xi}{r-2r_\xi}\right)}\right)^{-1}\nonumber\\
&\geq q_d(s)\left(1+\frac{2\left(\frac{2sr_\xi}{r-2r_\xi} + \frac{2r_\xi}{r-2r_\xi}\right)}{q_d\left(b'\right)\mathcal B\left(\frac{d+1}{2}, \frac{1}{2}\right)}\right)^{-1}\nonumber\\
&\geq q_d(s)\left(1+\frac{12\frac{r_\xi}{r}}{q_d\left(b'\right)\mathcal B\left(\frac{d+1}{2}, \frac{1}{2}\right)}\right)^{-1}\label{ineq_q_r-2rxi}.
\end{align}
It remains to find upper bounds for \( q_\cup \), \( c_{\cup, r'} \) and \( q_{\cap, r'} \). Firstly, note that for \( x\in \mathcal T \), we have
\begin{equation}
 q_{\cup, r'} \leq \frac{\vol\left(\mathcal T\cap B(x, \frac{r'}{L})\right) + 2 \vol\left(\mathcal T\cap (B(x, r') \setminus B(x, \frac{r'}{L})\right)}{\vol\left(\mathcal T\cap B(x, \frac{r'}{L})\right)} = 2 L^d - 1.\label{ineq_q_union_r}
\end{equation}
Analogously, using \( (1+x)^d < 1+2xd \) for \( 0\leq x \leq \frac{1}{d} \), we find
\begin{align}
q_\cup&\leq  \left(2\left(\frac{r+2r_\xi}{r-2r_\xi}\right)^d-1\right)\nonumber\\
& \leq \left(2 \left(1+\frac{5r_\xi}{r}\right)^d -1\right)\nonumber\\
&\leq \left(1+\frac{20dr_\xi}{r}\right)\label{ineq_q_union}
\end{align}
and
\begin{equation}
L^d \left(2L^d - 1\right) \leq 1 + 2880d\kappa^2 r^2.\label{ineq_L}
\end{equation}
Moreover, for \( s' \defined \frac{\|P(M_1')-P(M_2')\|}{r'} \),
\begin{align}
q_{\cap, r'} &= q_{\cup, r'} \frac{q_d\left(s'\right)}{q_d\left(s'L\right)}\nonumber\\
&\leq (2L^d - 1) \frac{q_d(s'L) + s'(L-1)\frac{2}{\mathcal B\left(\frac{d+1}{2}, \frac{1}{2}\right)}}{q_d(s'L)}\nonumber\\
&\leq (2L^d - 1) \left(1 + \frac{1440 \kappa^2 r^2}{q_d(b') \mathcal B\left(\frac{d+1}{2}, \frac{1}{2}\right)}\right).\label{ineq_q_intersec_r}
\end{align}
Finally, we derive a tractable bound for \( \frac{1}{q_d(b')\mathcal B\left(\frac{d+1}{2}, \frac{1}{2}\right)} \). 
Using only the first term of the series \cite{beta-function2}
\[\mathcal B\left(x, a, b\right) = x^a \sum\limits_{n=0}^\infty \frac{\Gamma(1 - b + n)}{\Gamma(1 - b)\Gamma(n + 1)(a + n)}x^n\text{,}\]
we get
 \begin{align}
 \frac{1}{q_d(b')\mathcal B\left(\frac{d+1}{2}, \frac{1}{2}\right)} &= \frac{2\mathcal B\left(\frac{d+1}{2}, \frac{1}{2} \right) - \mathcal B\left(1-\left(\frac{b'}{2}\right)^2, \frac{d+1}{2}, \frac{1}{2}\right) }{\mathcal B\left(1-\left(\frac{b'}{2}\right)^2, \frac{d+1}{2}, \frac{1}{2}\right) \mathcal B\left(\frac{d+1}{2}, \frac{1}{2}\right)} \nonumber\\  \nonumber
 &\leq \frac{2}{\mathcal B\left(1-\left(\frac{b'}{2}\right)^2, \frac{d+1}{2}, \frac{1}{2}\right) } \nonumber\\
 & \leq \frac{d+1}{\left(1 - \left(\frac{b'}{2}\right)^2\right)^\frac{d+1}{2}}.  \label{ineq_q_times_beta}
 \end{align}
Finally, putting \eqref{ineq_def_ql_qu}, \eqref{ineq_q_u}, \eqref{ineq_q_l}, \eqref{ineq_q_r+2rxi}, \eqref{ineq_q_r-2rxi}, \eqref{ineq_q_union_r}, \eqref{ineq_q_union}, \eqref{ineq_L}, \eqref{ineq_q_intersec_r} and \eqref{ineq_q_times_beta} together, we obtain 
\begin{equation}
M^{-1} \leq \frac{q_{\mathbb P} }{q_d(s)} \leq M\nonumber
\end{equation}
for
\[M\defined\left(1+2880d\kappa^2  r^2\right)\left(1 + \frac{1440  (d + 1) \kappa^2r^2 }{\left(1-\left(\frac{b'}{2}\right)^2\right)^\frac{d+1}{2}}\right)\left(1+20\frac{dr_\xi}{r}\right)\left(1 + \frac{12 (d + 1) \frac{r_\xi}{r}}{\left(1-\left(\frac{b'}{2}\right)^2\right)^\frac{d+1}{2}}\right).\]
According to our assumptions, both \(2880d\kappa^2  r^2\) and \( \frac{20d r_\xi}{r} \) are not larger than \( 4 \). In particular, \( M \) is bounded from above by \( (1+\varepsilon_{\mathcal M})(1+\varepsilon_\xi) \).
\end{proof}

\begin{proof}[Proof of Theorem \ref{thm_propagation}]
Note that the proof of \cite[Theorem 3.1]{AWC} relies only on the inequality \( \theta_{ij}^{(k)} \geq q_{ij}^{(k)} \) for \( \|X_i-X_j\|\leq h_k \). However, this is ensured by Proposition \ref{lem_gap_coef} and the construction of the adjusted volume coefficient.
\end{proof}

\begin{proof}[Proof of Corollary \ref{cor_propagation}]
This is a simple consequence of Theorem \ref{thm_propagation} and the union bound.
\end{proof}

\begin{proof}[Proof of Theorem \ref{thm_separation}]
Suppose \(x_i, x_j\in\mathbb R^D\) are \(r_\xi\)-close to two different clusters and \(\|x_i-x_j\|\leq h_k\). To simplify notation, we will implizitely condition on \(X_i = x_i\) and \(X_j = x_j\) for the remainder of this proof. For \( l=i, j \) we choose a point \( X_l'\in\mathcal C_{k_l}\) for  \( k_i \neq k_j \) such that \(\|X_l'-X_l\|\leq r_\xi\). Our assumptions imply that the density \(f\) in the overlap \( B(X_i', h_{k-1} + 2r_\xi) \cap B(X_j',h_{k-1} + 2r_\xi)\cap \mathcal M \) is bounded from above by \( (1-\epsilon)f_0 \). Let us denote the uniform measure on the manifold by \( \mu \) and the distribution with gap and without noise by \( \mathbb P_\epsilon \). We conclude 
\begin{align*}
 \theta_{ij}^{(k)}& \leq \frac{\mathbb P_\epsilon (B(X_1', r+2r_\xi) \cap B(X_2', r+2r_\xi))}{\mathbb P_\epsilon(B(X_1', r-2r_\xi)\cup B(X_2', r-2r_\xi))}\\
&\leq \frac{(1-\epsilon)f_0 A}{(1-\epsilon)f_0B+\epsilon f_0 C}\\
&= \frac{A}{B} \left(1- \frac{\epsilon C}{(1-\epsilon)B + \epsilon C}\right)
\end{align*}
with 
\begin{align*} 
 &A= \mu(B(X_1', r+2r_\xi)\cap B(X_2', r+2r_\xi)),\\
& B= \mu(B(X_1', r-2r_\xi)\cup B(X_2', r-2r_\xi))\\
 \text{and }&C =\mu(B(X_1', r-2r_\xi)) + \mu(B(M_2', r-2r_\xi)).
\end{align*}
The factor \( \frac{A}{B} \) is bounded from above by \( (1 + \varepsilon_{\mathcal M})(1 + \varepsilon_\xi) q_{ij}^{(k)} \) as shown in the proof of Proposition \ref{lem_gap_coef}. Moreover, \( B<C \) implies that the second factor is bounded from above by \( 1-\epsilon \), providing the upper bound
\[ \theta_{ij}^{(k)} \leq (1-\epsilon) (1 + \varepsilon_{\mathcal M})(1 + \varepsilon_\xi) q_{ij}^{(k)}.\]
Monotonicity of \( q_d \) and the lower bound of the depth \( \epsilon \) of the gap lead to
\begin{align}
\mathfrak q_{ij}^{(k)} - \theta_{ij}^{(k)}& \geq  \left((1 + \varepsilon_{\mathcal M})^{-1}(1 + \varepsilon_\xi)^{-1} - (1-\epsilon)(1 + \varepsilon_{\mathcal M})(1 + \varepsilon_\xi)\right)q_d(b)\nonumber\\
&\geq \left(\left(1+\frac{\epsilon}{7}\right)^{-1}-\left(1-\epsilon\right)\left(1+\frac{\epsilon}{7}\right)\right)q_d(b)\nonumber\\
&\geq\epsilon \frac{q_d(b)}{\sqrt{2}}.\label{ineq_lowerbd_diff_q_theta}
\end{align}                                                                                                                                                                              
Using Pinsker's inequality, we get
\begin{equation} \mathcal K\left( \mathfrak q_{ij}^{(k)}, \theta_{ij}^{(k)}\right)  \geq \epsilon^2 q_d(b)^2.\label{ineq_lowerbd_KL}\end{equation}
As \( \frac{n}{\log n} \geq \frac{2\beta}{ z_k^2} \), we can choose some \( \delta> 0 \) satisfying the inequalities
\begin{align}
 2\delta^2n &\geq \beta \log n \label{ineq_delta_N_Hoeffding}\\
\text{and } \delta n& \leq \frac{z_k n}{2}.\label{ineq_delta_N_notHoeffding}
\end{align}
Note that \(z_k \leq \mathbb P\left(B(X_i, h_{k-1}) \cup B(X_j, h_{k-1})\right) \). Hoeffding's inequality implies in view of \eqref{ineq_delta_N_Hoeffding} 
\begin{equation}
 N_{i\lor j}^{(k)}\geq (z_k-\delta) n \nonumber
\end{equation}
with probability at least \( 1-n^{-\beta} \). This implies together with \eqref{ineq_delta_N_notHoeffding} 
\begin{equation}
 N_{i\lor j}^{(k)}\geq \frac{z_k n}{2} \label{ineq_Hoeffding}
\end{equation}
with probability at least \( 1-n^{-\beta} \). 
On the other hand, by \cite[Lemma 5.1]{AWC} we have 
\begin{equation}
 \mathcal K(\widetilde\theta_{ij}^{(k)}, \theta_{ij}^{(k)}) < \frac{\beta \log n}{N_{i\lor j}^{(k)}}\label{ineq_Lemma_AWC}
\end{equation}
 with probability at least \( 1-2n^{-\beta} \). By the union bound, there exists an event E of probability at least \( 1-3n^{-\beta} \) on which both \eqref{ineq_Hoeffding} and \eqref{ineq_Lemma_AWC} hold. In the following let us fix an outcome of the event E. Then \eqref{ineq_Hoeffding} and \eqref{ineq_Lemma_AWC}  imply
 \begin{equation}
 \mathcal K(\widetilde\theta_{ij}^{(k)}, \theta_{ij}^{(k)}) < \frac{2\beta \log n}{z_k n}\nonumber
 \end{equation}
The assumption \( \frac{\epsilon^2 n}{\log n } \geq 2\alpha z_k^{-1} q_d(b)^{-2} \), \( \alpha > \beta > 0 \), implies
\begin{equation}
\mathcal K(\widetilde\theta_{ij}^{(k)}, \theta_{ij}^{(k)}) < \frac{\beta}{\alpha} \epsilon^2 q_d(b)^2\label{ineq_upperbd_KL_widetilde}.
\end{equation}
Note that \eqref{ineq_lowerbd_diff_q_theta} implies in particular \( \mathfrak q_{ij}^{(k)} > \theta_{ij}^{(k)} \). Since the function \( \mathcal K(\cdot, \theta) \) is strictly monotone on the interval \( [\theta, 1) \) and considering \( \frac{\beta}{\alpha} < 1 \), we conclude from \eqref{ineq_lowerbd_KL} and \eqref{ineq_upperbd_KL_widetilde}
\begin{equation}
 \widetilde \theta_{ij}^{(k)} < \mathfrak q_{ij}^{(k)}.\label{ineq_significant_gapcoeff}
\end{equation}
The triangle inequality and Pinsker's inequality yield
\begin{align}
 |\widetilde\theta_{ij}^{(k)}- \mathfrak q_{ij}^{(k)}| & \geq  |\theta_{ij}^{(k)}-\mathfrak q_{ij}^{(k)}|-|\widetilde\theta_{ij}^{(k)}-\theta_{ij}^{(k)}|\nonumber\\
 &\geq \epsilon\frac{q_d(b)}{\sqrt{2}}-\sqrt{\frac{1}{2}\mathcal K(\widetilde\theta_{ij}^{(k)}, \theta_{ij}^{(k)})}\nonumber\\
 &\numgeq{\ref{ineq_upperbd_KL_widetilde}} \epsilon \frac{q_d(b)}{\sqrt{2}} \left(1-\sqrt{\frac{\beta}{\alpha}}\right)\label{ineq_lower_diff_widetildetheta}
\end{align}
From Pinsker's inequality and the assumption \( \frac{\epsilon^2 n}{\log n } \geq 2\alpha z_k^{-1} q_d(b)^{-2} \) we deduce
\begin{align}
\mathcal K(\widetilde\theta_{ij}^{(k)},\mathfrak q_{ij}^{(k)})&\geq 2\left(\widetilde\theta_{ij}^{(k)}- \mathfrak q_{ij}^{(k)}\right)^2\nonumber\\
&\numgeq{\ref{ineq_lower_diff_widetildetheta}} \epsilon^2q_d(b)^2 \left(1-\sqrt{\frac{\beta}{\alpha}}\right)^2 \nonumber\\
&\geq \frac{\log n}{z_k n}2\alpha\left(1-\sqrt{\frac{\beta}{\alpha}}\right)^2\nonumber\\
&\numgeq{\ref{ineq_Hoeffding}} \frac{\log n}{N_{i\lor j}^{(k)}}\left(\sqrt{\alpha}-\sqrt{\beta}\right)^2\label{ineq_lowerbd_KL_tildetheta_q}
\end{align}
Finally, putting together \eqref{ineq_significant_gapcoeff} and \eqref{ineq_lowerbd_KL_tildetheta_q}, we conclude that any outcome of the event \( E \) satisfies
\begin{align*}
 T_{ij}^{(k)} &= N_{i\lor j}^{(k)}\mathcal K(\widetilde\theta_{ij}^{(k)}, \mathfrak q_{ij}^{(k)})\{\mathds 1({\widetilde\theta_{ij}^{(k)}< \mathfrak q_{ij}^{(k)}})-\mathds 1({\widetilde\theta_{ij}^{(k)}\geq \mathfrak q_{ij}^{(k)}})\} \\
 & \geq \left(\sqrt{\alpha}-\sqrt{\beta}\right)^2 \log n.
\end{align*}
The choice of \(x_i\) and \(x_j\) is irrelevant for this result, so it is also valid in the unconditional form.
\end{proof}

\begin{proof}[Proof of Theorem \ref{thm_rate}]
Let us denote the value of the constant density under the null hypothesis by \( f_0 \) and the Kullback-Leibler divergence by \( \mathcal D_{\text{KL}}(\cdot, \cdot) \).
Using \( 1=f_G|G|+f_V|V| \), we compute
\begin{align*}
 f_V& =\frac{1}{|G|+|V| - \delta |G|} \text{ and}\\
 f_G&=\frac{1-\delta}{|G|+|V| - \delta|G|}.
\end{align*}
Additivity of the Kullback-Leibler divergence and \( f_0 = \frac{1}{|V|+|G|} \) yields
\begin{align*}
 n^{-1} \mathcal D_{\text{KL}} (\mathbb P_0, \mathbb P_1)&= f_0 |G| \log \frac{f_0}{f_G} + f_0 |V| \log \frac{f_0}{f_V}\\
 &=\log\left(1-\delta \frac{|G|}{|G|+|V|}\right) - \frac{|G|}{|G|+|V|} \log(1-\delta)\\
 &= \frac{\delta^2}{2} \frac{|G|}{|G|+|V|}\left(1+\frac{|G|}{|G|+|V|}\right) + o(\delta^2),
\end{align*}
the latter follows from the Taylor expansion. As \( \mathcal D_{\text{KL}} (\mathbb P_0, \mathbb P_1)\to\infty \) is a necessary condition for consistent testing \cite[Section 2.4.2]{tsybakov_nonparametric_statistics}, we deduce that no test is able to separate the two cases consistently provided that \( n\delta^2\nrightarrow\infty \) as \( n\to\infty \).
\end{proof}

Before we prove Lemma \ref{lem_plane_monotonicity}, let us introduce the so-called \emph{general volume coefficient}.
\begin{defn}
Suppose \(r_1, r_2>0\), \(D\in\mathbb Z_{>0}\), \(M_1 = (0, \dots, 0) \in\mathbb R^D\) and \(M_2 = (1, 0, \dots, 0)\in\mathbb R^D\). By \(\lambda_D\) we denote the \(D\)-dimensional Lebesgue measure and \(B_D(\cdot, \cdot)\) denotes an euclidean Ball in \(\mathbb R^D\) with given center and radius. We define the \emph{\(D\)-dimensional general volume coefficient} by
\begin{align*}
q_D\left(r_1, r_2\right)\defined \frac{\lambda_D\left(B_D\left(M_1, r_1\right)\cap B_D\left(M_2, r_2\right)\right)}{\lambda_D\left(B_D\left(M_1, r_1\right)\cup B_D\left(M_2, r_2\right)\right)}
\end{align*}
\end{defn}
\begin{lem}\label{lem_general_formula_general_volume_coef}
For \(M_1\neq M_2\in\mathbb R^D\) and \(r_1, r_2>0\) we have
  \[ \frac{\lambda_D(B_D(M_1, r_1) \cap B_D(M_2, r_2))}{\lambda_D(B_D(M_1, r_1) \cup B_D(M_2, r_2))} = q_D\left(\frac{r_1}{\|M_1 - M_2\|}, \frac{r_2}{\|M_1 - M_2\|}\right) \]
\end{lem}
\begin{proof}
 This follows from the invariance of the quotient of two \(D\)-dimensional volumes under rotation, translation and uniform scaling.
\end{proof}

  \begin{lem}\label{lem_exact_formula_general_volume_coeff} Suppose \(r_1, r_2>0\). Using the usual order of arguments we denote the regularized incomplete beta function by \(I_\cdot(\cdot, \cdot)\). Then
\[q_D(r_1, r_2) = \begin{cases}
                   0 & ,r_1 + r_2 \leq 1 \\
               \left(  \frac{r_j}{r_i} \right)^D&, r_i - r_j \geq 1 \\
               \frac{ r_1 + r_2 - 1}{r_1 + r_2 + 1}&,D=1 \text{ and }r_1+r_2> 1\text{ and }| r_1-r_2| < 1\\
                   \frac{V_D^{\text{cap}}(r_1, r_2)+ V_D^{\text{cap}}(r_2, r_1) }{V^{\text{ball}}_D(r_1) + V_D^{\text{ball}}(r_2)-V_D^{\text{cap}}(r_1, r_2)- V_D^{\text{cap}}(r_2, r_1) } &,\text{otherwise}\\
                  \end{cases}
 \]
with
\begin{align*}
  V_D^\text{ball} (r_i) &= 2r_i^D\\
   V_D^{\text{cap}}(r_i, r_j) & = \begin{cases}
   r_i^DI_{1 - \left(\frac{1+r_i^2 - r_j^2}{2r_i}\right)^2} \left(\frac{D+1}{2}, \frac{1}{2}\right) & ,r_j^2 - r_i^2 \leq 1\\
                                   2r_i^D - r_i^D I_{1 - \left(\frac{1+r_i^2 - r_j^2}{2r_i}\right)^2}\left(\frac{D+1}{2}, \frac{1}{2}\right) & ,r_j^2 - r_i^2 > 1
                                  \end{cases}
\end{align*}
\end{lem}

\begin{figure}[H]
 \centering
 \includegraphics[scale = 0.06]{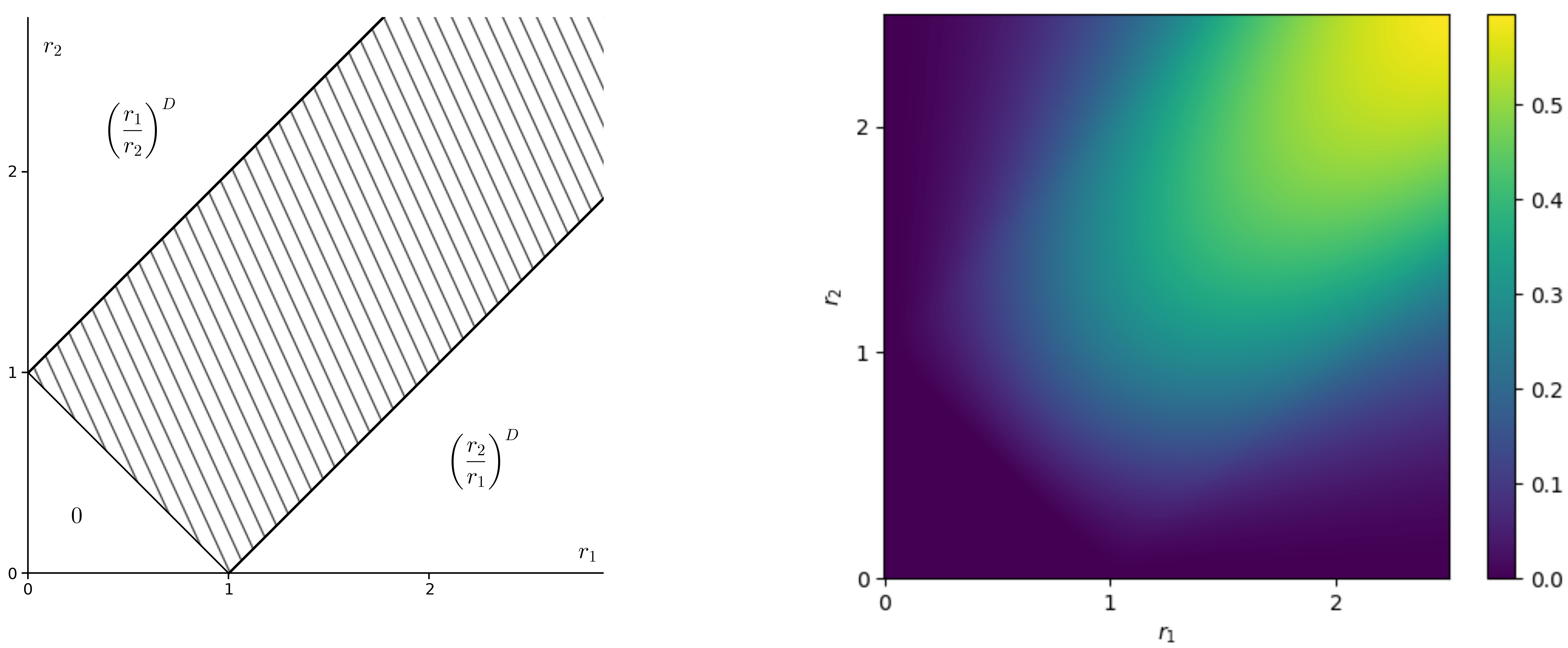}
 \caption{Left: Different regimes for formula of \(q_D(r_1, r_2)\) given in Lemma \ref{lem_exact_formula_general_volume_coeff}.\newline Right: Plot of \(q_2(r_1, r_2)\)}
\end{figure}

\begin{proof}
We only discuss the nontrivial regime where \(r_1+r_2 > 1\) and \(|r_2-r_1|<1\) for \(D>1\). Then the overlap of the two corresponding spheres with radii \(r_1\) and \(r_2\) around \(M_1=(0, \dots, 0)\) and \(M_2=(1, 0, \dots, 0)\) contains two points of the form \((x, \pm y, 0, \dots, 0)\). The coordinate equations of the two spheres yield
\begin{align*}
 x^2 + y^2 & = r_1^2\\
 (x-1)^2 + y^2 &=r_2^2
\end{align*}
implying
\begin{align*}
 x &= \frac{1+r_1^2 - r_2^2}{2}\\
 y&= \pm r_1\sqrt{1 - \left(\frac{1+r_1^2-r_2^2}{2r_1}\right)^2}
\end{align*}
We denote the smaller angle between the $x$-axis and the line through $M_1$ and $(x, y)$ by $\phi_1$. Analogously we define $\phi_2$, c.f. Figure \ref{fig_angle_cases}.
\begin{figure}[H]
                 \includegraphics[width=\linewidth]{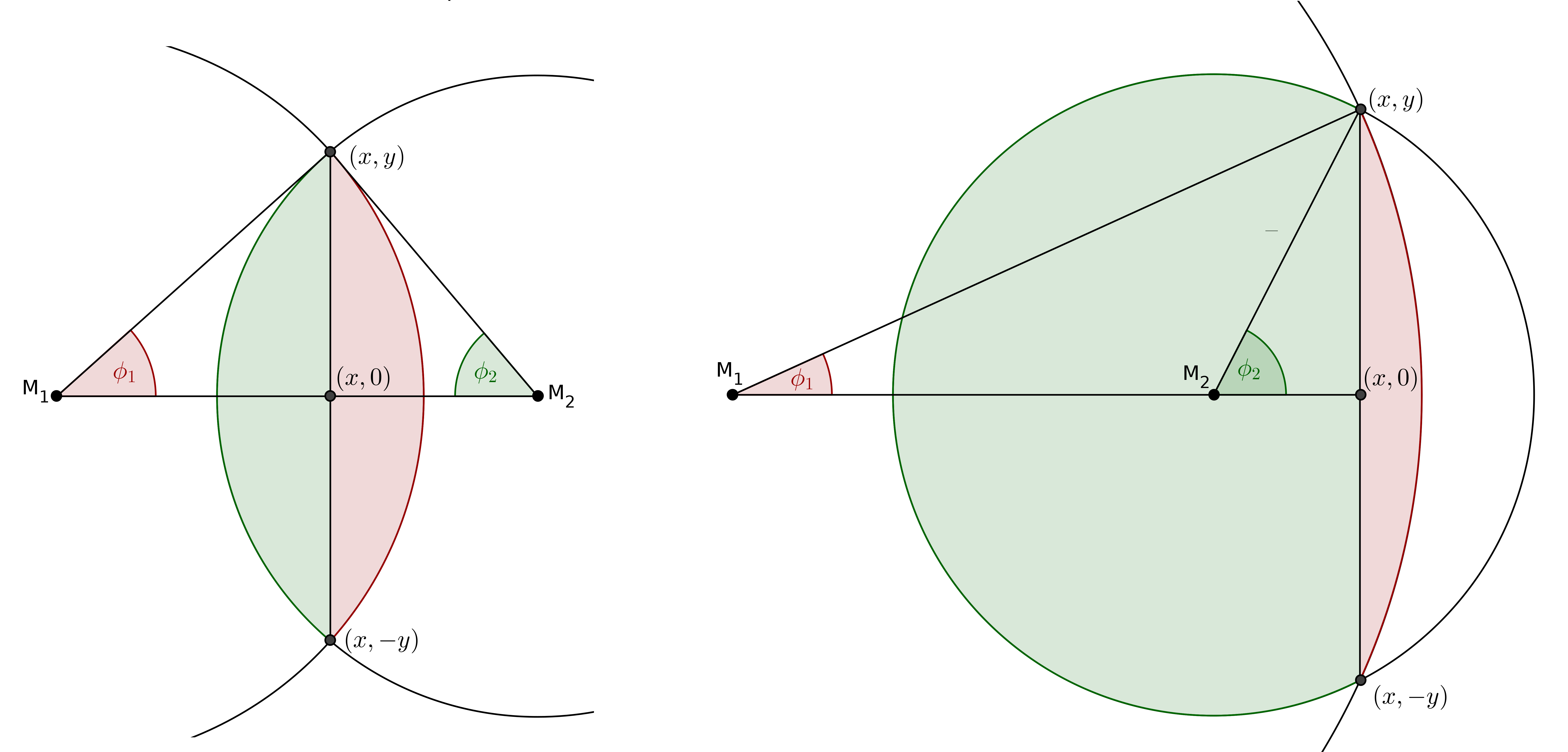}
 \caption{The volume of the overlap of two balls is the sum of the volumes of two caps that are shown in green and red. The corresponding angles used in the formulas of these volumes are highlighted in the same colour. On the left, we see the case $0<x<1$, whereas on the right $x>1$.} 
 \label{fig_angle_cases}
  \end{figure}
We conclude 
\begin{align*}
 \sin^2 \phi_1 &= \left(\frac{|y|}{r_1}\right)^2\\
 &= 1 - \left(\frac{1+r_1^2-r_2^2}{2r_1}\right)^2
\end{align*}
and
\begin{align*}
 \sin^2 \phi_2 & =  \left(\frac{|y|}{r_2}\right)^2\\
 &= \frac{r_1^2}{r_2^2} - \left(\frac{1+r_1^2-r_2^2}{2r_2}\right)^2\\
 & = 1 - \left(\frac{1+r_2^2-r_1^2}{2r_2}\right)^2
\end{align*}
Note that $x < 1$ is equivalent to $r_1^2 - r_2^2 < 1$ and $x > 0$ is equivalent to $r_2^2-r_1^2<1$. Using the formula for the volume of a hyperspherical cap given in \cite{formula_cap}, we conclude
\begin{align*}
 \lambda_D\left(B_D(M_1, r_1)\cap B_D(M_2, r_2)\right)& = \frac{\pi^{\frac{D}{2}}}{2\Gamma\left(\frac{D}{2}+1\right)}\left(V_D^\text{cap}(r_1, r_2) +V_D^\text{cap}(r_2, r_1)\right)\\
 \lambda_D\left(B_D(M_1, r_1)\cup B_D(M_2, r_2)\right)& = \frac{\pi^{\frac{D}{2}}}{2\Gamma \left(\frac{D}{2}+1\right)}\left( V^{\text{ball}}_D(r_1) + V_D^{\text{ball}}(r_2)-V_D^{\text{cap}}(r_1, r_2)- V_D^{\text{cap}}(r_2, r_1)\right)
\end{align*}
\end{proof}

\begin{proof}[Proof of Lemma \ref{lem_plane_monotonicity}]
Since $\mathcal Q$ is contineous and $\mathcal Q(t)=0$ for $t>t'_\text{max} \defined  \sup \{t : (H + tv) \cap B(M_1, r) \cap B(M_2, r)\neq \{\}\}$, we only need to dicuss the case $0<t< t'_\text{max}$. As $\mathcal Q(t)$ is contineous w.r.t. rotation of $H$ around the point $\frac{M_1+M_2}{2}$, we can w.l.o.g. assume that the vector $M_1-M_2$ is neither parallel nor orthogonal to $H$. Moreoever, we assume w.l.o.g. that there exists $d>0$ such that $M_1\in H + dv$. As a nontrivial intersection of a ball in $\mathbb R^{D+1}$ with a hyperplane is a $D$-dimensional ball, we can rewrite $\mathcal Q(t)$ using Lemma \ref{lem_general_formula_general_volume_coef}. The corresponding radii can be easily computed using Pythagoras' theorem, c.f. Figure \ref{fig_radii}. We get
 \begin{align*}
 \mathcal Q(t)&= q_D\left(r_1(t), r_2(t)\right)\\
  \text{with }r_1(t) & =\sqrt{\frac{1 - (t-d)^2}{\|M_1 - M_2\|^2 - 4d^2}} \\
 \text{and }r_2(t) &=\sqrt{\frac{1 - (t+d)^2}{\|M_1 - M_2\|^2 - 4d^2}}
 \end{align*}
  \begin{figure}[H]
  \includegraphics[scale =0.38]{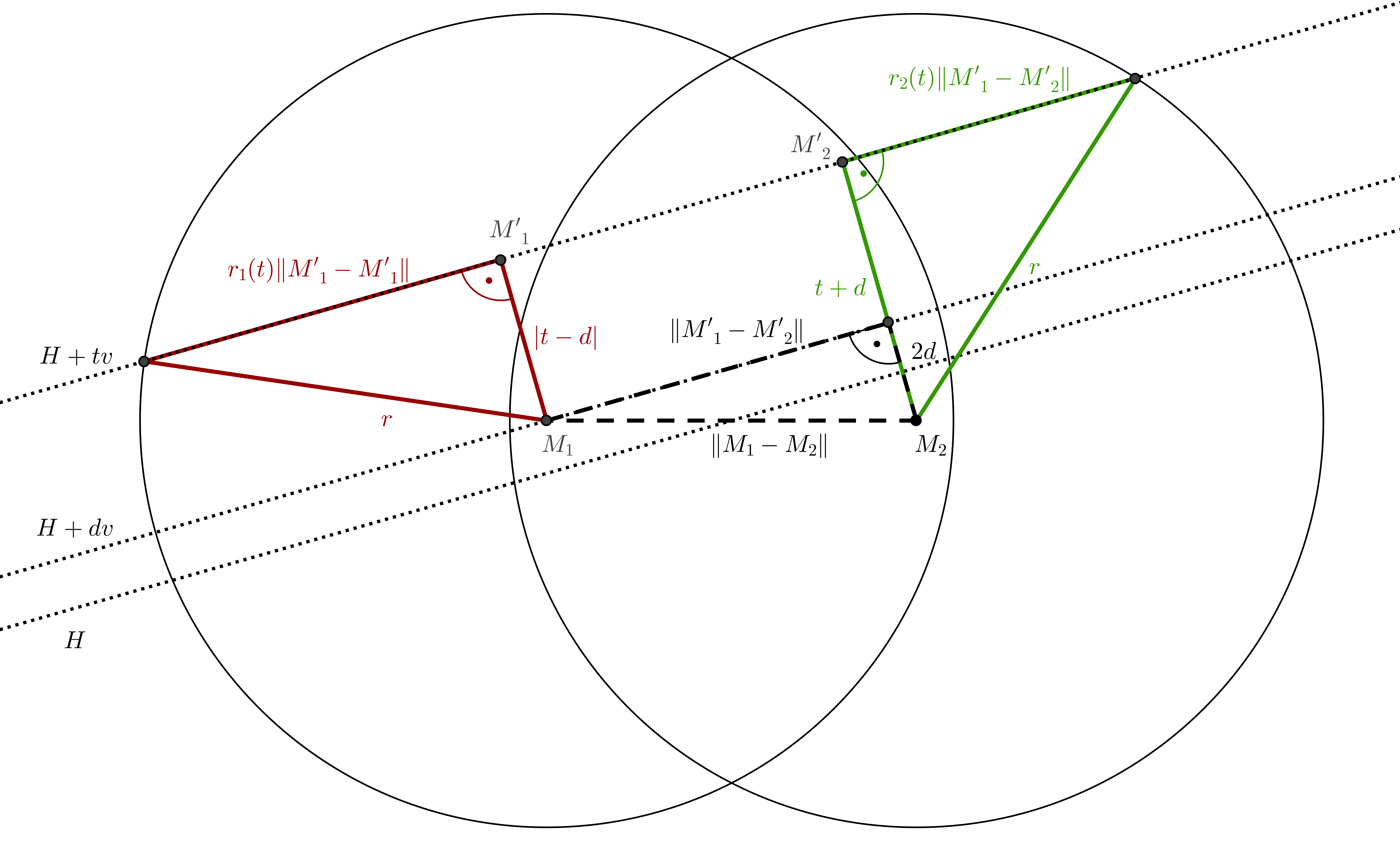}
         \caption{We denote $M_1+(t-d)v$ by $M_1'$ and $M_2+(t+d)v$ by $M_2'$. These are the center points of the two $D$-dimensional balls that form the intersection of the original $(D+1)$-dimensional balls with $H+tv$. According to Pythagoras' theorem, their radii are given by $(r^2 - (t-d)^2)^{\nicefrac{1}{2}}$ and $(r^2 - (t+d)^2)^{\nicefrac{1}{2}}$, whereas the distance between the center points is $\|M'_1-M'_2\| = (\|M_1-M_2\|^2 - 4d^2)^{\nicefrac{1}{2}}$.}
         \label{fig_radii}
\end{figure}
We have
\begin{align*}
\frac{d}{dt} r_1(t)&=\frac{-t+d}{\sqrt{1-(t-d)^2}\sqrt{\|M_1-M_2\|^2-(2d)^2}} \\
\frac{d}{dt} r_2(t)&=\frac{-t-d}{\sqrt{1-(t+d)^2}\sqrt{\|M_1-M_2\|^2-(2d)^2}}
\end{align*}
We oberserve the following relations between $r_1$ and $r_2$:
\begin{align}
 r_1(t) & > r_2(t)\label{ineq_r1r2}\\
\diff{t} r_2(t) &< 0 \label{ineq_r2} \\
 \left|\diff{t} r_1(t) \right| &<  \left|\diff{t} r_2(t) \right|  \label{ineq_r1r2_dt}
\end{align}
First, let us discuss the case when there exists an open environment $I$ containing $t$, such that $\mathcal Q(t') = r_1(t')^{-D} r_2(t')^D$ for all $t'\in I$. We conclude from \eqref{ineq_r1r2}, \eqref{ineq_r2} and \eqref{ineq_r1r2_dt}
\begin{align*}
 \diff{t} \mathcal Q(t) &= D \left(\frac{r_2(t)}{r_1(t)}\right)^{D-1} \frac{r_1(t) \left(\diff{t} r_2(t)\right) - \left(\diff{t} r_1(t) \right) r_2(t)}{r_1(t)^2 } \\
 & < 0
\end{align*}
Next, let us consider that case where $\mathcal Q = (r_1+r_2-1) (r_1+r_2+1)^{-1}$ on an open interval containing $t$. Again, we conclude from \eqref{ineq_r2} and \eqref{ineq_r1r2_dt}
\begin{align*}
 \diff{t} \mathcal Q(t) & =2 \frac{\diff{t} r_1(t) + \diff{t} r_2(t)}{(r_1+r_2+1)^2} \nonumber\\
 & < 0
\end{align*}
Finally, consider the case where $D>2$ and on an open environment around $t$ we have
 \begin{align}
  \mathcal Q &=q_D(r_1, r_2) \nonumber\\
  &=  \frac{V_D^{\text{cap}}(r_1, r_2)+ V_D^{\text{cap}}(r_2, r_1) }{V^{\text{ball}}_D(r_1) + V_D^{\text{ball}}(r_2)-V_D^{\text{cap}}(r_1, r_2)- V_D^{\text{cap}}(r_2, r_1) }\label{eq_analytic_formula_gap_coeff} 
 \end{align}
 for $V_D^\text{ball}(\cdot)$ and $V_D^\text{cap}(\cdot, \cdot)$ defined as in Lemma \ref{lem_exact_formula_general_volume_coeff}.
The terms $V^{\text{ball}}_D(\cdot)$ and $V_D^{\text{cap}}(\cdot, \cdot)$ denote the volume of the respective balls and caps up to the constant 
\[c  = \frac{2\Gamma\left(\frac{D}{2}+1\right)}{\pi^{\frac{D}{2}}}\] Recall that the derivative of the volume of a ball w.r.t. its radius is given by the surface area of the corresponding sphere. In particular, we have
\begin{align*}
  \diff{r_i} V_D^\text{ball}(r_j) = \begin{cases}
                                     c A^{\text{sphere}}_D(r_i) & ,i=j\\
                                     0 & ,i\neq j
                                    \end{cases}
\end{align*}
with $A^{\text{sphere}}_D(\cdot)$ denoting the surface area of a $D$-dimensional sphere with given radius. 
Similarly, it can be shown that the partial derivatives of the volume of the overlap $c^{-1} ( V_D^\text{cap}(r_1, r_2)+V_D^\text{cap}(r_2, r_1))$ w.r.t. $r_1$ and $r_2$ are up to the same constant given by the surface areas $A_D^\text{cap}(r_1, r_2)$ and $A_D^\text{cap}(r_2, r_1)$ of the the corresponding hyperspherical caps that form together the boundary of the overlap, c.f. Figure \ref{fig_partial}.
\begin{figure}[H]
\centering
 \includegraphics[scale = 0.22]{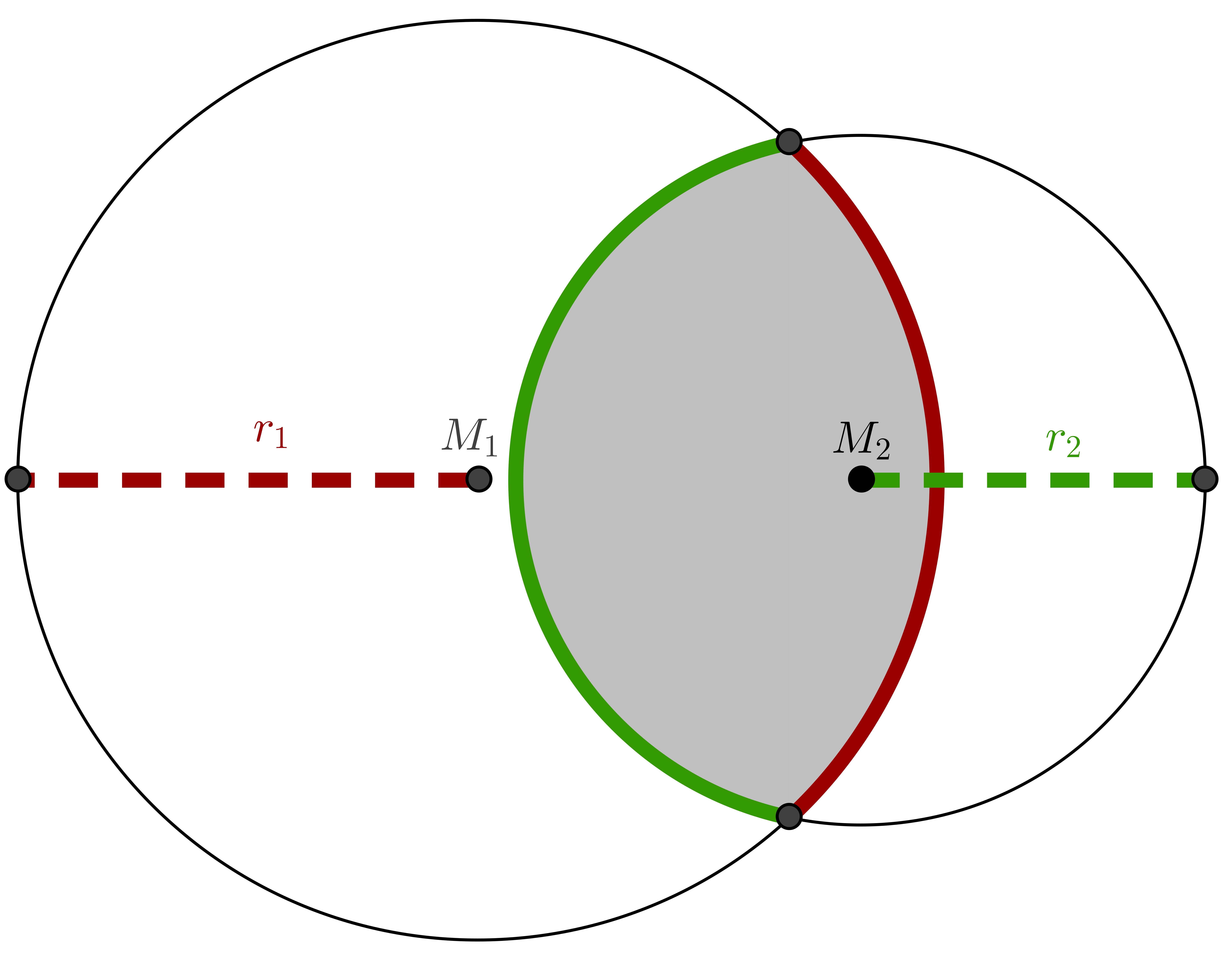}
 \caption{Case $D=2$: The derivative of the area of the intersection of the two balls (gray) w.r.t. $r_1$ ($r_2$) is given by the red (green) arc length
 }
 \label{fig_partial}
\end{figure}

Exact formulas for $A_D^\text{cap}(\cdot, \cdot)$ are given in \cite{formula_cap}. However, those are not needed for this proof. It is enough to observe the following relation 
\begin{align*}
 A_D^\text{cap}(r_1, r_2)< A_D^\text{cap}(r_2, r_1)
\end{align*}
as a consequence of \eqref{ineq_r1r2}. Let us introduce the notations 
\begin{align*}
 S^\text{ball} &\defined V^{\text{ball}}_D(r_1) + V_D^{\text{ball}}(r_2)\\
  S^\text{cap} &\defined V^{\text{cap}}_D(r_1, r_2) + V_D^{\text{cap}}(r_2, r_1)
\end{align*}
We conclude
\begin{align}
\text{and } \diff{r_1} S^\text{ball} & > \diff{r_2}S^\text{ball}\label{ineq_partial_ball}\\
\diff{r_1} S^\text{cap} &< \diff{r_2}S^\text{cap}\label{ineq_partial_cap}
\end{align}
In view of
\begin{align*}
 \diff{r_i}  q_D(r_1, r_2) &= \frac{\left(\diff{r_i} S^\text{cap}\right) S^\text{ball} - S^\text{cap}\left(\diff{r_i} S^\text{ball}\right)}{\left(S^\text{ball}-S^\text{cap}\right)^2}
\end{align*}
we conclude from \eqref{ineq_partial_ball} and \eqref{ineq_partial_cap}
\begin{align}
\diff{r_1} q_D  < \diff{r_2} q_D \label{ineq_partial_1}
\end{align}
Note that increasing the radii $r_1$ and $r_2$ by a common factor $C>1$ has the same effect on the coefficient of the volumes of the intersection and the union of the two corresponding balls as when moving the center point $M_2$ such that $\|M_1-M_2\|$ decreases by a factor $C^{-1}$. Considering Lemma \ref{lem_general_formula_general_volume_coef}, we observe  
\begin{align*}
q_D(Cr_1, Cr_2) > q_D(r_1, r_2)
\end{align*}
for $C>1$. This implies
\begin{align}
 r_1 \diff{r_1} q_D + r_2 \diff{r_2} q_D > 0\label{ineq_partial_2}
\end{align}
From \eqref{ineq_partial_1} and \eqref{ineq_partial_2} we deduce
\begin{align}
0 & <  \diff{r_2} q_D\label{ineq_partial_3}\\
 \text{and }\left|\diff{r_1} q_D\right| & < \diff{r_2} q_D\label{ineq_partial_4}
\end{align}
Note that $q_D(r_1, r_2)$ is differentiable at $(r_1(t), r_2(t))$ as the formula given in \eqref{eq_analytic_formula_gap_coeff} is valid on open environment. From \eqref{ineq_r2}, \eqref{ineq_r1r2_dt}, \eqref{ineq_partial_3} and \eqref{ineq_partial_4} we conclude
\begin{align*}
 \diff{t} \mathcal Q(t) &=\diff{t}r_1(t) \diff{r_1}q_D(r_1(t), r_2(t)) + \diff{t}r_2(t) \diff{r_2}q_D(r_1(t), r_2(t)) \\
 & < 0
\end{align*}
Lastly, $r_1(t)-r_2(t)$ is strictly monotonely increasing on $(0, t'_\text{max})$. In view of Lemma \ref{lem_exact_formula_general_volume_coeff}, this implies that $\mathcal Q(t)$ is differentiable on $(0, t'_\text{max}) \setminus S$ with a negative derivative for some finite subset $S\subset (0, t'_\text{max})$. The function $\mathcal Q$ is contineous on $\mathopen{[} 0, t_\text{max}\mathclose{)}$. Consequently, it is also monotonely decreasing.
\end{proof}
\begin{proof}[Lemma \ref{lem_half_space_coeff}]
The case $D=1$ is trivial. Let us assume $D>1$.
We prove the lemma by contradiction, i.e. we assume that there exists a counterexample such that
\begin{align}
 \frac{\lambda(\mathcal H\cap B(M_1, r) \cap B(M_2, r))}{\lambda(\mathcal H \cap (B(M_1, r) \cup B(M_2, r)))}<\frac{\lambda( B(M_1, r) \cap B(M_2, r))}{\lambda (B(M_1, r) \cup B(M_2, r))}  \label{ineq_half_spaces_contradition}
\end{align}
We can choose $\mathcal H$ such that for any other half-space of the form ${\mathcal H}' = \mathcal H + v'$ for some $v' \in\mathbb R^{D}$ containing $M_1$ and $M_2$ we have 
\begin{align}
  \frac{\lambda\left({ \mathcal H}\cap B(M_1, r)\cap B(M_2, r)\right)}{\lambda({\mathcal H}\cap (B(M_1, r)\cup B(M_2, r)))} \leq \frac{\lambda\left({ \mathcal H'}\cap B(M_1, r)\cap B(M_2, r)\right)}{\lambda({\mathcal H'}\cap (B(M_1, r)\cup B(M_2, r)))} \label{ineq_half_space_minimal_assumption}
\end{align}
There exists a unique half-space $\mathcal H_0$ whose boundary $H_0$ contains $\frac{M_1+M_2}{2}$ and is parallel to the boundary of $\mathcal H$. Note that by symmetricity,
\begin{align}
 \frac{\lambda\left(\mathcal H_0 \cap B(M_1, r)\cap B(M_2, r)\right)}{\lambda\left(\mathcal H_0\cap (B(M_1, r)\cup B(M_2, r))\right)} =  \frac{\lambda\left( B(M_1, r)\cap B(M_2, r)\right)}{\lambda\left(B(M_1, r)\cup B(M_2, r)\right)} \label{ineq_half_space_1}
\end{align} 
There exists a unique vector $v$ of norm 1 that is orthogonal to $H_0$ such that $\frac{M_1+M_2}{2}\in \mathcal H_0+v$. Moreover, for $t_\text{max}\defined \sup \{t : (H_0 + tv) \cap (B(M_1, r) \cup B(M_2, r))\neq \{\}\}$, there exists a unique $t_{\mathcal H} \in(0, t_\text{max})$ such that $\mathcal H = \mathcal H_0+t_{\mathcal H} v$. Let us denote the $(D-1)$-dimensional Lebesgue measure by $\lambda_{D-1}$. According to Fubini's theorem we have
\begin{align}
  &\frac{\lambda\left(\mathcal H\cap B(M_1, r)\cap B(M_2, r)\right)}{\lambda\left(\mathcal H\cap (B(M_1, r)\cup B(M_2, r))\right)} \nonumber \\
  =  & \frac{\lambda\left(\mathcal H_0 \cap B(M_1, r)\cap B(M_2, r)\right) + \int_0^{t_{\mathcal H}} \lambda_{D-1} ((H_0+tv)\cap B(M_1, r)\cap B(M_2, r)    ) dt}{\lambda\left(\mathcal H_0 \cap (B(M_1, r)\cup B(M_2, r))\right) + \int_0^{t_{\mathcal H}} \lambda_{D-1} ((H_0+tv)\cap (B(M_1, r)\cup B(M_2, r)   ) ) dt} \label{ineq_half_space_2}
\end{align}
From \eqref{ineq_half_spaces_contradition}, \eqref{ineq_half_space_1}, \eqref{ineq_half_space_2} and the monotonicity described in Lemma \ref{lem_plane_monotonicity}, we conclude
\begin{align}
 \frac{\lambda_{D-1}((H_0 + t_{\mathcal H}v)\cap B(M_1, r) \cap B(M_2, r))}{\lambda_{D-1}((H_0 + t_{\mathcal H}v) \cap (B(M_1, r) \cup B(M_2, r)))} <  \frac{\lambda\left(\mathcal H\cap B(M_1, r)\cap B(M_2, r)\right)}{\lambda\left( \mathcal H \cap (B(M_1, r)\cup B(M_2, r))\right)} \label{ineq_half_space_3}
\end{align}
Suppose $t'\in(t_{\mathcal H}, t_\text{max})$. Then 
\begin{align}
  &\frac{\lambda\left((\mathcal H_0 +t'v)\cap B(M_1, r)\cap B(M_2, r)\right)}{\lambda\left((\mathcal H_0 +t'v)\cap (B(M_1, r)\cup B(M_2, r))\right)} \nonumber \\
  =  & \frac{\lambda\left( \mathcal H \cap B(M_1, r)\cap B(M_2, r)\right) + \int_{t_{\mathcal H}}^{t'} \lambda_{D-1} ((H_0+tv)\cap B(M_1, r)\cap B(M_2, r)    ) dt}{\lambda\left(\mathcal H \cap (B(M_1, r)\cup B(M_2, r))\right) +  \int_{t_{\mathcal H}}^{t'} \lambda_{D-1} ((H_0+tv)\cap (B(M_1, r)\cup B(M_2, r)   ) ) dt} \label{ineq_half_space_4}
\end{align}
From \eqref{ineq_half_space_3}, \eqref{ineq_half_space_4} and Lemma \ref{lem_plane_monotonicity} we deduce 
\begin{align*}
 \frac{\lambda\left((\mathcal H_0 +t'v)\cap B(M_1, r)\cap B(M_2, r)\right)}{\lambda\left((\mathcal H_0 +t'v)\cap (B(M_1, r)\cup B(M_2, r))\right)} < \frac{\lambda\left(\mathcal H\cap B(M_1, r)\cap B(M_2, r)\right)}{\lambda\left(\mathcal H\cap (B(M_1, r)\cup B(M_2, r))\right)} 
\end{align*}
This is a contradiction to \eqref{ineq_half_space_minimal_assumption}.
\end{proof}

Before proving Proposition \ref{lem_gap_coeff_boundary}, we state the following generalization of Lemma \ref{lem_volume_lipschitz}. We denote the Lebesgue measure on a submanifold of \(\mathbb R^D\) by \(\lambda\).
\begin{lem}\label{lem_Lipschitz_integral}
 For a $C$-Lipschitz function $f_1:\mathcal M_1\to \mathcal M_2$ between two $d$-dimensional submanifolds of $\mathbb R^D$ and a measurable function $f_2$ on $\mathcal M_2$ we have
\begin{align*}
 \int_{\mathcal M_2} f_2 d\lambda  \leq C^d \int_{\mathcal M_1} f_2\circ f_1 d\lambda
\end{align*}
\end{lem}
\begin{proof}
 This follows from Lemma \ref{lem_volume_lipschitz} together with the definition of the Lebesgue integral of a positive function as a supremum of integrals of step functions.
\end{proof}

\begin{proof}[Proof of Proposition \ref{lem_gap_coeff_boundary}]
W.l.o.g. we consider only one cluster \(\mathcal C=\mathcal C_1\) and assume \(f\propto \mathbbm 1(\mathcal C)\). If the set of all possible superlevel sets is finite, the general result follows by summation. In case that this set is infinite, e.g. if $f$ is smooth and not constant, $f$ can be constructed as the limit of discrete functions.

Moving on, consider $M_i'\in\mathcal C$ of distance at most $r_\xi$ to $M_i$. Moreover, let us denote the projection on the tangent plane $\mathcal T$ to $\mathcal M$ at $M_1'$ by $P$. Depending on the context, we denote by \(\lambda\) either the Lebesgue measure on \(\mathcal M\) or a linear space such as the tangent space. We apply Lemma \ref{lem_local_lipschitz_param}. For $r\kappa \leq (120)^{-1} $, the projection $P$ is injective on the Ball $B(M_1', 3r)$ with an inverse that is Lipschitz with constant $L\defined 1+360\kappa^2 r^2$. Note that this ball contains $B(M_1', r+2r_\xi)\cup B(M_2', r+2r_\xi)$. From Lemma \ref{lem_Lipschitz_integral} we conclude
\begin{align*}
q_{\mathbb P}=\frac{\mathbb P(B(M_1, r)\cap B(M_2, r))}{\mathbb P(B(M_1, r)\cup B(M_2, r))}&\geq  \frac{\mathbb P_0(B(M_1', r-2r_\xi)\cap B(M_2', r-2r_\xi))}{\mathbb P_0(B(M_1', r+2r_\xi)\cup B(M_2', r+2r_\xi))}\\
&\geq L^{-d} \frac{\int_{T\cap B(P(M_1'), \frac{r-2r_\xi}{L})\cap B(P(M_2'), \frac{r-2r_\xi}{L})} f \circ P^{-1} d\lambda}{\int_{T \cap (B(P(M_1'), r+2r_\xi)\cup B(P(M_2'), r+2r_\xi))} f \circ P^{-1} d\lambda},
\end{align*}
where \(\mathbb P_0\) denotes the noiseless distribution. Moreover, we can rewrite the integral using the push-forward measure $(P\vert_{B(M_1', 3r)}^{-1})_\ast  (\mathbb P_0)$. For simplicity we just use the notation $P^{-1}_\ast \mathbb P_0$ as well as $Z_i\defined P(M_i')$.
We get the lower bound 
\begin{align*}
q_{\mathbb P} &\geq L^{-d} AB\\
A &= \frac{P^{-1}_\ast \mathbb P_0 \left(T \cap B(Z_1, \frac{r-2r_\xi}{L}) \cap B(Z_2, \frac{r-2r_\xi}{L})\right)}{P^{-1}_\ast \mathbb P_0 \left(T \cap (B(Z_1, \frac{r-2r_\xi}{L}) \cup B(Z_2, \frac{r-2r_\xi}{L}))\right)}\\
B &= \frac{P^{-1}_\ast \mathbb P_0 \left(T \cap (B(Z_1, \frac{r-2r_\xi}{L}) \cup B(Z_2, \frac{r-2r_\xi}{L}))\right)}{P^{-1}_\ast \mathbb P_0 \left(T \cap (B(Z_1, r+2r_\xi) \cup B(Z_2, r+2r_\xi))\right)}
\end{align*}
WLOG we assume that $P(\mathcal C)$ does not fully contain the intersection in term A. Then there exists $p\in P(\partial \mathcal C)\cap B(Z_1, \frac{r-2r_\xi}{L})\cap B(Z_2, \frac{r-2r_\xi}{L})$. Consider a ball of radius $2r$ around $P^{-1}(p)$ and let's denote by $\mathcal T'$ the tangent plane of dimension $d-1$ to $\partial \mathcal C$ at $P^{-1}(p)$. If $\kappa' r\leq 80^{-1}$ the inverse of the restriction (to the ball around $P^{-1}(p)$) of the projection of $\partial \mathcal  C$ to $\mathcal T'$ is $L_{\mathcal C}\defined 1+160(\kappa' r)^2$-Lipschitz. By Pythagoras theorem the distance of $\partial \mathcal C$ to $\mathcal T'$ inside the considered Ball is bounded from above by
\begin{align*}
2r \sqrt{L_{\mathcal C}^2 - 1}& =2r \sqrt{320 (\kappa' r )^2+160^2 (\kappa' r)^4}\\
 &\leq 2 \sqrt{324} \kappa' r^2\\
 &=36 \kappa' r^2
\end{align*}
 As the projection onto $\mathcal T$ is $1$-Lipschitz, also the distance of $P(\partial \mathcal C) \cap B(p, \frac{2r}{L})$ to $P(\mathcal T')$ is bounded by the same term. I. p. there exists half-planes $H_2\subset H_1$ of dimension $d$ in $\mathcal T$  whos boundaries are parallel at a distance $72 \kappa' r^2$ and 
\begin{align*}
 H_2\cap B(p, \frac{2r}{L}) \subset P(\mathcal C)\cap B(p, \frac{2r}{L})\subset H_1\cap B(p, \frac{2r}{L})
 \end{align*}
For the denominator of $A$ we get
\begin{align}
 &P^{-1}_\ast \mathbb P_0 \left(T \cap \left(B(Z_1, \frac{r-2r_\xi}{L}) \cup B(Z_2, \frac{r-2r_\xi}{L})\right)\right)\nonumber\\
 &\leq \frac{1}{\lambda(\mathcal M)}\lambda \left(H_1 \cap \left(B(Z_1, \frac{r-2r_\xi}{L}) \cup B(Z_2, \frac{r-2r_\xi}{L})\right)\right) \label{upper_bound_denominator_A}
\end{align}
whereas for the nominator we get 
\begin{align}
 &P^{-1}_\ast \mathbb P_0 \left(\mathcal T \cap B(Z_1, \frac{r-2r_\xi}{L}) \cap B(Z_2, \frac{r-2r_\xi}{L})\right)\nonumber\\
 &\geq \frac{1}{\lambda(\mathcal M)}\lambda \left(H_2 \cap B(Z_1, \frac{r-2r_\xi}{L}) \cap B(Z_2, \frac{r-2r_\xi}{L})\right)\nonumber \\
& \geq\frac{1}{\lambda(\mathcal M)} \left[  \lambda \left(H_1 \cap B(Z_1, \frac{r-2r_\xi}{L}) \cap B(Z_2, \frac{r-2r_\xi}{L})\right) - \lambda \left((H_1  \setminus H_2) \cap B\left(Z_1, \frac{r-2r_\xi}{L}\right)\right)\right]\nonumber\\
& \geq \frac{1}{\lambda(\mathcal M)}  \left[ \lambda \left(H_1 \cap B(Z_1, \frac{r-2r_\xi}{L}) \cap B(Z_2, \frac{r-2r_\xi}{L})\right) - 72 \kappa' r^2 \lambda_{d-1} \left(B_{d-1} \left(\cdot, \frac{r-2r_\xi}{L}\right)\right)\right] \label{lower_bd_A_term}
\end{align}
In the above, we denote by \(\lambda_{d-1}(B_{d-1}(\cdot, r'))\) the volume of a \((d-1)\)-dimensional ball of radius \(r'\).
We have 
 \begin{align}
 \frac{72 \kappa' r^2 \lambda_{d-1} \left(B_{d-1} \left(\cdot, \frac{r-2r_\xi}{L}\right)\right)}{\lambda \left(H_1 \cap (B(Z_1, \frac{r-2r_\xi}{L}) \cup B(Z_2, \frac{r-2r_\xi}{L})\right)} & \leq 144 \kappa' r^2 \frac{\lambda_{d-1} \left(B_{d-1} \left(\cdot, \frac{r-2r_\xi}{L}\right)\right)}{\lambda_{d} \left(B_{d} \left(\cdot, \frac{r-2r_\xi}{L}\right)\right)}\nonumber\\
 & \leq 144 \pi^{-\frac{1}{2}} \kappa' r L \frac{r}{r-2r_\xi} \frac{\Gamma\left(\frac{d+2}{2}\right)}{\Gamma\left(\frac{d+1}{2}\right)}\label{upper_bound_projection_boundary}
 \end{align}
Due to the upper bound assumption on \(r_\xi\) we have \[\frac{r}{r-2r_\xi}\leq \frac{10}{9}\]
Moreoever, the upper bound assumption on \(r\) with respect to the reach implies 
\[L \leq \frac{41}{40}\]
The last factor can be upper bounded utilizing the logarithmic convexity of the gamma function
\begin{align*}
 \frac{\Gamma\left(\frac{d+2}{2}\right)}{\Gamma\left(\frac{d+1}{2}\right)} & \leq \sqrt{\frac{\Gamma\left(\frac{d+3}{2}\right)}{\Gamma\left(\frac{d+1}{2}\right)}}\\
 & = \sqrt{\frac{d+1}{2}}
\end{align*}
Together, we conclude from \eqref{upper_bound_projection_boundary}
 \begin{align}
 \frac{72 \kappa' r^2 \lambda_{d-1} \left(B_{d-1} \left(\cdot, \frac{r-2r_\xi}{L}\right)\right)}{\lambda \left(H_1 \cap (B(Z_1, \frac{r-2r_\xi}{L}) \cup B(Z_2, \frac{r-2r_\xi}{L})\right)} & \leq 66 \kappa' r \sqrt{d+1} \nonumber\\
 &\definedas \delta\label{upper_bound_projection_boundary_delta_term}
 \end{align}
 Our assumptions ensure \(\delta \leq \frac{1}{2}\), i.p. \((1-\delta) \geq (1+2\delta)^{-1}\). Using Lemma \ref{lem_half_space_coeff}, we conclude from \eqref{upper_bound_denominator_A}, \eqref{lower_bd_A_term} and \eqref{upper_bound_projection_boundary_delta_term}
\begin{align*}
 A &\geq q_d\left(L \frac{\|Z_1-Z_2\|}{r-2r_\xi} \right) (1+2\delta)^{-1} \\
 & \geq q_d\left(L \frac{\|M_1-M_2\|+2r_\xi }{r-2r_\xi} \right) (1+2\delta)^{-1} \\
 & = q_d(s) \left(\frac{q_d(s)}{q_d(L\frac{sr+2r_\xi}{r-2r_\xi})}\right)^{-1} (1+2\delta)^{-1}  
 \end{align*}
Using that the absolute value of the derivative of $q_d$ is bounded by $\frac{2}{\mathcal B\left(\frac{d+1}{2}, \frac{1}{2}\right)}$, we get
\begin{align*}
\frac{q_d(s)}{q_d(L\frac{sr+2r_\xi}{r-2r_\xi})} & \leq 1+ \frac{\frac{2}{\mathcal B\left(\frac{d+1}{2}, \frac{1}{2}\right)} \left( L\frac{sr+2r_\xi}{r-2r_\xi} -s \right)}{q_d(L\frac{sr+2r_\xi}{r-2r_\xi})}\\
& \leq 1 + 2 \frac{ L\frac{sr+2r_\xi}{r-2r_\xi} -s}{q_d(b') \mathcal B\left(\frac{d+1}{2}, \frac{1}{2}\right)}\\
& =1+ \frac{2}{q_d(b') \mathcal B\left(\frac{d+1}{2}, \frac{1}{2}\right)} \frac{(L-1)sr +2L r_\xi +2sr_\xi}{r-2r_\xi}\\
&\leq 1+  \frac{2}{q_d(b') \mathcal B\left(\frac{d+1}{2}, \frac{1}{2}\right)}  \left(6 \frac{r_\xi}{r}  +3(L-1) \right)\\
&\numleq{\ref{ineq_q_times_beta}} 1+\frac{2(d+1)}{\left(1-\frac{b'^2}{4}\right)^\frac{d+1}{2}}\left(6 \frac{r_\xi}{r}  +3(L-1) \right)
\end{align*}
Next, let us consider $B$. 
According to the upper bound \eqref{upper_bound_denominator_A} we have
\begin{align*}
 P^{-1}_\ast \mathbb P_0 \left(\mathcal T \cap \left(B(Z_1, \frac{r-2r_\xi}{L}\right) \cup B\left(Z_2, \frac{r-2r_\xi}{L}\right)\right) & \geq \frac{1}{4 \lambda(\mathcal M)} \lambda \left(B(\cdot, \frac{r-2r_\xi}{L})\right) 
\end{align*}
Consequently,
\begin{align*}
 B &= \frac{P^{-1}_\ast \mathbb P_0 \left(\mathcal T \cap (B(Z_1, \frac{r-2r_\xi}{L}) \cup B(Z_2, \frac{r-2r_\xi}{L}))\right)}{P^{-1}_\ast \mathbb P_0 \left(\mathcal T \cap (B(Z_1, r+2r_\xi) \cup B(Z_2, r+2r_\xi))\right)}\\
 &\geq \left(1+8\frac{\lambda(B(\cdot, r+2r_\xi) \setminus B(\cdot, \frac{r-2r_\xi}{L}))}{\lambda(B(\cdot, \frac{r-2r_\xi}{L}))}\right)^{-1}\\
 &=\left( 1 + 8 \frac{(r+2r_\xi)^d - \left(\frac{r-2r_\xi}{L}\right)^d}{\left(\frac{r-2r_\xi}{L}\right)^d} \right)^{-1}\\
 &=\left( 1+8 \left(\frac{L(r+2r_\xi)}{r-2r_\xi}\right)^d - 8\right)^{-1}
\end{align*}
Putting everything together, we end up
\begin{align*}
 q_{\mathbb P} &\geq L^{-d} AB\\
 &\geq  q_d(s) L^{-d} (1+2\delta)^{-1} \left(1+\frac{2(d+1)}{\left(1-\frac{b'^2}{4}\right)^\frac{d+1}{2}}\left(6 \frac{r_\xi}{r}  +3(L-1) \right)\right)^{-1} \left( 1+8 \left(\frac{L(r+2r_\xi)}{r-2r_\xi}\right)^d - 8 \right)^{-1}
\end{align*}
where
\begin{align*}
 \delta=  66 \kappa' r \sqrt{d+1}
\end{align*}
The last two factors can be lower bounded as follows
\begin{align*}
 \left(1+\frac{2(d+1)}{\left(1-\frac{b'^2}{4}\right)^\frac{d+1}{2}}\left(6 \frac{r_\xi}{r}  +3(L-1) \right)\right)^{-1}  &\geq \left(1+ \frac{12(d+1)\frac{r_\xi}{r}}{\left(1-\left(\frac{b'}{2}\right)^2\right)^\frac{d+1}{2}}\right)^{-1}\left(1+ \frac{2160(d+1)(\kappa r)^2}{\left(1-\left(\frac{b'}{2}\right)^2\right)^\frac{d+1}{2}}\right)^{-1}\\
 \left( 1+8 \left(\frac{L(r+2r_\xi)}{r-2r_\xi}\right)^d - 8 \right)^{-1} &\geq \left(1+ 8(L^d-1)\right)^{-1} \left(1+ 8\left( \left( \frac{r+2r_\xi}{r-2r_\xi}\right)^d-1\right)\right)^{-1}\end{align*}
We reorder the factors of the resulting lower bound by variables and get
\begin{align*}
 q_{\mathbb P} & \geq q_d(s) A_{\mathcal M} A_{\partial \mathcal C} A_\xi
\end{align*}
with
\begin{align*}
A_{\mathcal M} &= L^{-d} \left(1+8(L^d-1)\right)^{-1}\left(1+ \frac{2160(d+1)(\kappa r)^2}{\left(1-\left(\frac{b'}{2}\right)^2\right)^\frac{d+1}{2}}\right)^{-1}\\
A_{\partial \mathcal C} &=\left(1+132\kappa' r \sqrt{d+1} \right)^{-1}\\
A_\xi &= \left(1+ 8\left( \left( \frac{r+2r_\xi}{r-2r_\xi}\right)^d-1\right)\right)^{-1}\left(1+ \frac{12(d+1)\frac{r_\xi}{r}}{\left(1-\left(\frac{b'}{2}\right)^2\right)^\frac{d+1}{2}}\right)^{-1}
\end{align*}
Using the inequality $(1+x)^d \leq 1+2xd$ for $0<x\leq \frac{1}{d}$, we get
\begin{align*}
 L^{-d} \left(1+8(L^d-1)\right)^{-1} & \geq 1+11520 d \kappa^2 r^2
\end{align*}
Using the inequalities $760\kappa^2r^2(d+1)\leq 1$ and $\left(1-\left(\frac{b'}{2}\right)^2\right)^{\frac{d+1}{2}}\leq \frac{3}{4}$ we can simplifify
\begin{align*}
A_{\mathcal M } &\geq \left(1+ \frac{45360(d+1)(\kappa r)^2}{\left(1-\left(\frac{b'}{2}\right)^2\right)^\frac{d+1}{2}}\right)^{-1}
\end{align*}
Next, we discuss the term \(A_\xi\). Since $\frac{r_\xi}{r} \leq \frac{1}{10}$, we have \(\frac{r+2r_\xi}{r-2r_\xi}\leq 1+5\frac{r_\xi}{r}\). In view of $\frac{r_\xi}{r}\leq \frac{1}{5d}$ this implies analogously
\begin{align*}
 \left(1+ 8\left( \left( \frac{r+2r_\xi}{r-2r_\xi}\right)^d-1\right)\right)^{-1} &  \geq \left(1+80d \frac{r_\xi}{r}\right)^{-1}
\end{align*}
Using again $\frac{r_\xi}{r} \leq \frac{1}{5d}$ and $\left(1-\left(\frac{b'}{2}\right)^2\right)^{\frac{d+1}{2}}\leq \frac{3}{4}$, we simplify
\begin{align*}
 A_{\partial \mathcal C} &\geq \left(1+ \frac{264(d+1)\frac{r_\xi}{r}}{\left(1-\left(\frac{b'}{2}\right)^2\right)^\frac{d+1}{2}}\right)^{-1}
\end{align*}
The final result is 
\begin{align*}
 q_{\mathbb P} & \geq q_d(s)  \left(1+\epsilon_{\mathcal M}\right)^{-1}\left(1+\epsilon_{\mathcal \xi}\right)^{-1} \left(1+\epsilon_{\partial \mathcal C}\right)^{-1}
\end{align*}
\end{proof}
\begin{proof}[Proof of Theorem \ref{thm_propagation_boundary}]
Again, we can follow the proof of \cite[Theorem 3.1]{AWC}. It relies only on the inequality \( \theta_{ij}^{(k)} \geq q_{ij}^{(k)} \) for \( \|X_i-X_j\|\leq h_k \). This is ensured by Proposition \ref{lem_gap_coeff_boundary} and the construction of the adjusted volume coefficient.
\end{proof}

\begin{proof}[Proof of Corollary \ref{cor_prop_and_sep}]
This result combines Theorem \ref{thm_separation} and Theorem \ref{thm_propagation_boundary}. Note that for the proof of Theorem \ref{thm_separation}, we also need to consider the modification of the adjusted volume coefficient from 
\[\mathfrak q_{ij}^{(k)}=(1+\varepsilon_{\mathcal M})^{-1} (1+\varepsilon_{\xi})^{-1} q_d\left(\frac{\|X_i-X_j\|}{h_{k-1}}\right)\]
to
\[\mathfrak q_{ij}^{(k)} =(1+\epsilon_{\mathcal M})^{-1} (1+\epsilon_{\xi})^{-1}  \left(1+\epsilon_{\partial \mathcal C}\right)^{-1} q_d\left(\frac{\|X_i-X_j\|}{h_{k-1}}\right).\]
However, our assumption 
\[\varepsilon\geq 7\left(1+\epsilon_{\mathcal M}\right)\left(1+\epsilon_{\mathcal \xi}\right) \left(1+\epsilon_{\partial \mathcal C}\right) - 7\]
ensures that inequality \eqref{ineq_lowerbd_diff_q_theta} is still valid. So the results from both theorems are valid under the considered assumptions. Application of the union bound leads to the final result.
\end{proof}

\bibliography{mybib}

\end{document}